\documentclass[12pt,reqno]{amsart}

\usepackage[textheight=22cm,textwidth=16.6cm,headsep=1cm,centering]{geometry}
\usepackage{color}
\usepackage{esint,amssymb}
\usepackage{graphicx}
\usepackage{tikz}

\usepackage{hyperref}

\newtheorem{theorem}{Theorem}
\newtheorem{proposition}[theorem]{Proposition}
\newtheorem{lemma}[theorem]{Lemma}
\newtheorem{corollary}[theorem]{Corollary}

\theoremstyle{definition}
\newtheorem{remark}[theorem]{Remark}

\newtheorem{definition}[theorem]{Definition}

\newcommand{\cref}[1]{Corollary~\ref{c.#1}}

\numberwithin{equation}{section}
\numberwithin{theorem}{section}

\newcommand{\R}{\mathbb{R}}

\newcommand{\HH}{\mathcal{H}}

\newcommand{\ep}{\varepsilon}

\newcommand{\average}{{\mathchoice {\kern1ex\vcenter{\hrule height.4pt
width 6pt depth0pt} \kern-9.7pt} {\kern1ex\vcenter{\hrule
height.4pt width 4.3pt depth0pt} \kern-7pt} {} {} }}
\newcommand{\ave}{\average\int}

\renewcommand{\bar}{\overline}
\renewcommand{\tilde}{\widetilde}

\newcommand{\NN}{\boldsymbol N}

\newcommand{\J}{\mathcal A}

\begin{document}

\title[Regularity of stable solutions up to dimension 9]
{Stable solutions to semilinear elliptic equations\\ are smooth up to dimension 9}

\begin{abstract}
In this paper we prove the following long-standing conjecture:
{\it stable solutions to semilinear
elliptic equations are bounded (and thus smooth) in dimension $n \leq 9$.}

This result, that was only known to be true for $n\leq4$, is optimal: $\log(1/|x|^2)$ is a $W^{1,2}$ singular stable solution for $n\geq10$.

The proof of this conjecture is a consequence of a new universal estimate: we prove that, in dimension $n \leq 9$, stable solutions are bounded in terms only of their $L^1$ norm, independently of the nonlinearity.
In addition, in every dimension we establish a higher integrability result for the gradient and optimal integrability results for the solution in Morrey spaces.

As one can see by a series of classical examples, all our results are sharp.
Furthermore, as a corollary we obtain that extremal solutions of Gelfand problems are  $W^{1,2}$ in every
dimension and they are smooth in dimension $n \leq 9$. 
This answers to two famous open
problems posed by Brezis and Brezis-V\'azquez.
\end{abstract}

\author[X. Cabr\'e]{Xavier Cabr\'e}
\address{﻿X.C.\textsuperscript{1,2,3} ---
\textsuperscript{1}ICREA, Pg.\ Lluis Companys 23, 08010 Barcelona, Spain \& 
\textsuperscript{2}Universitat Polit\`ecnica de Catalunya, Departament de Matem\`{a}tiques, 
Diagonal 647, 08028 Barcelona, Spain \& 
\textsuperscript{3}BGSMath, Campus de Bellaterra, Edifici C, 08193 Bellaterra, Spain.
}
\email{xavier.cabre@upc.edu}

\author[A. Figalli]{Alessio Figalli}
\address{A.F. --- ETH Z\"urich, Mathematics Dept., R\"amistrasse 101, 8092 Z\"urich, Switzerland.}
\email{alessio.figalli@math.ethz.ch}

\author[X. Ros-Oton]{Xavier Ros-Oton}
\address{X.R.\textsuperscript{4,5,6} ---
\textsuperscript{4}Universit\"at Z\"urich,
Institut f\"ur Mathematik, 
Winterthurerstrasse 190, 8057 Z\"urich, Switzerland \& 
\textsuperscript{5}ICREA, Pg.\ Llu\'is Companys 23, 08010 Barcelona, Spain \& 
\textsuperscript{6}Universitat de Barcelona, Departament de Matem\`atiques i Inform\`atica, Gran Via de les Corts Catalanes 585, 08007 Barcelona, Spain.}
\email{xavier.ros-oton@math.uzh.ch}

\author[J. Serra]{Joaquim Serra}
\address{J.S. --- ETH Z\"urich, Mathematics Dept., R\"amistrasse 101, 8092 Z\"urich, Switzerland.}
\email{joaquim.serra@math.ethz.ch}

\date{\today}

\thanks{X.C. is supported by grants MTM2017-84214-C2-1-P and MdM-2014-0445 (Government of Spain), and is a member of the research group 2017SGR1392 (Government of Catalonia).\\ A.F. and J.S. have received funding from the European Research Council
under the Grant Agreement No. 721675 ``Regularity and Stability in Partial Differential Equations (RSPDE)''.\\ 
X.R. has received funding from the European Research Council under the Grant Agreement No. 801867 ``Regularity and singularities in elliptic PDE (EllipticPDE)''.\\
Part of this work has been done while X.C. was visiting ETH Z\"urich. 
X.C. would like to thank the FIM
(Institute for Mathematical Research) at ETH Z\"urich for the kind hospitality and for the financial support.
}

\maketitle
{\small
\tableofcontents
}

\section{Introduction}

Given $\Omega\subset \R^n$ a bounded domain and $f:\R\to \R$,
we consider $u:\Omega \to \R$ a solution to the
semilinear equation
\begin{equation}
\label{eq:PDE}
-\Delta u=f(u) \qquad \text{in }\Omega\subset \R^n.
\end{equation}
If we define $F(t):=\int_0^tf(s)\,ds$, then \eqref{eq:PDE} corresponds to the Euler-Lagrange equation for the energy functional
$$
\mathcal E[u]:=\int_{\Omega}\Bigl(\frac{|\nabla u|^2}2-F(u)\Bigr)\,dx.
$$
In other words, $u$ is a critical point of $\mathcal E$, namely 
$$
\frac{d}{d\epsilon}\Big|_{\epsilon=0}\mathcal E[u+\epsilon\xi]=0 \qquad \mbox{ for all }\xi\in C^\infty_c(\Omega) 
$$
(the space of $C^\infty$ functions with compact support in $\Omega$).
Consider the second variation of $\mathcal E$, that when $f\in C^1$ is given by
$$
\frac{d^2}{d\epsilon^2}\Big|_{\epsilon=0}\mathcal E[u+\epsilon\xi]
=\int_{\Omega}\Bigl(|\nabla \xi |^2-f'(u)\xi^2\Bigr)\,dx.
$$
Then, one says that $u$ is a {stable} solution of equation \eqref{eq:PDE} in $\Omega$ if the second variation is nonnegative, namely
\begin{equation*}\label{stability-intro}
\int_{\Omega} f'(u) \xi^2\,dx\leq \int_{\Omega} |\nabla \xi|^2\,dx  \qquad \mbox{ for all }
\xi\in C^\infty_c({\Omega}).
\end{equation*}
{Note that stability of $u$ is considered within the class of functions agreeing with $u$ near the boundary of $\Omega$.}

Our interest is in nonnegative nonlinearities $f$ that grow at $+\infty$ faster than linearly.
{In this case it is well-known that, independently of the Dirichlet boundary conditions that one imposes on \eqref{eq:PDE}, the energy $\mathcal E$ admits no absolute minimizer.}\footnote{{To see this, take $v \in C^1_c(\Omega)$ with $v\geq0$ and $v\not\equiv 0$, and given $M>0$ consider $$
\mathcal E[u+Mv]=\frac{1}{2}\int_\Omega \big|\nabla (u+Mv)\big|^2\,dx - \int_\Omega F(u+Mv)\,dx.
$$ Since $f$ grows superlinearly at $+\infty$, it follows that $F(t)\gg t^2$ for $t$ large. This leads to $\mathcal E[u+Mv]\to -\infty$ as $M\to+\infty$, which shows that the infimum of the energy among all functions with the same boundary data as $u$ is $-\infty$.}
}
However, we will see that in many instances there exist nonconstant stable solutions, such as local minimizers.
The regularity of stable solutions to semilinear elliptic equations is a very classical topic in elliptic equations, initiated in the seminal paper of Crandall and Rabinowitz \cite{CR}, which has given rise to a huge literature on the topic; see the monograph \cite{Dup} for an extensive list of results and references.

Note that this question is a PDE analogue of another fundamental problem in mathematics, namely the regularity of stable minimal surfaces. 
As it is well known, stable minimal surfaces in $\R^n$ may not be smooth in dimension $n$ larger than 7 \cite{Simons, BDG}, and it is a fundamental open problem whether they are smooth in dimension $n\leq 7$. Up to now this question has been solved only in dimension $n=3$ by Fischer-Colbrie and Schoen \cite{FishS} and Do Carmo and Peng \cite{dCP}.

Note that, also in our PDE problem, the dimension plays a key role.
Indeed, when 
\begin{equation}
\label{eq:n10}
n\geq 10, \quad u=\log\frac{1}{|x|^2}, \quad \text{and}\quad f(u)=2(n-2)e^u,
\end{equation}
we are in the presence of a singular $W^{1,2}_0(B_1)$ stable solution of \eqref{eq:PDE} in $\Omega=B_1$ ---as easily shown using Hardy's inequality. 
On the other hand,
\begin{itemize}
\item if $f(t)=e^t$ or $f(t)=(1+t)^p$  with $p>1$,
\item or more in general if $f \in C^2$ is positive, increasing, convex, and $\lim_{t\to +\infty} \frac{f(t)f''(t)}{f'(t)^2}$ exists\footnote{The existence of the limit $c:=\lim_{t\to+\infty}\frac{f(t)f''(t)}{f'(t)^2}\geq 0$ is a rather strong assumption. Indeed, as noticed in \cite{CR}, if it exists then necessarily $c \leq 1$ (otherwise $f$ blows-up in finite time). Now, when $c=1$ the result follows by \cite[Theorem 1.26]{CR}, while $c<1$ implies that $f(t)\leq C(1+t)^p$ for some $p$ and then the result follows by \cite[Lemma 1.17]{CR}.},
\end{itemize}
then it is well-known since the 1970's that $W^{1,2}_0(\Omega)$ stable solutions are 
bounded (and therefore smooth, by classical elliptic regularity theory \cite{GT}) when $n \leq 9$, see \cite{CR}.
Notice that among general solutions (not necessarily stable), an $L^\infty$ bound only holds for subcritical and critical nonlinearities.\footnote{{We recall that a nonlinearity $f$ is called subcritical (resp. critical/supercritical) if 
$|f(t)|\leq C(1+|t|)^p$ for some $p<\frac{n+2}{n-2}$ (resp. $p=\frac{n+2}{n-2}$/resp. $p>\frac{n+2}{n-2}$).
While solutions to subcritical and critical equations are known to be bounded,  in the supercritical case one can easily construct radially decreasing unbounded $W^{1,2}$ solutions.
}}

All these results motivated the following long-standing\footnote{As we shall explain in Section \ref{sect:extremal}, this conjecture is strongly related to an open problem stated by Brezis in the context of ``extremal solutions'' in \cite{Br}.}

\vspace{3mm}

\noindent
\textbf{Conjecture:} {\it Let $u \in W^{1,2}_0(\Omega)$ be a stable solution to \eqref{eq:PDE}.
Assume that $f$ is positive, nondecreasing, convex, and superlinear at $+\infty$, and let $n\leq 9$. Then $u$ is bounded.}

\vspace{3mm}

In the last 25 years, several attempts have been made in order to prove this result. In particular,
partial positive answers to the conjecture above have been given (chronologically):
\begin{itemize}
\item by Nedev, when $n \leq 3$ \cite{Ned00};
\item by Cabr\'e and Capella when $\Omega=B_1$ and $n\leq 9$ \cite{CC06};

\item by Cabr\'e when $n=4$ and $\Omega$ is convex \cite{C10} (see \cite{C19} for an alternative proof);

\item by Villegas when $n=4$ \cite{Vil13};
\item by Cabr\'e and Ros-Oton when $n\leq 7$ and $\Omega$ is a convex domain ``of double revolution''~\cite{CR-O};
\item by Cabr\'e, Sanch\'on, and Spruck when $n=5$ and $\limsup_{t\to +\infty}\frac{f'(t)}{f(t)^{1+\ep}}<+\infty$ for every $\ep>0$~\cite{CSS}.
\end{itemize}

The aim of this paper is to give a full proof of the conjecture stated above.
Actually, as we shall see below, the interior boundedness of solutions requires no convexity or monotonicity of $f$. 
This fact was only known in dimension $n\leq4$, by a result of the first author \cite{C10}.\footnote{In fact for $n\leq 4$, or for $n\leq 9$ in the radial case, the interior boundedness results cited above (as well as the global boundedness in convex domains) do not require the nonnegativeness of $f$.}
In addition, even more surprisingly, both in the interior and in the global settings we can prove that $W^{1,2}$ {\it stable solutions are universally bounded for $n\leq 9$}, namely they are bounded in terms only of their $L^1$ norm, with a constant that is independent of the nonlinearity~$f$.

\subsection{Main results}
In order to prove our result on the regularity of stable solutions up to the boundary we will be forced to work with nonlinearities $f$ that are only locally Lipschitz (and not necessarily $C^1$). 
Hence, it is important for us to extend the definition of stability to this class of nonlinearities. 
For this, we need to choose a precise representative for $f'$.

\begin{definition}\label{defi1.1}
Let $f:\R\to \R$ be a locally Lipschitz function, and let $u \in W^{1,2}(\Omega)$ be a weak solution to \eqref{eq:PDE}, in the sense that $f(u)\in L^1_{\rm loc}(\Omega)$ and
\begin{equation}\label{weak-sol-intro}
\int_\Omega \nabla u\cdot \nabla \varphi\,dx=\int_\Omega f(u)\varphi\,dx\quad \mbox{ for all }\, \varphi\in C^\infty_c(\Omega).
\end{equation}
Then, we say that $u$ is a stable solution in $\Omega$ if $f'_-(u)\in L^1_{\rm loc}(\Omega)$ and
\begin{equation}\label{stabilityLip}
\int_\Omega f'_-(u) \xi^2\,dx\leq \int_{\Omega} |\nabla \xi|^2 \,dx\qquad \mbox{ for all } 
\xi\in C^\infty_c(\Omega),
\end{equation}
where $f'_{-}$ is defined as 
\begin{equation}
\label{eq:f'-}
f'_-(t) : = \liminf_{h\to 0}\frac{f(t+h)-f(t)}{h}\qquad {\rm for}\ t \in \R.
\end{equation} 
\end{definition}

As we shall see later, in our proofs we only use \eqref{stabilityLip} with test functions $\xi$ that vanish in the set $\{|\nabla u|=0\}$.
Hence, as a consequence of Lemma \ref{lem:eq Du}(i), in this situation the notion of stability is independent of the particular representative chosen for $f'$.

Our first main result provides a universal interior a priori bound on the $C^\alpha$ norm of solutions when $n\leq 9$. Actually, in every dimension we can prove also a higher integrability result for the gradient (with respect to the natural energy space $W^{1,2}$).
Since the result is local, we state it in the unit ball.
Also, {because stable solutions $u$ can be approximated by smooth ones (at least when $u\in W^{1,2}_0(\Omega)$ and $f$ is convex; see \cite[Section 3.2.2]{Dup}), 
we shall state the result as an a priori bound assuming that $u$ is smooth.}

\begin{theorem}
\label{thm:L1-CalphaW12}
Let $B_1$ denote the unit ball of $\R^n$. 
Assume that $u\in C^2(B_1)$ is a stable 
solution of
$$
-\Delta u=f(u) \quad \text{in }B_1,
$$
with  $f:\R \to\R$ locally Lipschitz and nonnegative.

Then
\begin{equation}
\label{eq:W12g L1 int}
\|\nabla u\|_{L^{2+\gamma}({B}_{1/2})} \le C\|u\|_{L^1(B_1)},
\end{equation}
where $\gamma>0$ and $C$  are dimensional constants.
In addition, if $n \leq 9$ then
\begin{equation}
\label{eq:Ca L1 int}
\|u\| _{C^\alpha(\overline{B}_{1/2})}\leq C\|u\|_{L^1(B_1)},
\end{equation}
where $\alpha>0$ and $C$  are dimensional constants.
\end{theorem}

\begin{remark}
As mentioned before, it is remarkable that the interior estimates hold with bounds that are {\it independent} of the nonlinearity $f$. 
Note that, also in the global regularity result Theorem \ref{thm:globalCalpha}, we can prove a bound independent of $f$.
\end{remark}

Combining the previous interior bound with the moving planes method, we obtain a universal bound on $u$ when $\Omega$ is convex.
\begin{corollary}
\label{thm:convex}
Let $n \leq 9$ and let $\Omega\subset \R^n$ be any bounded convex $C^1$ domain. 
Assume that $f: \R\to \R$ is locally Lipschitz and nonnegative.
Let $u\in C^0(\overline\Omega)\cap C^2(\Omega)$ be a stable solution of
$$
\left\{
\begin{array}{cl}
-\Delta u=f(u) & \text{in }\Omega\\
u=0 & \text{on }\partial\Omega.
\end{array}
\right.
$$

Then there exists a constant $C$, depending only on $\Omega$, such that
\begin{equation}
\label{eq:boundary bdd}
\|u\|_{L^\infty(\Omega)}\leq C  \|u\|_{L^1(\Omega)}.
\end{equation}
\end{corollary}

We now state our second main result, which concerns the global regularity of stable solutions in general $C^3$ domains when the nonlinearity is convex and nondecreasing.
As we shall explain in the next section, this result completely solves two open problems posed by Brezis and Brezis-V\'azquez in \cite{Br,BV}.
Again, we work with classical solutions and prove an a priori estimate.
{In this case it is crucial for us to assume $f$ to be convex and nondecreasing. Indeed, the proof of regularity up to the boundary will rely on a very general closedness result for stable solutions with convex nondecreasing nonlinearities, that we prove in Section~\ref{sect:compact}.}

\begin{theorem}\label{thm:globalCalpha}
Let  $\Omega\subset \R^n$ be a bounded domain of class $C^3$. 
Assume that $f: \R \to \R$ is nonnegative, nondecreasing, and convex. 
Let $u\in C^0(\overline\Omega) \cap C^2(\Omega)$ be a stable solution of 
$$
\left\{
\begin{array}{cl}
-\Delta u=f(u) & \text{in }\Omega\\
u=0 & \text{on }\partial\Omega.
\end{array}
\right.
$$

Then
\begin{equation}
\label{eq:W12g L1 glob}
\|\nabla u\|_{L^{2+\gamma}(\Omega)}\le C\|u\|_{L^1(\Omega)},
\end{equation}
where $\gamma>0$ is a dimensional constant and $C$  depends only on $\Omega$.
In addition, if $n \leq 9$ then
\begin{equation}
\label{eq:C0a L1 glob}
\|u\| _{C^\alpha(\overline\Omega)} \leq C\|u\|_{L^1(\Omega)},
\end{equation}
where $\alpha>0$ is a dimensional constant and $C$ depends only on $\Omega$.
\end{theorem}

As an immediate consequence of such a priori estimates, we will prove the long-standing conjecture stated above.

\begin{corollary}\label{thm:conjecture}
Let $\Omega\subset \R^n$ be any bounded domain of class $C^3$. 
Assume that $f: \R \to \R$ is  nonnegative, nondecreasing, convex, and satisfies
\[\frac{f(t)}{t}\geq \sigma(t)\longrightarrow +\infty\quad \mbox{as}\quad t\to+\infty\]
for some function $\sigma:\R\to \R$. 
Let $u\in W^{1,2}_0(\Omega)$ be any stable weak solution of \eqref{eq:PDE} and assume that $n \leq 9$. 
Then
\[\|u\|_{L^\infty(\Omega)}\leq C,\]
where $C$ is a constant depending only on $\sigma$ and $\Omega$.
\end{corollary}

The key point here is to prove the bounds for classical solutions (Theorem \ref{thm:globalCalpha}).
Once this is done, a well known approximation argument (see \cite[Theorem 3.2.1 and Corollary~3.2.1]{Dup}) shows that the same bounds \eqref{eq:W12g L1 glob}-\eqref{eq:C0a L1 glob} hold for every $W^{1,2}_0(\Omega)$ stable weak solution~$u$.
Finally, to control $\|u\|_{L^1(\Omega)}$ in \eqref{eq:W12g L1 glob}, we use Proposition \ref{prop:L1}.

\subsection{Application: $W^{1,2}_0$ and $L^\infty$ regularity of extremal solutions} 
\label{sect:extremal}
Let $f:[0,+\infty)\to \R$ satisfy $f(0)>0$ and be nondecreasing, convex, and superlinear at $+\infty$ in the sense that 
\[\lim_{t\to+\infty}\frac{f(t)}{t}=+\infty.\]
Given a constant $\lambda > 0$ consider the nonlinear elliptic problem
\begin{equation}
\label{eq:lambda}
\left\{
\begin{array}{cl}
-\Delta u=\lambda f(u) & \text{in }\Omega\\
u>0 & \text{in }\Omega\\
u=0 & \text{on }\partial\Omega,
\end{array}
\right.
\end{equation}
where $\Omega\subset \R^n$ is a smooth bounded domain.
We say that $u$ is a classical solution if $u \in C^0(\overline\Omega)\cap C^2(\Omega)$.

In the literature, this problem is usually referred to as the ``Gelfand problem'', or a ``Gelfand-type problem''.
It was first presented by Barenblatt in a volume edited by Gelfand \cite{Gelf}, and was motivated by problems occurring in combustion\footnote{Originally, Barenblatt introduced problem \eqref{eq:lambda} for the exponential nonlinearity $f(u)=e^u$ (arising as an approximation of a certain empirical law). Nowadays, the terminology of Gelfand or Gelfand-type problem applies to all $f$ satisfying the assumptions above.}. 
Later, it was studied by a series of authors; see for instance  \cite{Br,BV,Dup,C17} for a complete account on this topic.

The basic results concerning \eqref{eq:lambda} can be summarized as follows
(see for instance \cite[Theorem 1 and Remark 1]{Br} or the book \cite{Dup} by Dupaigne):

\begin{theorem}[see \cite{Br,BV,Dup}]\label{thm:lambda star}
There exists a constant $\lambda^\star \in (0,+\infty)$ such that:
\begin{itemize}
\item[(i)] For every $\lambda \in (0,\lambda^\star)$ there is a unique $W^{1,2}_0(\Omega)$ stable solution $u_\lambda$ of \eqref{eq:lambda}. 
Also, $u_\lambda$ is a classical solution and $u_\lambda<u_{\lambda'}$ for $\lambda<\lambda'$.

\item[(ii)] For every $\lambda >\lambda^\star$ there is no classical solution.

\item[(iii)] For $\lambda=\lambda^\star$ there exists a unique $L^1$-weak solution $u^\star$, in the following sense: $u^\star \in L^1(\Omega)$, $f(u^\star){\rm dist}(\cdot,\partial \Omega)\in L^1(\Omega)$, and
$$
-\int_\Omega u^\star \Delta \zeta\,dx=\lambda^\star\int_{\Omega}f(u^\star)\zeta\,dx \qquad \text{for all }\zeta \in C^{2}(\overline\Omega)\text{ with } \zeta_{|\partial\Omega}= 0.
$$
This solution is called the {\em extremal solution} of  \eqref{eq:lambda}
and satisfies $u_\lambda\uparrow u^\star$ as $\lambda\uparrow \lambda^\star$.
\end{itemize}
\end{theorem}

The uniqueness of weak solution for $\lambda=\lambda^\star$ is a delicate result that was proved by Martel \cite{Martel}.

In \cite[Open problem 1]{Br}, Brezis asked the following:
 
\vspace{3mm}

\noindent
\textbf{Open problem 1:} {\it Is there something ``sacred'' about dimension 10? More precisely, is it possible in ``low'' dimensions to construct some $f$ (and some $\Omega$) for which the extremal solution $u^\star$ is unbounded? Alternatively, can one prove in ``low'' dimension that $u^\star$ is
smooth for every $f$ and every $\Omega$? }

\vspace{3mm}

To connect this to the conjecture stated before, note that Brezis' problem can be thought as an a priori bound for the stable solutions $\{u_\lambda\}_{\lambda<\lambda^\star}$.
Hence, understanding the regularity of extremal
 solutions is equivalent to understanding a priori estimates for stable classical solutions.

Note that, a priori, extremal solutions are merely in $L^1(\Omega)$.
It is then natural to ask whether extremal solutions do belong to the natural energy space $W^{1,2}_0(\Omega)$. This important question was posed by Brezis and V\'azquez in \cite[Open problem 1]{BV}:

\vspace{3mm}

\noindent
\textbf{Open problem 2:} {\it Does there exist some $f$ and $\Omega$ for which the extremal solution is a weak\footnote{In the sense of Theorem \ref{thm:lambda star}(iii).} solution not in $W^{1,2}_0(\Omega)$?}

\vspace{3mm}

Concerning this problem, it has been proved that $u^\star$ belong to the energy space $W^{1,2}_0(\Omega)$ when $n\leq 5$ by Nedev \cite{Ned00}, for every $n$ when $\Omega$ is convex also by Nedev \cite{NedPrep}, and finally when $n=6$ by Villegas \cite{Vil13}.
Here we prove that $u^\star\in W^{1,2}_0(\Omega)$ for every $n$ and for every smooth domain $\Omega$, thus giving a conclusive answer also to this second open problem.

Note that, thanks to the superlinearity of $f$, it follows by Proposition \ref{prop:L1} that the $L^1(\Omega)$ norms of the functions $\{u_\lambda\}_{\lambda<\lambda^\star}$ are uniformly bounded by a constant depending only on $f$ and~$\Omega$. Hence, by applying Theorem \ref{thm:globalCalpha} to the functions $\{u_\lambda\}_{\lambda<\lambda^\star}$ and letting  $\lambda \uparrow \lambda^\star$, we immediately deduce that extremal solutions are always $W^{1,2}$ (actually even $W^{1,2+\gamma}$) in every dimension, and that they are universally bounded (and hence smooth) in dimension $n\leq 9$. 
We summarize this in the following:

\begin{corollary}\label{thm:energy-solns}
Let  $\Omega\subset \R^n$ be a bounded domain of class $C^3$. 
Assume that $f: [0,+\infty) \to (0,+\infty)$ is nondecreasing, convex, and superlinear at $+\infty$, and let $u^\star$ denote the extremal solution of \eqref{eq:lambda}.

Then $u^\star \in W^{1,2+\gamma}_0(\Omega)$ for some dimensional exponent $\gamma>0.$
In addition, if $n \leq 9$ then $u^\star$ is bounded and it is therefore a classical solution.
\end{corollary}

\subsection{The case $\mathbf{n\geq 10}$}  \label{section1.3}
In view of the results described in the previous sections, it is natural to ask what can one say about stable solutions in dimension $n \geq 10.$ 
Our strategy of proof can be used to provide optimal (or perhaps almost optimal) integrability estimates in Morrey spaces in every dimension, as stated next (see Section \ref{sect:n10} for more details and for Morrey estimates for the gradient of stable solutions).

Recall that Morrey norms are defined as
\[\|w\|_{M^{p,\beta}(\Omega)}^p:=\sup_{y\in\overline\Omega,\ r>0}r^{\beta-n}\int_{\Omega\cap B_r(y)}|w|^p\,dx,\]
for $p\geq1$ and $\beta\in(0,n)$.

\begin{theorem}\label{thm11}
Let $u\in C^2(B_1)$ be a stable 
solution of
$$
-\Delta u=f(u) \quad \text{in }B_1\subset \R^n,
$$
with  $f:\R \to\R$ locally Lipschitz.
Assume that $n\geq 10$ and  define
\begin{equation}
\label{eq:pn}
p_n:=\left\{
\begin{array}{ll}
\infty &\text{if }n=10,\\
\frac{2(n-2\sqrt{n-1}-2)}{n-2\sqrt{n-1}-4} &\text{if }n\geq 11.
\end{array}
\right.
\end{equation}

Then
\begin{equation}
\label{eq:p n11}
\Vert u\Vert_{M^{p,2+\frac{4}{p-2}}(B_{1/2})}\leq C \Vert u\Vert_{L^1(B_1)}\quad\ \text{ for every }\ \
p<p_n,
\end{equation}
where  $C$ depends only on $n$ and $p$.

In addition, if {$f$ is nonnegative and nondecreasing,}  $\Omega\subset \R^n$ is a bounded domain of class~$C^3$, and $u\in C^0(\overline\Omega) \cap C^2(\Omega)$ is a stable solution of 
$$
\left\{
\begin{array}{cl}
-\Delta u=f(u) & \text{in }\Omega\\
u=0 & \text{on }\partial\Omega,
\end{array}
\right.
$$
then 
\begin{equation}
\label{eq:p n11 global}
\Vert u\Vert_{M^{p,2+\frac{4}{p-2}}(\Omega)}\leq C \Vert u\Vert_{L^1(\Omega)}\quad\ \text{ for every }\ \
p<p_n,
\end{equation}
for some constant $C$ depending only on $p$ and $\Omega$.
\end{theorem}

It is interesting to observe that the above result is essentially optimal.
To see this we recall that, in dimension $n=10$, the function $u=\log(1/|x|^2)$ is an unbounded $W^{1,2}_0(B_1)$ stable solution in $B_1$
(see \eqref{eq:n10}, and recall that it can be approximated by stable classical solutions by \cite[Section 3.2.2]{Dup}).
Also, as shown in  \cite{BV}, for $n\geq 11$ the function $u(x)=|x|^{-2/(q_n-1)}-1$ is the extremal solution of
\begin{equation}
\left\{ \begin{array}{cl}
 -\Delta u  = \lambda^\star (1+u)^{q_n} & \textrm{in }B_{1}\\
u > 0 & \textrm{in }B_{1}\\
 u  =  0 & \textrm{on }\partial B_{1},
\end{array}
\right.
\label{eq:05}
\end{equation}
with $\lambda^\star=\frac{2}{q_n-1}\big(n-2-\frac{2}{q_n-1}\big)$ and $q_n:=
\frac{n-2\sqrt{n-1}}{n-2\sqrt{n-1}-4}$. 
In particular, it is easy to see that $u \in M^{p,2+\frac{4}{p-2}}(B_{1/2})$ if and only if $p\leq p_n$.
It is an open question whether \eqref{eq:p n11} holds with $p=p_n$ for a general stable solution $u$.

\subsection{Idea of the proofs}

The starting point is the stability inequality for $u$, i.e.,
\begin{equation}\label{idea-proof-stability}
\int_{B_1} f'(u) \xi^2\,dx\leq \int_{B_1} |\nabla \xi|^2\,dx  \qquad \mbox{ for all }
\xi\in C^\infty_c({B_1}).
\end{equation}

In order to get a strong information on $u$, one has to choose an appropriate test function $\xi$ in \eqref{idea-proof-stability}.
Most of the papers on this topic (including those of Crandall-Rabinowitz \cite{CR} and Nedev \cite{Ned00}) have considered $\xi=h(u)$ for some appropriate function $h$ depending on the nonlinearity $f$.
The main idea in the $L^\infty$ estimate of \cite{C10} for $n\leq4$ was to take, instead, $\xi=|\nabla u|\varphi(u)$, and choose then a certain $\varphi$ depending on the solution $u$ itself.

Here, a first key idea in our proofs is to take a test function of the form
\[\xi=(x\cdot \nabla u)|x|^{(2-n)/2}\zeta,\]
with $0\leq \zeta\leq 1$ a smooth cut-off function equal to $1$ in $B_\rho$ and vanishing outside $B_{3\rho/2}$.
Thanks to this, we can prove the following inequality (see Lemma \ref{conseqestab2}):
there exists a dimensional constant $C$ such that
\begin{equation}
\label{eq:2n10}
(n-2)(10-n)\int_{B_\rho}|x|^{-n}|x\cdot \nabla u|^2\,dx \leq C\rho^{2-n}\int_{B_{3\rho/2}\setminus B_\rho}|\nabla u|^2\,dx \qquad \mbox{for all }\,0<\rho<{\textstyle\frac23}.
\end{equation}
From this inequality we see immediately that for $3 \leq n\leq 9$ we get a highly nontrivial information. 
While of course one can always assume that $n\geq 3$ (if $n\leq 2$ it suffices to add some superfluous variables to reduce to the case $n=3$), here we see that the assumption $n\leq 9$ is crucial.

Thus, when $n \leq 9$, the above inequality tells us that the radial derivative of $u$ in a ball is controlled by the total gradient in an annulus. 
Still, it is important to notice that \eqref{eq:2n10} does \emph{not} lead to an $L^\infty$ bound for general solutions $u$ to $-\Delta u=f(u)$ in dimensions $n\leq9$.\footnote{This can be seen by taking functions $u$ in $\R^3$ depending only on two variables; see Remark \ref{rem22}.} 
Thus, we still need to use stability again in a crucial way.

If we could prove that for stable solutions the radial derivative $x\cdot \nabla u$ and the total derivative $\nabla u$ have comparable size in $L^2$ at every scale, then {we could control the right hand side of \eqref{eq:2n10} with $\int_{B_{3\rho/2}\setminus B_\rho}|x|^{-n}|x\cdot \nabla u|^2\,dx$. This would imply}  that
$$
\int_{B_\rho}|x|^{-n}|x\cdot \nabla u|^2\,dx \leq C\int_{B_{3\rho/2}\setminus B_\rho}|x|^{-n}|x\cdot \nabla u|^2\,dx,
$$
and by a suitable iteration and covering argument we could conclude that $u \in C^{\alpha}$.
This is indeed the core of our interior argument: we show that  the radial derivative and the total derivative have comparable size in $L^2$ (at least whenever the integral of $|\nabla u|^2$ on balls enjoys a doubling property; see Lemma \ref{lem:22}). 
This is based on a delicate compactness argument, which relies on a series of a priori estimates:
\begin{enumerate}
\item curvature-type estimates for the level sets of $u$, which follow by taking $\xi=|\nabla u|\eta$
 as a test function in the stability inequality; see Lemma \ref{conseqestab};
\item the higher $L^{2+\gamma}$ integrability of the gradient, which follows from (1) and a suitable Dirichlet energy estimate, \eqref{ahgiohwiob1}, on each level set of $u$; see Proposition \ref{prop:W12+ep};
\item a general compactness argument for superharmonic functions; see Lemma \ref{strongconvergence};
\item the non-existence of nontrivial 0-homogeneous superharmonic functions; see the proof of Lemma \ref{lem:22}.
\end{enumerate}
Combining all these ingredients, we prove Theorem \ref{thm:L1-CalphaW12}.

\smallskip

For the boundary estimate we would like to repeat the interior argument described above near a boundary point.
We note that, whenever the boundary is completely flat and contains the origin, since $x\cdot \nabla u$ vanishes on the flat boundary then one can still use the test function $\xi=(x\cdot \nabla u)|x|^{(2-n)/2}\zeta$ to deduce the analogue of \eqref{eq:2n10}. Actually, a suitable variant of this test function allows us to obtain a similar estimate even when the boundary is $C^3$-close to a hyperplane (see Lemma \ref{lem:x Du def}). 
In addition, when the boundary is $C^3$-close to a hyperplane, we are able to prove the higher $L^{2+\gamma}$ integrability of the gradient near the boundary (see Proposition \ref{prop:W12+epbdry}), and from there we can conclude that the $W^{1,2}$ norm near the boundary can be controlled only in terms of the $L^1$ norm (see Proposition \ref{prop:L1controlsW12bdry}).

Unfortunately, even if the boundary is completely flat, one cannot {repeat the argument used in the interior case to} deduce that the radial derivative controls the total gradient near a boundary point ---which was a crucial point in the interior case. 
Indeed, while in the interior case the proof relied on the non-existence of nontrivial 0-homogeneous superharmonic functions in a neighborhood of the origin (see the proof of Lemma \ref{lem:22}), in the boundary case such superharmonic functions may exist! 
Hence, in this case we need to exploit in a stronger way the fact that $u$ solves a semilinear equation (and not simply that $u$ is superharmonic since $f\geq 0$). 
However, since our arguments are based on a compactness technique, we need bounds that are independent of the nonlinearity $f$.

A new key ingredient here is presented in Section \ref{sect:compact}: we are able to prove that, whenever the nonlinearity is convex and nondecreasing ---but possibly taking the value $+\infty$ in an interval $[M,\infty)$--- the class of stable solutions is closed under $L^1_{\rm loc}$ convergence (see Theorem \ref{thm:stability}). Note that this is particularly striking, since no compactness assumptions are made on the nonlinearities!

With this powerful compactness theorem at hand, we are able to reduce ourself to a flat-boundary configuration, control the gradient by its radial component, and prove Theorem~\ref{thm:globalCalpha}.

\smallskip

Finally, the case $n\geq 11$ is obtained by choosing the test function $\xi=(x\cdot \nabla u)|x|^{-a/2}\zeta$, where $a=a_n\in(0,n-2)$ are suitable exponents, while in the case $n=10$ we choose $\xi=(x\cdot \nabla u)|x|^{-4}\bigl|\log|x|\bigr|^{-\delta/2}\zeta$, with $\delta>0$.

\vspace{4mm}

The techniques and {ideas introduced in} this paper {are robust enough to be used for proving analogues of our results in} other nonlinear problems.
This is done in a series of forthcoming works by Miraglio, Sanch\'on, and the first author \cite{CMS19} for the $p$-Laplacian, and by Sanz-Perela and the first author \cite{CS19} for the fractional Laplacian.

\subsection{Structure of the paper}

In Section \ref{sect:interior} we exploit the stability of $u$ and choose a series of different test functions to deduce inequality \eqref{eq:2n10} as well as a universal
$W^{1, 2+\gamma}$ bound in terms only of the $L^1$ norm of the solution. This is used in Section \ref{sect:int C0a} to prove our interior estimate of Theorem \ref{thm:L1-CalphaW12}.

In Section \ref{sect:compact} we prove that the class of stable solutions with convex nondecreasing nonlinearities is closed in $L^1_{\rm loc}$,
while in Section \ref{sect:bdry W12g} we obtain a 
$W^{1, 2+\gamma}$ bound near the boundary in terms of the $L^1$ norm when $\partial\Omega$ is a {small} $C^3$-deformation of a hyperplane. 
These results are used in Section \ref{sect:global} to prove Theorem \ref{thm:globalCalpha} via a blow-up and covering argument.

Finally, in Section \ref{sect:n10} we deal with the case $n\geq 10$ and prove Theorem \ref{thm11}.

In the appendices we collect a series of technical lemmata and we show a classical a priori estimate on the $L^1$ norm of solutions to  Gelfand problems.

\section{Interior $W^{1,2+\gamma}$ estimate}
\label{sect:interior}

In this section we begin by proving a series of interior estimates that follow by choosing suitable test functions in the stability inequality. Then we show a universal
$W^{1, 2+\gamma}$ bound in terms only of the $L^1$ norm of the solution.
This is done by first controlling $\|\nabla u\|_{L^{2+\gamma}}$ by  $\|\nabla u\|_{L^{2}}$, and then $\|\nabla u\|_{L^{2}}$ by $\|u\|_{L^{1}}$.

Here and in the sequel, we shall use subscripts to denote partial derivatives (i.e., $u_i=\partial_iu,$ $u_{ij}=\partial_{ij}u$, etc.).

As mentioned in the introduction, our first key estimate for stable solutions comes from considering the test function $\xi=(x\cdot \nabla u) \eta$, and then take $\eta=|x|^{(2-n)/2}\zeta$ for some cut-off function $\zeta$.
We split the computations in two steps since this will be useful in the sequel.

We denote by $C^{0,1}_c(B_1)$ the space of Lipschitz functions with compact support in $B_1$. 

\begin{lemma}\label{conseqestab2}
Let $u\in C^2(B_1)$ be a stable solution of $-\Delta u=f(u)$ in $B_1\subset \R^n$,  with $f$ locally Lipschitz.  
Then, for all $\eta \in C^{0,1}_c(B_1)$ we have
\begin{equation}\label{eq:firstest}
 \int_{B_1} \Big( \big\{(n-2)\eta +2 x\cdot \nabla \eta \big\}\eta\, |\nabla u|^2 
 - 2(x\cdot \nabla u) \nabla u\cdot \nabla(\eta^2) 
 - |x\cdot \nabla u|^2 |\nabla \eta|^2 \Big)\,dx \le 0.
\end{equation}
As a consequence, for all $\zeta \in C^{0,1}_c(B_1)$ we have
\begin{equation}\label{firstest-n-2}
\begin{split}
\frac{(n-2)(10-n)}{4}\int_{B_1} & |x|^{-n}  |x\cdot\nabla u|^2  \zeta^2\, dx \\ 
&\hspace{-3cm} \leq \int_{B_1}  (-2) |x|^{2-n} |\nabla u|^2\zeta (x\cdot\nabla\zeta)\, dx   
+\int_{B_1} 4 |x|^{2-n} (x\cdot\nabla u) \zeta\,  \nabla u\cdot\nabla\zeta \, dx \\
& \hspace{-2cm}
+\int_{B_1} (2-n) |x|^{-n} |x\cdot\nabla u|^2\zeta (x\cdot\nabla\zeta)  \, dx
+\int_{B_1} |x|^{2-n} |x\cdot\nabla u|^2 |\nabla\zeta|^2 \, dx.
\end{split}
\end{equation}

In particular, if $3 \leq n\leq 9$, then for all $\rho< 2/3$ it holds 
\begin{equation}\label{eq:11}
\int_{B_\rho}|x|^{-n}|x\cdot \nabla u|^2\,dx \leq C\rho^{2-n}\int_{B_{3\rho/2}\setminus B_\rho}|\nabla u|^2\,dx,
\end{equation}
where $C$ is a dimensional constant.
\end{lemma}

\begin{proof}
We split the proof in three steps.
\\

\noindent
{\bf Step 1:} {\it Proof of \eqref{eq:firstest}}.
{We note that, by approximation, \eqref{stabilityLip} holds for all $\xi \in C^{0,1}_c(B_1)$. Hence, we can consider as test function in \eqref{stabilityLip}
a function of the form} $\xi=\textbf{c}\eta$,
where $\textbf{c}\in W^{2,p}_{\rm loc}(B_1)$ for some $p>n$, and $\eta\in C^{0,1}_c(B_1)$. 
Then, a simple integration by parts gives that 
\begin{equation}\label{eq:07}
\int_{B_1} \bigl( \Delta \textbf{c}+f_-'(u)\textbf{c}\bigr) \textbf{c}\,\eta^{2}\, dx \leq 
\int_{B_1} \textbf{c}^{2}\left|\nabla \eta\right|^{2} dx.
\end{equation}
We now choose $\textbf{c}(x):=x\cdot \nabla u(x)$ (this function belongs to $W^{2,p}_{\rm loc}(B_1)$ for every $p<\infty$ by Lemma \ref{lem:eq Du}(ii)).
Then, by a direct computation and using Lemma \ref{lem:eq Du}(ii) again,
we deduce that 
$$
\Delta \textbf{c}=x\cdot \nabla \Delta u +2\sum_{i=1}^n u_{ii}=-f_-'(u)\textbf{c}+2\Delta u
$$
a.e. in $B_1$.
Hence, substituting this identity in \eqref{eq:07} we get
\begin{align*}
\int_{B_1} |x\cdot \nabla u|^{2}&\left|\nabla \eta\right|^{2} \,dx \geq 
\int_{B_1}\bigl( \Delta \textbf{c}+f_-'(u)\textbf{c}\bigr) \textbf{c}\,\eta^{2}\, dx=2\int_{B_1} (x\cdot \nabla u)\Delta u\,\eta^{2}\, dx\\
&=\int_{B_1}\Big({\rm div}\bigl(2(x\cdot \nabla u)\nabla u - |\nabla u|^2x\bigr)+(n-2)|\nabla u|^2\Bigr)\eta^2\,dx\\
&=\int_{B_1}\Big(-2(x\cdot \nabla u) \nabla u\cdot \nabla (\eta^2) +|
\nabla u|^2x\cdot \nabla (\eta^2)+(n-2)|\nabla u|^2\eta^2\Bigr)\,dx,
\end{align*}
and \eqref{eq:firstest} follows.\\

\noindent
{\bf Step 2:} {\it Proof of \eqref{firstest-n-2}.}
Given $a<n$,
we would like to take the function $\eta:=|x|^{-a/2}\zeta$ with $\zeta\in C^{0,1}_c(B_1)$ as a test function in \eqref{eq:firstest}.
Since, $\eta$ is not Lipschitz for $a>0$, we approximate it by the $C^{0,1}_c(B_1)$ function
\[\eta_\varepsilon:= \min\{|x|^{-a/2},\varepsilon^{-a/2}\}\zeta\]
for $\varepsilon\in(0,1)$, which agrees with $\eta$ in $B_1\setminus B_\varepsilon$.
We have that $\eta_\varepsilon\to\eta$ and $\nabla \eta_\varepsilon\to \nabla \eta$ a.e. in $B_1$ as $\varepsilon\downarrow0$.
At the same time, every term in \eqref{eq:firstest} with $\eta$ replaced by $\eta_\varepsilon$ is bounded in absolute value by $C|x|^{-a}|\nabla u|^2\leq \tilde C|x|^{-a}\in L^1_{\rm loc}(B_1)$ (since $u\in C^2(B_1)$).
Hence, the dominated convergence theorem gives that \eqref{eq:firstest} also holds with $\eta:=|x|^{-a/2}\zeta$.

Now, noticing that
\begin{equation}
\label{eq:x-a 1}
x\cdot \nabla \eta=-\frac{a}2|x|^{-a/2}\zeta+|x|^{-a/2}x\cdot \nabla \zeta,\quad \nabla (\eta^2)=-a|x|^{-a-2}\zeta^2x+2|x|^{-a}\zeta \nabla \zeta
\end{equation}
and
\begin{equation}
\label{eq:x-a 2}
|\nabla \eta|^2=\Big|-\frac{a}2|x|^{-a/2-2}\zeta x+|x|^{-a/2}\nabla \zeta \Big|^2=\frac{a^2}4|x|^{-a-2}\zeta^2 + |x|^{-a}|\nabla \zeta|^2-a|x|^{-a-2}\zeta (x\cdot \nabla \zeta),
\end{equation}
\eqref{firstest-n-2} follows from \eqref{eq:firstest} by choosing $a=n-2$.
\\

\noindent
{\bf Step 3:} {\it Proof of \eqref{eq:11}.}
Given $\rho \in (0,2/3)$, we consider a Lipschitz function $\zeta$, with $0\leq \zeta\leq 1$, such that $\zeta_{|B_{\rho}}=1$, $\zeta_{|\R^n\setminus B_{3\rho/2}}=0$,
and  $|\nabla \zeta| \leq C/\rho$. Using this function in \eqref{firstest-n-2}
and noticing that $|x|$ is comparable to $\rho$ inside ${\rm supp}(\nabla \zeta)\subset \overline B_{3\rho/2}\setminus B_\rho$, the result follows easily.
\end{proof}

\begin{remark}\label{rem22}
To deduce our $L^\infty$ estimate from \eqref{eq:11}, we will need to use again the stability of $u$. 
In fact, there exist $W^{1,2}$ weak solutions of semilinear equations (with $f>0$) which satisfy \eqref{eq:11} (in balls $B_\rho=B_\rho(y)$ centered at any point $y\in B_1(0)$) and are unbounded.

For instance, with $n=3$ take $u(x_1,x_2,x_3)=\tilde u(x_1,x_2)$, where $\tilde u$ is unbounded but belongs to $W^{1,2}_{\rm loc}(\R^2)$. 
One can then verify that \eqref{eq:11} holds inside every ball $B_\rho=B_\rho(y)$.
At the same time, by taking $\tilde u$ to be radially decreasing in $\R^2$, we can guarantee that $\tilde u$ solves a semilinear equation (and hence also $u$) for some nonlinearity $f$. 
An example is $\tilde u(\rho)=\log|\log \rho|$ in a small neighborhood of the origin, which leads to a smooth nonlinearity $f>0$.
\end{remark}

The key point to deduce boundedness from \eqref{eq:11} will be a higher $L^{2+\gamma}$ integrability result for the gradient of the solution, that we establish in the remaining of this section.

Towards this, we exploit again the stability of $u$ by choosing now, as another test function, $\xi=|\nabla u|\eta$ with $\eta$ a cut-off.
In the case when $u\in C^3$ this choice of test function and the following lemma are due to Sternberg and Zumbrun \cite{SZ}.
We verify next that the result holds also when $f$ is locally Lipschitz.

\begin{lemma}\label{conseqestab}
Let $u\in C^2(B_1)$ be a stable solution of $-\Delta u=f(u)$ in $B_1\subset \R^n$,  with $f$ locally Lipschitz.    Then, for all $\eta \in C^{0,1}_c(B_1)$ we have
\[
\int_{B_1}  \J^2 \eta ^2dx \le \int_{B_1} |\nabla u|^2 |\nabla \eta|^2dx,
\]
where\footnote{Even though we will not use it here, it is worth noticing that the quantity $\mathcal A$ controls the second fundamental form of the level sets of $u$. This was crucially used in \cite{C10}, in combination with the Sobolev-type inequality of  Michael-Simons and Allard, to prove regularity of stable solutions up to dimension $n \leq 4$.}
\begin{equation}\label{defAAA}
\J :=  
\begin{cases} \left( \sum_{ij} u_{ij}^2  - \sum_{i} \left(\sum_{j} u_{ij} \frac{u_j}{|\nabla u|} \right)^2  \right)^{1/2} \quad  \quad &\mbox{if  } \nabla u\neq 0
\\
0 &\mbox{if  } \nabla u=0.
\end{cases}
\end{equation}
\end{lemma}

When $u\in C^3$ (and $f\in C^1$), this follows from the stability inequality \eqref{idea-proof-stability} plus the fact that
\[|\nabla u|\bigl(\Delta|\nabla u|+f'(u)|\nabla u|\bigr)=\mathcal A^2\quad {\rm in}\quad \{\nabla u\neq0\};\]
see \cite{C10} for a proof.
We give here an aternative proof that does not require to compute~$\Delta |\nabla u|$.

\begin{proof}[Proof of Lemma \ref{conseqestab}]
We begin from the identity
$$
-\Delta u_i=f'_-(u)u_i\qquad {\rm for}\quad i=1,\dots,n;
$$
see Lemma \ref{lem:eq Du}(ii).
Multiplying this identity by $u_i\eta^2$ and integrating by parts, we obtain
$$
\int_{B_1}\Big(|\nabla (u_i\eta)|^2-(u_i)^2|\nabla \eta|^2\Bigr)\,dx = \int_{B_1}\nabla u_i\cdot \nabla (u_i\eta^2)\,dx=\int_{B_1}f'_-(u)u_i^2\eta^2\,dx,
$$
so that summing over $i$
 we get
\begin{equation}\label{aioghioe1}
\int_{B_1} \Big(\sum_{i} \big|\nabla (u_{i} \eta)\big|^2 -  |\nabla u|^2 |\nabla \eta|^2\Big)\,dx=    \int_{B_1}f_-'(u) |\nabla u|^2\eta^2\,dx. 
\end{equation}
On the other hand, testing the stability inequality \eqref{stabilityLip} with the Lipschitz function $|\nabla u|\eta$, we obtain 
\begin{equation}\label{aioghioe2}
 \int_{B_1} |\nabla (|\nabla u|\eta) |^2\,dx \ge \int_{B_1} f_-'(u)|\nabla u|^2\eta^2
\,dx. \end{equation}
Hence, combining \eqref{aioghioe1} with \eqref{aioghioe2} gives
\[
\int_{B_1}|\nabla u|^2 |\nabla \eta|^2\,dx\ge \int_{B_1} \Big( \sum_{i} \big|\nabla (u_{i} \eta)\big|^2 - |\nabla( |\nabla u|\eta) |^2 \Big)\,dx.
\]
Then, a direct computation shows that, inside the set  $\{\nabla u\neq 0\}$,
\[
  \sum_{i} \big|\nabla (u_{i} \eta)\big|^2 - |\nabla( |\nabla u|\eta) |^2  =  \biggl(\sum_{i,j} u_{ij}^2  - \sum_{i}\Big(\sum_{j} \frac{u_{ij} u_j}{|\nabla u|}\Big)^2\biggr) \eta^2 = \J^2\eta^2.
\]
On the other hand, since $\nabla u$ is Lipschitz, then $D^2u=0$ a.e. in $\{\nabla u=0\}$ (see, e.g., \cite[Theorem 1.56]{Tro}).
Therefore $\sum_{i} \big|\nabla (u_{i} \eta)\big|^2 - |\nabla( |\nabla u|\eta) |^2=0$ a.e. inside $\{\nabla u=0\}$, concluding the proof.
\end{proof}

Next we prove a general result that gives, in every dimension, a higher integrability result for the gradient of stable solutions.

\begin{proposition}\label{prop:W12+ep}
Let $u\in C^2(B_1)$ be a stable solution of $-\Delta u=f(u)$ in $B_1\subset \R^n$,   with $f$ locally Lipschitz and nonnegative. 
Then
\[
 \|\nabla u\|_{L^{2+\gamma}(B_{3/4})}  \le C \|\nabla u\|_{L^{2}(B_{1})}, 
\]
where $\gamma>0$ and $C$ are dimensional constants.
\end{proposition}

\begin{proof}
Without loss of generality, we can assume that $\|\nabla u\|_{L^{2}(B_{1})}=1$ (this normalization will be particularly convenient in Step 3).
Let $\eta\in C^\infty_c(B_1)$ be a nonnegative cut-off function with $\eta\equiv 1$ in $B_{3/4}$.
\\

\noindent
{\bf Step 1:}  
{\it We show that} 
\begin{equation}\label{estdiv}
\int_{B_1} \big| {\rm div}(|\nabla u|\, \nabla u) \big| \eta^2 \,dx\le C.
\end{equation}

Set $\nu:=-\frac{\nabla u}{|\nabla u|}$ in the set $\{|\nabla u|\neq 0\}$, and $\nu=0$ in $\{|\nabla u|=0\}$. 
We begin from the pointwise identity
\begin{equation}\label{hwuighwiu}
 {\rm div}(|\nabla u|\, \nabla u)  = |\nabla u| \biggl( \sum_{ij} \frac{u_{ij}u_i u_j}{|\nabla u|^2} + \Delta u\biggr)  
=    -|\nabla u|\,{\rm tr}\big( D^2 u - (D^2u[\nu, \nu]) \nu\otimes\nu\big)   +  2|\nabla u|\Delta u
\end{equation}
in the set $\{|\nabla u|\neq0\}$.
Also, we note that $\J^2$ (as defined in Lemma \ref{conseqestab}) is larger or equal than half the squared Hilbert-Schmidt norm of the matrix $D^2 u - (D^2u[\nu, \nu])\,\nu\otimes\nu$,\footnote{This is easily seen by writing $D^2 u(x)$ in the orthonormal basis given by $\nu(x)$ and the principal directions of the level set of $u$ at $x$.} and hence there exists a dimensional constant $C$ such that
\begin{equation}
\label{hwuighwiu2}
\big|{\rm tr}\big( D^2 u - (D^2u[\nu, \nu]) \nu\otimes\nu\big)\big|
\leq C\J. 
\end{equation}
Furthermore, thanks to Lemma \ref{conseqestab} we obtain (note that, in the next integrals, 
{we can indistinctly integrate in $B_1$ or in $B_1\cap\{|\nabla u|\neq0\}$})
\[\begin{split}
 -\int_{B_{1}} 2|\nabla u|&\,\Delta u \, \eta^2\,dx
 = -\int_{B_1} |\nabla u|\,{\rm tr}\big( D^2 u - (D^2u[\nu, \nu]) \nu\otimes\nu\big)\,  \eta^2\,dx -   \int_{B_1} {\rm div}(|\nabla u|\, \nabla u)\,\eta^2\,dx
 \\
 &\le  C\left(\int_{B_1} |\nabla u|^2\eta^2\,dx\right)^{1/2} \left(\int_{B_1} \J^2 \eta^2\,dx\right)^{1/2} +   \int_{B_1} \,|\nabla u|\, \nabla u \cdot \nabla (\eta^2)\,dx\le C.
\end{split}\]
Hence, combining this bound with \eqref{hwuighwiu} and \eqref{hwuighwiu2},
and using again Lemma \ref{conseqestab} together with the fact that $\Delta u\le 0$, we get
\[\begin{split}
\int_{B_{1}} \big| {\rm div}(|\nabla u|\, \nabla u) \big| \eta^2 \,dx &\le \int_{B_{1}} -2|\nabla u|\,\Delta u\, \eta^2\,dx  + C \int_{B_{1}} |\nabla u| \J\, \eta^2\,dx \\
&\leq C+ C\left(\int_{B_1} |\nabla u|^2\eta^2\,dx\right)^{1/2} \left(\int_{B_1} \J^2 \eta^2\,dx\right)^{1/2} \le C,
\end{split}\]
as desired.\\

\noindent
{\bf Step 2:} {\it We show that, for a.e. $t\in \R$,}
\begin{equation}\label{ahgiohwiob1}
\int_{\{u=t\}\cap B_{3/4}} |\nabla u|^2 d\HH^{n-1}  \le  C.
\end{equation}

We claim that, for a.e. $t\in\R$, we have
\begin{equation}\label{claim-12}
 \int_{\{u=t\}\cap B_{3/4}} |\nabla u|^2 d\HH^{n-1} \leq \int_{\{u=t\}\cap B_{1}} |\nabla u|^2\eta^2 d\HH^{n-1} =- \int_{\{u>t\} \cap B_1 }  {\rm div}\big(|\nabla u|\, \nabla u \,\eta^2\big) \,dx.
\end{equation}
Note that this bound, combined with \eqref{estdiv}, implies \eqref{ahgiohwiob1}. So, we only need to prove the validity of \eqref{claim-12}.

To show \eqref{claim-12} some care is needed to deal with the divergence, since we cannot use Sard's theorem here ($u$ is only $C^2$). 
Thus, 
to prove it, we consider 
$s \mapsto H_\epsilon(s)$ a smooth approximation of the indicator function of $\R_+$,
so that $H_\epsilon'(s)\rightharpoonup^*\delta_0$ as $\epsilon\to 0.$
Then, for any given $t \in \R$ we can apply Lemma \ref{lem:coarea}  with $g=H'_\epsilon(u-t)|\nabla u|^2\eta^2$ to get
\begin{multline*}
-\int_{B_1}H_\epsilon(u-t){\rm div}\big(|\nabla u|\, \nabla u \,\eta^2\big) \,dx
=
\int_{B_1}H'_\epsilon(u-t)\,\nabla u\cdot \big(|\nabla u|\, \nabla u \,\eta^2\big)\,dx\\
=\int_{B_1}H'_\epsilon(u-t)|\nabla u|^3\eta^2\,dx
=\int_\R H'_\epsilon(\tau -t)\biggl(\int_{\{u=\tau\}\cap B_{1}} |\nabla u|^2\eta^2 d\HH^{n-1} \biggr)\,d\tau.
\end{multline*}
In particular, whenever $t$ is a Lebesgue point for the $L^1$ function $\tau \mapsto \int_{\{u=\tau\}\cap B_{1}} |\nabla u|^2\eta^2 d\HH^{n-1}$, letting $\epsilon\to 0$ we deduce \eqref{claim-12}, as claimed.
\\

\noindent
{\bf Step 3:} {\it Conclusion.}

First note that, by the standard Sobolev-Poincar\'e inequality, for some dimensional $p>2$ we have
\begin{equation}\label{ashgowo}
\left(\int_{B_{1}} |u-\bar u|^p \, dx\right)^{\frac 1 p } \le C\left( \int_{B_1} |\nabla u|^2\,dx\right)^{\frac 1 2 }=C,
\end{equation}
where $\bar u := \ave_{B_1} u$. 
Thus, using \eqref{ashgowo} and Lemma \ref{lem:coarea} with $g=\frac{|u-\bar u|^p}{|\nabla u|}1_{\{|\nabla u| \neq 0\}}$, we obtain 
\begin{equation}\label{ahgiohwiob2}
 \int_{\R} dt \int_{\{u=t\}\cap B_{1}\cap \{|\nabla u|\neq0\}}  |t-\bar u|^p \,|\nabla u|^{-1} \, d\HH^{n-1} =   \int_{B_1} |u-\bar u|^p1_{\{|\nabla u| \neq 0\}}  \, dx  \le C.
\end{equation}
Also, since $p>2$, we may choose dimensional constants $q>1$ and $\theta\in (0,1/3)$ such that $p/q = (1-\theta)/\theta$.  Thus, 
defining 
\[
h(t) : =  \max\big\{1, |t-\bar u|\big\}
\]
and using the coarea formula (Lemma \ref{lem:coarea}) and H\"older inequality (note that $p\theta -q(1-\theta)=0$),  we obtain
\[
\begin{split}
 &\int_{B_{3/4}} |\nabla u|^{3-3\theta} \,dx=\int_{\R} dt\int_{\{u=t\}\cap B_{3/4}\cap \{|\nabla u|\neq0\}}  h(t)^{p\theta -q(1-\theta)}  |\nabla u|^{-\theta + 2(1-\theta)} \,d\HH^{n-1} 
\\
&\le   \left(\int_{\R} dt\int_{\{u=t\}\cap B_1\cap \{|\nabla u|\neq0\}} \hspace{-2mm} h(t)^{p}  |\nabla u|^{-1}\,d\HH^{n-1} \right)^\theta  \bigg(\int_{\R} dt\int_{\{u=t\}\cap B_{3/4}} \hspace{-3mm}  h(t)^{-q} |\nabla u|^2 \,d\HH^{n-1}  \bigg)^{1-\theta}.
\end{split}
\]

Observe now that, thanks to \eqref{ahgiohwiob2} and  the definition of $h(t)$, we have
\begin{align*}
\int_{\R} dt\int_{\{u=t\}\cap B_1\cap \{|\nabla u|\neq0\}}    h(t)^{p}  |\nabla u|^{-1}\,d\HH^{n-1}& 
\leq \int_{\bar u-1}^{\bar u+1} dt\int_{\{u=t\}\cap B_1\cap \{|\nabla u|\neq0\}}      |\nabla u|^{-1}\,d\HH^{n-1} +C\\
&\leq |B_1|+C\leq  C.
\end{align*}
Also, since $q>1$ it follows that  $\int_{\R} h(t)^{-q}dt$ is finite, and thus  \eqref{ahgiohwiob1} implies that
\[
\int_{\R} dt \, h(t)^{-q} \int_{\{u=t\}\cap B_{3/4}}  |\nabla u|^2\,d\HH^{n-1}  \le C\int_{\R} h(t)^{-q}\,dt \le C.
\] 
Therefore, we have proved that
\[
\int_{B_{3/4}} |\nabla u|^{3-3\theta}  \,dx \le C
\]
for some dimensional constants  $\theta\in(0,1/3)$  and $C$, as desired.
\end{proof}

We conclude this section with the following useful result.

\begin{proposition}\label{prop:L1controlsW12}
Let $u\in C^2(B_1)$ be a stable solution of $-\Delta u=f(u)$ in $B_1\subset \R^n$,   with $f$ locally Lipschitz and nonnegative. 
Then
\[
\|\nabla u\|_{L^{2}(B_{1/2})}  \le C \|u\|_{L^{1}(B_{1})}, 
\]
where $C$ is a dimensional constant.
\end{proposition}

\begin{proof}
Since $-\Delta u\ge 0$ we can apply Lemma \ref{strongconvergence}(i) to the constant sequence $v_k=u$ to get
\[
\|\nabla u\|_{L^{1}(B_{1/2})}\leq  C\|u\|_{L^{1}(B_{1})}.
\]
Also, it follows from Proposition  \ref{prop:W12+ep} that
\[
\|\nabla u\|_{L^{2+\gamma}(B_{1/2})}  \le C \|\nabla u\|_{L^{2}(B_{1})}.
\]
Therefore, by H\"older and Young inequalities, for every $\epsilon>0$ we have 
\[\begin{split}
\|\nabla u\|_{L^{2}(B_{1/2})}^2& \le \|\nabla u\|_{L^{1}(B_{1/2})}^{\frac{\gamma}{1+\gamma}}\|\nabla u\|_{L^{2+\gamma}(B_{1/2})}^{\frac{2+\gamma}{1+\gamma}}
\le C\|u\|_{L^{1}(B_{1})}^{\frac{\gamma}{1+\gamma}}  \|\nabla u\|_{L^{2}(B_{1})}^{\frac{2+\gamma}{1+\gamma}} \\
&\le \epsilon  \|\nabla u\|_{L^{2}(B_{1})}^2 + \frac{C}{\epsilon} \|u\|_{L^{1}(B_{1})}^2.
\end{split}\]

Applying this estimate to the functions $u_{r,y}(x):=u(y+rx)$, where $B_r(y)\subset B_1$ (note that $u_{r,y}$ is a stable solution to the semilinear equation $-\Delta u_{r,y}=f_{r}(u_{r,y})$ in $B_1$ with $f_r(t)=r^2f(t)$, so all the previous results apply to $u_{r,y}$ as well), we conclude that
\begin{align*}
r^{n+2}\int_{B_{r/2}(y)}|\nabla u|^2\,dx&\leq 
\epsilon r^{n+2}\int_{B_{r}(y)}|\nabla u|^2\,dx
+ \frac{C}{\epsilon}\left(\int_{B_{r}(y)}|u|\,dx\right)^2\\
& \leq \epsilon r^{n+2}\int_{B_{r}(y)}|\nabla u|^2\,dx
+\frac{C}{\epsilon}\left(\int_{B_{1}}|u|\,dx \right)^2
\end{align*}
for every $\epsilon>0$.
By Lemma \ref{lem_abstract} applied with
$\sigma(B):=\|\nabla u\|_{L^2(B)}^2$, the result follows.
\end{proof}

\section{\for{toc}{Interior $C^\alpha$ estimate for $n\leq9$, and global estimate in convex domains}\except{toc}{Interior $C^\alpha$ estimate and global estimate in convex domains: proof of Theorem~\ref{thm:L1-CalphaW12} and Corollary~\ref{thm:convex}}}

\label{sect:int C0a}

We begin this section by proving that, under a doubling assumption on $|\nabla u|^2dx$, the radial derivative of a stable solution controls its full derivative.

\begin{lemma}\label{lem:22}
Let $u\in C^2(B_2)$ be a stable solution of $-\Delta u=f(u)$ in $B_2\subset \R^n$, with $f$ locally Lipschitz and nonnegative. 
Assume that
$$
\int_{B_1}|\nabla u|^2\,dx\geq \delta \int_{B_2}|\nabla u|^2\,dx
$$
for some $\delta>0$.
Then there exists a constant $C_\delta$, depending only on $n$ and $\delta$, such that
$$
\int_{B_{3/2}}|\nabla u|^2 \,dx\leq C_\delta \int_{B_{3/2}\setminus B_1}|x\cdot \nabla u|^2\,dx.
$$
\end{lemma}

\begin{proof}
Assume the result to be  false.
Then, there exists a sequence of stable solutions $u_k$ (with $f_k\geq 0$ varying)
such that
\begin{equation}\label{eq:doubling}
\int_{B_1}|\nabla u_k|^2\,dx\geq \delta \int_{B_2}|\nabla u_k|^2\,dx,
\quad 
\int_{B_{3/2}} |\nabla u_k|^2\,dx =1,
\quad\text{and}\quad \int_{B_{3/2}\setminus B_1}|x\cdot \nabla u_k|^2\,dx\to 0.
\end{equation}
Now, thanks to \eqref{eq:doubling},
\begin{equation}\label{3.1bis}
 \int_{B_2}|\nabla u_k|^2\,dx\leq \frac{1}{\delta} \int_{B_1}|\nabla u_k|^2\,dx\leq   \frac{1}{\delta}\int_{B_{3/2}} |\nabla u_k|^2\,dx = \frac{1}{\delta} \le C.
\end{equation}
Therefore, using Proposition \ref{prop:W12+ep} (rescaled from $B_1$ to $B_2$) we obtain 
\[
 \int_{B_{3/2}}|\nabla u_k|^{2+\gamma}\,dx \le C.
\]

Hence, the sequence of superharmonic functions 
\[
v_k : = u_k -\ave_{B_2}u_k
\]
satisfies 
\[\|v_k\|_{L^1(B_2)} \leq C\|v_k\|_{L^2(B_2)} \leq C\]
(thanks to H\"older and Poincar\'e inequalities, and by \eqref{3.1bis}), as well as
\[
\|\nabla v_k\|_{L^2(B_{3/2})}=  1,\quad  \| v_k\|_{W^{1,2+\gamma}(B_{3/2})}\le C,
\quad\text{and} 
\quad
\int_{B_{3/2}\setminus B_1}|x\cdot \nabla v_k|^2\,dx\to 0.
\]
Thus it follows from Lemma \ref{strongconvergence} applied with $r=\frac32<2=R$ that, up to a subsequence, $v_k \to  v$ strongly in $W^{1,2}(B_{3/2})$ where $v$ is a superharmonic function in $B_{3/2}$ satisfying 
\[
\|\nabla v\|_{L^2(B_{3/2})}=  1 \qquad \mbox{and} \qquad x\cdot \nabla v \equiv 0 \quad \mbox{a.e. in }B_{3/2}\setminus B_1.
\]
From the fact that $v$ is $0$-homogeneous and superharmonic in the annulus  $B_{3/2}\setminus B_1$, it follows that $v=c_0$ inside $B_{3/2}\setminus B_1$ for some constant $c_0\in \R$.
Indeed, by the mean value property (or by Theorem 8.17 of \cite{GT}, since $u\in W^{1,1}_{\rm loc}\subset L^{\frac{n}{n-1}}_{\rm loc}$ by Lemma \ref{strongconvergence}), $v$ is bounded from below in $B_{3/2}\setminus B_1$.
As a consequence, by $0$-homogeneity, $\inf_{B_{3/2}\setminus B_1}v = \inf_{B_{1/4}(x_0)}v$ for some point $x_0\in \partial B_{5/4}$.
Hence, by the strong maximum principle (Theorem 8.19 of \cite{GT}), $v$ is constant in $B_{3/2} \setminus B_1$, {as desired.

In particular, we have proved that $v_{|\partial B_1}=c_0$, so by the maximum principle for superharmonic functions we get $v \geq c_0$ inside $B_1$.

Combining all this together, we get that}
$$
v \geq c_0\quad \text{in }B_{3/2}\qquad \text{and}\qquad 
v \equiv c_0\quad \text{in }B_{3/2}\setminus B_1,
$$
and by the strong maximum principle for superharmonic functions we get $v\equiv c_0$ in $B_{3/2}$, a contradiction with $\|\nabla v\|_{L^2(B_{3/2})}=  1$.
\end{proof}

The following lemma will be used a couple of times in the paper to prove geometric decay of certain integral quantities satisfying appropriate recurrence relations.

\begin{lemma}\label{lem:33}
Let $\{a_j\}_{j\geq 0}$ and $\{b_j\}_{j\ge 0}$ be two sequences of nonnegative numbers satisfying  $a_0\le M$,  $b_0\le M$,  
\[
b_j\le b_{j-1}   \quad \mbox{and}\qquad a_j+ b_j \le L a_{j-1} \qquad \mbox{for\ all}\ \,j\ge 1,
\]
and
\begin{equation}\label{wnohnwono}
\mbox{if } \quad a_{j}  \ge \frac 1 2  a_{j-1}  \quad \mbox{then}\quad    b_{j} \le  L(b_{j-1} -b_j) \qquad \mbox{for\ all}\ \,j\ge 1,
\end{equation}
for some positive constants $M$ and $L$.
Then there exist $\theta\in (0,1)$ and $C>0$, depending only on $L$, such that 
\[
b_j \le  C M \theta^j \qquad \mbox{for\ all}\  \,j\ge0.
\]
\end{lemma}

\begin{proof}
Define, for $\ep>0$ to be chosen, 
\[c_j : = a_j^\ep b_j.\]
We consider two cases, depending whether $a_{j}  < \frac 1 2  a_{j-1}$ or not.

\noindent
{\it - Case 1:} If $a_{j}  < \frac 1 2  a_{j-1}$, then 
since $b_j\le b_{j-1}$ we get
\[c_j = a_j^\ep b_j \leq  2^{-\ep} a_{j-1}^\ep b_{j-1}  =  2^{-\ep} c_{j-1}.\]

\noindent
{\it - Case 2:} If $a_{j}  \ge \frac 1 2  a_{j-1}$ we can apply \eqref{wnohnwono} and we have
$b_{j} \le  L(b_{j-1} -b_j)$ or, equivalently,
\[
b_{j} \le  \frac{L}{1+L} b_{j-1}.
\]
Therefore, using that $a_j \le L a_{j-1}$, we have
\[c_j = a_j^\ep b_j \leq  L^\ep a_{j-1}^\ep \frac{L}{1+L} b_{j-1}  =  \theta^{1+\ep} c_{j-1},\]
where we choose first  $\ep>0$  such that  $2^{-\ep} = L^{1+\ep}/(1+L)$ (this can be done since we may assume from the beginning that $L>1/2$), and then we define $\theta:= (2^{-\ep})^{\frac{1}{1+\ep}} = L/(1+L)^{\frac{1}{1+\ep}}$. 

Hence, we have proven that in both cases $c_j \leq \theta^{1+\ep}c_{j-1}$ {for some $\theta \in (0,1)$}.
By iterating this estimate we conclude that
$c_j \le \theta^{(1+\ep)j}c_0$. 

Finally, recalling that $b_j\le La_{j-1}$, $b_j\le b_{j-1}$,  $a_0\le M$, and $b_0\le M$, {recalling the definition of $c_{j-1}$ and $c_0$} we obtain  
\[
b_j^{1+\ep} \le  L^\ep a_{j-1}^{\ep}b_{j-1} {= L^\ep c_{j-1} \le \frac{L^\ep}{\theta^{1+\ep}}  \, \theta^{(1+\ep)j}c_0 \le C \theta^{(1+\ep)j} M^{1+\ep}}
\]
and the lemma follows.
\end{proof}

We can now prove Theorem \ref{thm:L1-CalphaW12}.

\begin{proof}[Proof of Theorem \ref{thm:L1-CalphaW12}]
We begin by noticing that, combining Propositions \ref{prop:W12+ep} and \ref{prop:L1controlsW12}, we immediately get the bound
$$
\|\nabla u\|_{L^{2+\gamma}(B_{3/8})}\leq C\|u\|_{L^1(B_1)}.
$$
Hence \eqref{eq:W12g L1 int} follows by a classical scaling and covering argument.

We are left with proving \eqref{eq:Ca L1 int}.
For this we may assume that $3 \leq n \leq 9$. (Indeed, recall that in case $n\leq2$ one can easily reduce to the case $n=3$ by adding extra artificial variables. Note that the stability condition is preserved under this procedure).
Given $\rho \in (0,1),$ we define the 
quantities 
$$
\mathcal D(\rho):=\rho^{2-n}\int_{B_\rho}|\nabla u|^2\,dx
\qquad \textrm{and} \qquad
\mathcal R(\rho):=\int_{B_\rho}|x|^{-n}|x\cdot \nabla u|^2\,dx .
$$
We split the proof  of  \eqref{eq:Ca L1 int} in three steps.
\\

\noindent
{\bf Step 1:}  {\it We prove that there exists a dimensional exponent $\alpha>0$ such that
$$
\mathcal R(\rho)\leq C\rho^{2\alpha} \|\nabla u\|^2_{L^2(B_{1/2})}
$$
for all $\rho \in (0,1/4)$.}

Recall that, by \eqref{eq:11}, for every $\rho\in (0,1/4)$ it holds
\begin{equation}\label{agnwiohwio}
\mathcal R(\rho)\leq C\rho^{2-n}\int_{B_{3\rho/2}\setminus B_\rho}|\nabla u|^2\,dx.
\end{equation}
Hence, if $\mathcal D(\rho)\geq \frac12 \mathcal D(2\rho)$ then we can apply Lemma \ref{lem:22} with $\delta=1/2$ to the function $u(\rho \,\cdot)$, and we deduce that
$$
\rho^{2-n}\int_{B_{3\rho/2}}|\nabla u|^2\,dx \leq C\rho^{-n} \int_{B_{3\rho/2}\setminus B_\rho}|x\cdot \nabla u|^2\,dx\leq C \bigl(\mathcal R(3\rho/2)-\mathcal R(\rho)\bigr)
$$
for some dimensional constant $C$.
Combining this bound with \eqref{agnwiohwio} and using that $\mathcal R$ is nondecreasing, we deduce that
\begin{equation}\label{ahsiogwio}
\mathcal R(\rho)\leq C\bigl(\mathcal R(2\rho)-\mathcal R(\rho)\bigr)
\qquad \text{provided $ \mathcal D(\rho)\geq {\textstyle \frac12}  \mathcal D(2\rho)$}.
\end{equation}
Thus, if we define $a_j : = \mathcal D(2^{-j-2})$, $b_j := \mathcal R(2^{-j-2})$ we have, for some dimensional constant $L>0$, 
\begin{itemize}
\item $b_j\le b_{j-1}$ for all $j\ge 1$ (since $\mathcal R$ is nondecreasing); 
\item $a_j+ b_j \le L a_{j-1}$   for all $j\ge 1$ (by \eqref{agnwiohwio});
\item if $a_{j}  \ge \frac 1 2  a_{j-1}$ then $b_{j} \le  L(b_{j-1} -b_j)$, for all  $ j\ge 1$ (by \eqref{ahsiogwio}).
\end{itemize}
Therefore, by Lemma \ref{lem:33} we deduce that 
\[
b_j \le C M \theta^{j},
\]
where $\theta\in(0,1)$ and  $M:= a_0 + b_0 \le C \|\nabla u\|^2_{L^2(B_{1/2})}$ (here we used again \eqref{agnwiohwio} in order to bound $b_0$).

Choosing $\alpha>0$ such that $2^{-2\alpha}=\theta$, Step 1 {follows easily.}
\\

\noindent
{\bf Step 2:}  {\it We show that 
\begin{equation}\label{ESTEST}
[u]_{C^\alpha(\overline B_{1/8})}\leq C\|\nabla u\|_{L^2(B_{3/4})},
\end{equation}
where $\alpha$ and $C$ are positive dimensional constants.}

Applying Step 1 to the function $u_y(x):=u(x+y)$ with $y\in B_{1/4}$, since $B_{1/2}(y)\subset B_{3/4}$ we get
$$
\int_{B_\rho(y)}|x-y|^{-n}|(x-y)\cdot \nabla u|^2\,dx\leq C\rho^{2\alpha}\int_{B_{3/4}}|\nabla u|^2\,dx\qquad \mbox{for all }\,\rho \leq 1/2.
$$
In particular,
$$
\rho^{2-n}\int_{B_\rho(y)}\Bigl|\frac{x-y}{|x-y|}\cdot \nabla u\Bigr|^2\,dx\leq C\rho^{2\alpha}\int_{B_{3/4}}|\nabla u|^2\,dx\qquad \mbox{for all }\,y \in B_{1/4},\,\rho \leq 1/2.
$$
Then, given $z \in B_{1/8}$, we can
average the above inequality with respect to $y\in B_{\rho/4}(z)$ to get
$$
\rho^{2-n}\int_{B_{\rho/8}(z)}|\nabla u|^2\,dx\leq C\rho^{2\alpha}\int_{B_{3/4}}|\nabla u|^2\,dx\qquad \mbox{for all }\,\rho \leq 1/2.
$$
Since $z \in B_{1/8}$ is arbitrary, by classical estimates on  Morrey spaces {(see for instance \cite[Theorem 7.19]{GT}) we deduce  \eqref{ESTEST}.}
\\

\noindent
{\bf Step 3:} {\it Proof of  \eqref{eq:Ca L1 int}.}

Note that, using Proposition \ref{prop:L1controlsW12} and a standard scaling and covering argument, we have
$$
\|\nabla u\|_{L^2(B_{3/4})}\leq C\|u\|_{L^1(B_1)}.
$$
Hence, it follows by Step 2 that $[u]_{C^\alpha(\overline B_{1/8})}\leq C\|u\|_{L^1(B_1)}$.
Also, by classical interpolation estimates, we have the bound
$$\|u\|_{L^\infty(B_{1/8})}\leq C\left([u]_{C^\alpha(\overline B_{1/8})}+\|u\|_{L^1(B_{1/8})}\right).$$
Combining these estimates, we conclude that 
$$
\|u\|_{C^\alpha(\overline B_{1/8})}\leq C\|u\|_{L^1(B_{1})}.
$$
Finally, \eqref{eq:Ca L1 int} follows by a classical scaling and covering argument.
\end{proof}

We conclude the section by proving global regularity in convex domains.

\begin{proof}[Proof of Corollary \ref{thm:convex}]
First of all, since $f \geq 0$ we have that $u$ is superharmonic, so by the maximum principle $u \geq 0$ in $\Omega$.

Since $\Omega$ is a {bounded convex domain of class $C^1$}, by the classical moving planes method there exists $\rho_0>0$, depending only on $\Omega$, such that
\begin{equation}
\label{eq:bdd near bdry}
u(x)\leq \max_{\Gamma_0}u \qquad \mbox{for all }\,x\in  N_{0},
\end{equation}
where $N_0:=\Omega \cap \{y\,:\,{\rm dist}(y,\partial\Omega)<\rho_0\}$ and $\Gamma_0:=\{y \in \Omega\,:\,{\rm dist}(y,\partial\Omega)=\rho_0\}$.\footnote{{Here we are using that, in any convex $C^1$ domain, we can start the classical moving planes method at any boundary point.

We note that the classical moving planes method is usually stated for strictly convex $C^1$ domains. If $\Omega$ is merely convex (instead of strictly convex), then the boundary may contain a piece of a hyperplane.
Still, by a simple contradiction argument one can show that, given any boundary point, there exist hyperplanes that separate a small cap around this point from their reflected points, and such that the reflected points are contained inside $\Omega$.
This suffices to use the moving planes method in a neighborhood of any boundary point.}}

Hence, it follows by Theorem \ref{thm:L1-CalphaW12} that $u \leq C\|u\|_{L^1(\Omega)}$  inside $\Omega\setminus N_0$, where $C$ depends only on $\Omega$ and $\rho_0$.
Thus, recalling \eqref{eq:bdd near bdry}, we conclude that $0 \leq u \leq C\|u\|_{L^1(\Omega)}$ inside $\Omega$.
\end{proof}

\section{A general closedness result for stable solutions with convex nondecreasing nonlinearities}
\label{sect:compact}

The goal of this section is to establish a very strong closedness property for stable solutions to equations with convex, nondecreasing, and nonnegative nonlinearities. 
As mentioned in the introduction, in addition to its own interest, this result will play a crucial role in the proof of the global regularity result of Theorem \ref{thm:globalCalpha}.

Define
\[
\mathcal C :=  \big\{ f: \R \to [0,+\infty] \,:\ \mbox{$f$ is lower semicontinuous, nondecreasing, and convex} \big\}.
\]
Note that functions $f\in \mathcal C$ are nonnegative but are allowed to take the value $+\infty$. 
This fact is important, since limits of nondecreasing convex nonlinearities $f_k:\R\to \R$ could become $+\infty$ in an interval $[M,\infty)$; this is why, in $\mathcal C$, we must allow $f$ to take the value $+\infty$.

For $f\in \mathcal C$ and $t\in\R$ such that $f(t)<+\infty$, the following is the definition and a property for $f'_-(t)$:
\begin{equation}\label{eq:f' convex}
f'_{-}(t) := \lim_{h\downarrow 0} \frac{f(t)-f(t-h)}{h} \ge \frac{f(t_2)-f(t_1)}{t_2-t_1} \quad \mbox{for all }t_1 < t_2  \leq t.
\end{equation}
If $f(t)=+\infty$ for some $t \in \R$, then we simply set $f'_-(t)=+\infty$.

Given an open set $U\subset \R^n$, we define 
\begin{equation}\label{class-S}
\mathcal S(U):= \left\{u\in W^{1,2}_{\rm loc}(U): \begin{array}{c} u \mbox{ is a stable weak solution of } \\ -\Delta u=f(u) \mbox{ in } U, \mbox{ for some } f\in \mathcal C \end{array}\right\}.
\end{equation} 
The meaning of weak solution is that of Definition \ref{defi1.1}.
In particular, since $f(u) \in L^1_{\rm loc}(U)$ then $f(u)$ is finite a.e., and since $f$ is nondecreasing we deduce that $f<+\infty$ on $(-\infty,\sup_{U}u)$.
Note also that, similarly, since $f'_-\geq0$ and $f$ is convex, we have that $f'_-<+\infty$ in $(-\infty,\sup_{U}u)$.

The following theorem states that,  given an open set $U\subset \R^n$, the set $\mathcal S(U)$ is 
closed in $L^1_{\rm loc}(U)$. 
This is particularly surprising since no bound is required on the nonlinearities.

\begin{theorem}
\label{thm:stability}
Let $U\subset \R^n$ be an open set. 
Let $u_k\in \mathcal S(U)$, and assume that 
$u_k\to u$ in $L^1_{\rm loc}(U)$ for some $u\in L^1_{\rm loc}(U)$. 

Then, $u\in \mathcal S(U)$
and the convergence $u_k\to u$ holds in $W^{1,2}_{\rm loc} (U)$.
If, in addition, $n\leq 9$ then the convergence also holds in $C^0(U)$.
\end{theorem}

{For the proof of this result we shall use} the interior estimates of Theorem \ref{thm:L1-CalphaW12}.
However we proved these interior estimates for $C^2$ solutions, while solutions in the class $\mathcal S(U)$ are in general only in $W^{1,2}$ ---notice that it may happen that $f(u(x_0))=f(\sup_U u)=+\infty$ for some $x_0\in U$. 
Thus, we will need to prove first that the interior estimates of Theorem \ref{thm:L1-CalphaW12} extend to all weak solutions in the class $\mathcal S (B_1)$ (see Corollary \ref{CalphaforS} below).
For this, we need the following useful approximation result.

\begin{proposition}\label{prop:approximation}
Let $f\in \mathcal C$ and assume that $u\in W^{1,2}(B_1)$ is a  stable weak solution of $-\Delta u = f(u)$ in $B_1$, with $f(u)\in L^1(B_1)$.

Then, one of the following holds:
\begin{itemize}
\item[(i)] $u\in C^2(B_1)$ and $f$ is real valued and Lipschitz on $(-\infty,\sup_{B_1}u)$.\footnote{Throughout the paper, whenever we say that a function $g$ is Lipschitz on some set $A$, we mean uniformly Lipschitz (even if the set $A$ is open), namely
$$
\sup_{x,y\in A,\,x\neq y}\frac{|g(x)-g(y)|}{|x-y|}<+\infty.
$$
This is in contrast with the terminology ``$g$ is locally Lipschitz'', which means that $g$ is Lipschitz on any compact subset of its domain of definition.
}
\item[(ii)] There exist a family of  nonlinearities $\{f_\ep\}_{\ep \in (0,1]} \subset \mathcal C$ and a family of stable solutions $\{u_\ep\}_{\ep \in (0,1]}\subset C^2(B_1)\cap W^{1,2}(B_1)$  of 
\[\left\{
\begin{array}{cl}
-\Delta u_\ep = f_\ep(u_\ep) \quad &\mbox{in }B_1
\\
u_\ep = u  &\mbox{on }\partial B_1
\end{array}\right.
\]
such that $f_\ep \leq f$, $u_\ep \leq u$, and both $f_\ep \uparrow f$ $($pointwise in $\R)$ and $u_\ep \uparrow u$  $($a.e. and weakly in $W^{1,2}(B_1))$ as $\ep\downarrow 0$. Furthermore, $f_\ep$ is real valued and Lipschitz on $(-\infty,\sup_{B_{r}}u_\ep]$ for every  $r<1$.
\end{itemize}
\end{proposition}

\begin{proof}
If $f'_{-}(\sup_{B_1} u) < +\infty$,\footnote{If $\sup_{B_1} u=+\infty$, we define $f'_{-}(\sup_{B_1} u):=\lim_{t\to +\infty} f'_{-}(t)$.} then $f$ is real valued and Lipschitz on $(-\infty,\sup_{B_1}u)$ (here we use that $f$ is nonnegative, nondecreasing, and convex).
Thus $|f(u)|\leq C(1+|u|)$, and  by classical elliptic regularity \cite{GT} $u$ is of class $C^2$ inside $B_1$.
Thus, in this case, (i) in the statement holds.

As a consequence, in order to establish (ii) we may assume that 
\begin{equation}
\label{eq:f'}
f'_{-}( \sup_{B_1} u ) = +\infty.
\end{equation}

\noindent
{\bf Step 1:} {\it Construction of $f_\ep$ and $u_\ep$.}

Given $\ep \in (0,1]$, if $\sup_{B_1} u <+\infty$ we define $f_\ep$ by $f_\ep(t):=(1-\ep)f(t)$. Instead, when  $\sup_{B_1} u =+\infty$ we set
$$
f_\ep(t):=\left\{\begin{array}{cl}
(1-\ep)f(t) \quad &\mbox{ for }t \leq \ep^{-1},
\\
(1-\ep) \left( f(\ep^{-1})+f'_-(\ep^{-1})(t-\ep^{-1}) \right)  &\mbox{ for }t \geq \ep^{-1}.
\end{array}\right.
$$
Note that in both cases $f_\ep\in \mathcal C$, $f_\ep \leq f$, and $f_\ep \uparrow f$  pointwise as $\ep\downarrow 0$.

We  now construct the functions $u_\ep$. We first define the function $u_\ep^{(0)}$ to be the harmonic extension of $u$.
Indeed, since $u\in W^{1,2}(B_1)$, the Dirichlet energy $\int_{B_1} |\nabla v|^2$ admits a minimizer $u_\ep^{(0)} $ in the convex set 
$\{v \in W^{1,2}(B_1)\,:\,v-u\in  W^{1,2}_0(B_1)\}$. 
Note that $u_\ep^{(0)}\leq u$ since $u$ is weakly superharmonic. 

To construct $u_\ep$ for $\ep\in (0,1)$ we start a monotone iteration by defining, for $j\ge 1$, the function $u_\ep^{(j)}$ as the solution to the linear problem 
\begin{equation}
\label{eq:u eps j}
\left\{\begin{array}{cl}
-\Delta u_\ep^{(j)} = f_\ep(u_\ep^{(j-1)}) \quad &\mbox{in }B_1
\\
u_\ep^{(j)} = u  &\mbox{on }\partial B_1.
\end{array}\right.
\end{equation}
Note that we can start the iteration since $0\leq f_\ep ( u_\ep^{(0)} )\leq f_\ep (u)\leq f(u)\in L^1(B_1)$. All the other problems also make sense, since we have that $u_\ep^{(j)}\le u$ for all $j\ge 0$. Indeed,
\begin{equation*}
\begin{split}
-\Delta (u-u_\ep^{(j)})  &=  f(u)-f_\ep(u_\ep^{(j-1)})= \big(f(u)-f_\ep(u)\big) + \big(f_\ep(u)- f_\ep(u_\ep^{(j-1)})  \big) \\
&\geq f_\ep(u)- f_\ep(u_\ep^{(j-1)})\qquad \mbox{for all }\,j \geq 1,
\end{split}
\end{equation*}
and since  $f_\ep$ is nondecreasing it follows
by induction that $u_\ep^{(j)}\le u$.

To prove that the sequence is monotone, note that, since $f_\ep\ge0$, it follows by the maximum principle that $u_\ep^{(1)} \ge u_\ep^{(0)}$.
Also, since $f_\ep$ is nondecreasing, the inequality
\[
-\Delta (u_\ep^{(j)}-u_\ep^{(j-1)}) =  f_\ep(u_\ep^{(j-1)})-f_\ep(u_\ep^{(j-2)})  \qquad \mbox{for all }\,j \geq 2
\]
proves, by induction on $j$, that $u_\ep^{(j)}\ge u_\ep^{(j-1)}$. 

Analogously, since  $f_\ep \leq f_{\ep'}$ for $\ep'<\ep$, using that $u_\ep^{(0)}=u_{\ep'}^{(0)}$ and that 
$$
-\Delta (u_{\ep'}^{(j)}-u_\ep^{(j)})  = f_{\ep'}(u_{\ep'}^{(j-1)})-f_\ep(u_\ep^{(j-1)})\qquad \mbox{for all }\,j \geq 1,
$$
again by induction we get
\begin{equation}\label{hgwiohiowh 2}
u_\ep^{(j)}\le u_{\ep'}^{(j)}\qquad \text{ for all $j\ge 0$ and $\ep'<\ep$.}
\end{equation}

\smallskip

\noindent
{\it Claim 1: the functions $u_\ep^{(j)}$ belong to $W^{1,2}(B_1)$ and their $W^{1,2}$-norms are uniformly bounded in $j$ and $\ep$.}\\
Indeed, 
since
\[\int_{B_1} \nabla u_\ep^{(j)}\cdot\nabla (u-u_\ep^{(j)})\,dx=\int_{B_1}f_\ep(u_\ep^{(j-1)})(u-u_\ep^{(j)})\,dx\geq0\]
we have
\[\int_{B_1} |\nabla (u-u_\ep^{(j)})|^2\,dx\leq \int_{B_1} \nabla u\cdot\nabla (u-u_\ep^{(j)})\,dx \leq \|\nabla u\|_{L^2(B_1)}\|\nabla (u-u_\ep^{(j)})\|_{L^2(B_1)},\]
and therefore
\begin{equation}\label{unifW12}
\|\nabla u_\ep^{(j)}\|_{L^2(B_1)} \leq \|\nabla (u-u_\ep^{(j)})\|_{L^2(B_1)}+\|\nabla u\|_{L^2(B_1)}\leq 2\|\nabla u\|_{L^2(B_1)}.
\end{equation}
Since $u_\ep^{(j)}-u$ vanishes on $\partial B_1$,
the claim follows by Poincar\'e inequality.

\smallskip

Thanks to Claim 1, we can define
\[
u_\ep := \lim_{j\to\infty }u_\ep^{(j)} \leq u,
\]
where $u_\ep$ is both a pointwise limit (since the sequence is nondecreasing in $j$) and a weak $W^{1,2}(B_1)$ limit.
Then we have that $u_\ep\in W^{1,2}(B_1)$ is a weak solution of
\[\left\{
\begin{array}{cl}
-\Delta u_\ep = f_\ep(u_\ep) \quad &\mbox{in }B_1
\\
u_\ep = u  &\mbox{on }\partial B_1.
\end{array}\right.
\]
We now want to show that $u_\ep$ is of class $C^2$. For this, we prove the following:

\smallskip

\noindent
{\it Claim 2: the functions  $u_\ep^{(j)}$ belong to $C^{2,\beta}_{\rm loc}(B_1)$, for every $\beta\in (0,1)$, 
and their norms in this space are uniformly bounded with respect to $j$. In addition, $f_\ep$ is real valued and Lipschitz on $(-\infty,\sup_{B_{r}}u_\ep]$, for every $r<1$.
}\\
To prove this result, we distinguish two cases, depending whether $\sup_{B_1}u$ is finite or not.\\
{\it - Case (i): $\sup_{B_1}u<+\infty$}.\\
Note that, since in this case $f_\ep = (1-\ep)f$, we have $-\Delta (u-u_\ep) \ge \ep f(u) \geq 0$.
Also, $f(u)$ cannot be identically zero, since $f'_{-}(\sup_{B_1} u) = +\infty$ by \eqref{eq:f'}. Thus, it follows by the Harnack inequality that, for all $r\in (0,1)$ and $\ep>0$, there exists a constant $\delta_{\ep,r}>0$ such that $u_\ep \le u- \delta_{\ep,r}$ in $B_r$.

In addition, as already observed after \eqref{class-S}, the fact that $-\Delta u=f(u)$
with  $f\in \mathcal{C}$ leads to $f<+\infty$ on $(-\infty,\sup_{B_1}u).$
Hence, using again that $f\in \mathcal{C}$ (thus $f\geq 0$ is convex and nondecreasing), we obtain that
\[
 \| f \|_{C^{0,1}((-\infty,t])} \le C(f,t) <\infty \qquad \mbox{for all } t<\sup_{B_1} u.
\]
Therefore, since
$u_\ep^{(j)} \leq u_\ep \leq \sup_{B_1} u-\delta_{\ep,r}$ in $B_r$,
by standard elliptic regularity (see for instance \cite[Chapter 6]{GT}) we obtain that $u_\ep^{(j)} \in C^{2,\beta}_{\rm loc}(B_1)$ for all $\beta\in (0,1)$, uniformly in~$j$, as desired.
Furthermore, since $u_\ep \le u- \delta_{\ep,{r}}$ in $B_{r}$, $f_\ep$ is real valued and Lipschitz on $(-\infty,\sup_{B_{r}}u_\ep]$.
\\
{\it - Case (ii): $\sup_{B_1}u=+\infty$}.\\
In this case we note that, by construction, $f_\ep$ is globally Lipschitz on the whole $\R$ and $|f_\ep(t)|\leq C_\ep(1+|t|)$.
Hence, thanks to the uniform $W^{1,2}$ bound on $u_\ep^{(j)}$ (see \eqref{unifW12}),
using \eqref{eq:u eps j} and 
standard elliptic regularity (see for instance \cite[Chapter 6]{GT}), it follows by induction on~$j$ that $u_\ep^{(j)} \in C^{2,\beta}_{\rm loc}(B_1)$ for all $\beta\in (0,1)$, uniformly with respect to $j$.

\smallskip

Thanks to Claim 2, we have that $u_\ep$ is the limit of a sequence of functions uniformly  bounded in $C^{2,\beta}_{\rm loc}(B_1)$, and hence $u_\ep \in C^2(B_1)$.
\\

\noindent
{\bf Step 2:} {\it The solutions $u_\ep$ are stable.}

Since $u_\ep\leq u$, it follows by the definition of $f_\ep$ that $f'_-(u)\geq (f_\ep)'_-(u_\ep)$ in $B_1$. Hence,
the stability of $u$ gives that
\[\int_{B_1}|\nabla \xi|^2\,dx\geq \int_{B_1}f'_-(u)\,\xi^2\,dx\geq  \int_{B_1}(f_\ep)'_-(u_\ep)\,\xi^2\,dx\]
for all $\xi\in C^\infty_c(B_1)$.
Thus, $u_\ep$ is stable.
\\

\noindent
{\bf Step 3:} {\it $u_\ep \uparrow u$ as $\ep \downarrow 0$.}

Recall that $u_\ep \leq u_{\ep'}\leq u$ for $\ep'<\ep$, and that the functions $u_\ep$ are uniformly bounded in $W^{1,2}$
(see \eqref{hgwiohiowh 2} and \eqref{unifW12}).
Assume by contradiction that $u_\ep \uparrow  u^* \leq  u$  as $\ep \downarrow 0$ and $u^* \not\equiv  u$. Then, by the convergence of $f_\ep$ to $f$, $u^*$ solves
\[
-\Delta u^* = f(u^*)\quad \textrm{in}\ B_1, \qquad  u-u^* \in W^{1,2}_0(B_1),\quad  u-u^* \geq 0,\qquad  u-u^*\not\equiv 0,
\]
and thus, by the Harnack inequality applied to the superharmonic function $u-u^*$,
for any $r<1$ there exists a positive constant $\delta_r$ such that
 $u-u^*\geq \delta_r>0$ inside $B_r$. On the other hand,
testing the stability inequality for $u$ with $u-u^*$ we obtain 
\[
\int_{B_1} \big( f(u)-f(u^*)\big)(u-u^*)\,dx =  \int_{B_1} |\nabla (u-u^*)|^2\,dx \ge \int_{B_1} f'_{-}(u) (u-u^*)^2\,dx.
\]
Recalling \eqref{eq:f' convex} and that $u> u^*$, this leads to $f'_{-}(u) (u-u^*)^2 = \big( f(u)-f(u^*)\big)(u-u^*) $ a.e. in $B_1$ and (since $f$ is convex) we deduce that $f$ is linear in the interval $[u^*(x),u(x)]$ for a.e. $x\in B_1$.

Let $r<1$ and note that the intervals $[u^*(x),u(x)]$ have length at least $\delta_r$ for a.e.\ $x \in B_r$.
Hence, since $u$ and $u^*$ belong to $W^{1,2}(B_r)$, the union of these intervals as $x$ varies a.e.\ in $B_r$ covers all the interval $(\inf_{B_r} u^*, \sup_{B_r} u)$.\footnote{
Here it is crucial that the union of these intervals covers the full interval $(\inf_{B_r} u^*, \sup_{B_r} u)$, and not just a.e.
A way to see this is to note that, since the intervals $[u^*(x),u(x)]$ have length at least $\delta_r$, if this was not true then the essential image of $u$  (resp. $u^*$) would miss an interval of length $\delta_r$ inside its image. However, $W^{1,2}$ functions cannot jump between two different values, as can be seen by using the classical De Giorgi's intermediate value lemma  (see for instance \cite[Lemma 1.4]{CV}, or \cite[Lemma 3.13]{Xavi2} for an even simpler proof). 
}
This leads to $f$ being linear on the whole interval $(\inf_{B_r} u^*, \sup_{B_r} u)$.
 Letting $r\to 1$, this gives that $f$ is linear on  $(\inf_{B_1} u^*, \sup_{B_1} u)$,
contradicting $f'_-(\sup_{B_1} u) = +\infty$ (recall \eqref{eq:f'}) and concluding the proof.
\end{proof}

As a consequence, we find the following.

\begin{corollary}\label{CalphaforS}
The interior estimates of Theorem \ref{thm:L1-CalphaW12}  extend to all weak solutions in the class $\mathcal S (B_1)$.
\end{corollary}

\begin{proof}
In case (i) of Proposition \ref{prop:approximation}, when $\sup_{B_1} u <+\infty$ we have that the limits of $f(t)$ and $f_{-} '(t)$, as $t \uparrow  \sup_{B_1} u$, exist and are finite. This follows from $f$ being convex and Lipschitz in $(-\infty, \sup_{B_1} u)$, as stated in case (i). Thus, we can extend $f$ on $[\sup_{B_1}u,+\infty)$  to a globally Lipschitz, nondecreasing, convex function in all of $\R$, and then apply Theorem \ref{thm:L1-CalphaW12}. Obviously, there is no need to make the extension if $\sup_{B_1} u =+\infty$.

In case (ii) of Proposition \ref{prop:approximation}, take $r<1$. Since $f_\ep$ is Lipschitz on $(-\infty,\sup_{B_{r}}u_\ep]$, we can extend $f_\ep$ on $[\sup_{B_{r}}u_\ep,+\infty)$ to a globally Lipschitz, nondecreasing, convex function in all of~$\R$, and then apply Theorem \ref{thm:L1-CalphaW12} (rescaled from $B_{1}$ to $B_r$) to $u_\ep.$ Letting $\ep \downarrow 0$, this proves the validity of the interior estimates of Theorem \ref{thm:L1-CalphaW12} inside $B_{r/2}$, and letting $r\to 1$ yields the result.
\end{proof}

We can now prove Theorem \ref{thm:stability}.

\begin{proof}[Proof of  Theorem \ref{thm:stability}]
By assumption we have a sequence  $u_k\in \mathcal S(U)$ of weak solutions of $-\Delta u_k = f_k(u_k)$, with $f_k\in \mathcal C$ and $U$ an open set of  $\R^n$, such that $u_k\to u$ in $L^1_{\rm loc}(U)$. 
Then, by  Corollary \ref{CalphaforS} and Lemma \ref{strongconvergence}, the previous convergence also holds  in $W^{1,2}_{\rm loc} (U)$.
Also, up to a subsequence, we can assume that $u_k\to u$ a.e.
If $n\leq9$, the same results give that $u_k\to u$ locally uniformly in $U$.
{However, since in order to prove $u\in \mathcal S(U)$ we are not assuming $n\leq 9$, we cannot use this information.}
\\

\noindent
{\bf Step 1:} {\it A compactness estimate on $f_k$}.

Let $M:=\sup_U u\in (-\infty,+\infty]$, and let $m<M$.
We claim that
\begin{equation}\label{hwioghwo}
\limsup_{k\to\infty} f_k(m)<\infty.
\end{equation}
Indeed, let $x_0\in U$ be a Lebesgue point for $u$ such that\footnote{The existence of such a point is guaranteed again by the fact that $W^{1,2}$ functions cannot jump,
as noted in Step 3 of the proof of Proposition~\ref{prop:approximation}.}
\[m<u(x_0)<M,\]
and set $\delta:=u(x_0)-m>0.$
Since $x_0$ is a Lebesgue point, there exists $\ep_0>0$ such that $\overline{B}_{2\ep_0}(x_0)\subset U$ and
$$
\ave_{B_\ep(x_0)}|u(x)-u(x_0)|\,dx \le \frac{\delta}2\qquad \textrm{for all }\,\ep \in (0,2\ep_0].
$$ 
In particular, for $k$ sufficiently large we have
$$
m\leq \ave_{B_\ep(x_0)}u_k\,dx\leq \ave_{B_\ep(x_0)}|u_k|\,dx \le |u(x_0)|+\delta\qquad \textrm{for all }\,\ep \in (0,2\ep_0].
$$ 
Hence, since $f_k$ is nondecreasing and convex, applying
Jensen's inequality and  Lemma \ref{strongconvergence}(a) we get 
\[\begin{split}f_k(m) &\leq 
f_k\biggl(\ave_{B_{\ep_0}(x_0)}u_k\,dx\biggr)\leq 
\ave_{B_{\ep_0}(x_0)} f_k(u_k)\,dx =\ave_{B_{\ep_0}(x_0)} (-\Delta u_k)\,dx\\
&\leq C\ep_0^{-2} \ave_{B_{2{\ep_0}}(x_0)} |u_k|\,dx \leq C\ep_0^{-2}\bigl(|u(x_0)|+\delta\bigr)\end{split}\]
for a dimensional constant $C$ and all $k$ sufficiently large, proving \eqref{hwioghwo}.

Notice now that, since
\[
(f_k)_-'(m)\leq \frac{f_k(m+\delta)-f_k(m)}{\delta}\leq \frac{f_k(m+\delta)}{\delta}
\]
and $m+\delta=u(x_0)<M$, \eqref{hwioghwo} applied with $m$ replaced by $m+\delta$
implies that the functions $f_k$ are uniformly Lipschitz on $(-\infty, m]$.
Hence, by Ascoli-Arzel\`a Theorem and a diagonal argument, we deduce the existence of a function $f:(-\infty, M)\to \R$ such that that $f_k\to f$ uniformly on $(-\infty,m]$ for every $m<M$.
Also, since $f_k$ are nonnegative, nondecreasing, and convex, extending $f$ to all $\R$ by defining $f(M): = \lim_{t\uparrow M} f(t)$ and $f(t) := +\infty$ for  $t>M$, it is easy to check that $f\in \mathcal C$.
\\

\noindent
{\bf Step 2:} {\it $-\Delta u = f(u)$ in $U$.}

For every $\xi\in C^\infty_c(U)$ we have
\begin{equation}
\label{eq:PDE k}
\begin{split}
\int_U \nabla u\cdot \nabla \xi\,dx& = -\int_U u\Delta\xi\,dx = -\lim_{k\to \infty} \int_U u_k\Delta\xi\,dx 
 =\lim_{k\to \infty} \int_U \nabla u_k\cdot \nabla \xi\,dx \\
 & =\lim_{k\to \infty} \int_U f_k(u_k) \,\xi\,dx.
\end{split}\end{equation}
Note that, since since $f_k \to f$ locally uniformly on $(-\infty,M)$ and $u_k\to u$ a.e., it follows that
\begin{equation}
\label{eq:fk to f}
f_k(u_k)\to f(u)\qquad \text{a.e. inside }\{u <M\}.
\end{equation}
In the following,  $\eta \in C^\infty_c(U)$ denotes a nonnegative cut-off function such that $\eta=1$ on the support of~$\xi$.

\vspace{2mm}

\noindent
{\it Case 1:} $M=+\infty$. 
We have
$$
\int_{{{\rm supp}(\xi)}}f_k(u_k)u_k\,dx \leq 
\int_{U}f_k(u_k)|u_k|\eta\,dx 
=\int_{U}(-\Delta u_k)|u_k|\eta\,dx
=\int_U \nabla u_k\cdot\nabla (|u_k|\eta)dx \leq C
$$
for some constant $C$ independent of $k$, 
where the last bound follows from the $W^{1,2}_{\rm loc}$ boundedness of $u_k$.
In particular, given a continuous function $\varphi:\R\to [0,1]$ such that $\varphi=0$ on $(-\infty,0]$ and $\varphi=1$ on $[1,+\infty)$,
 we deduce that
\begin{equation}
\label{eq:fk geq j}
\begin{split}
\int_{{{\rm supp}(\xi)}}
f_k(u_k)\,\varphi(u_k-j)\,dx&\leq 
\int_{{{\rm supp}(\xi)}\cap \{u_k>j\}}
f_k(u_k)\,dx\\
&\leq 
\frac1j\int_{{{\rm supp}(\xi)}\cap \{u_k>j\}}
f_k(u_k)u_k\,dx\leq \frac{C}{j}\qquad \textrm{for all }\,j>1.
\end{split}\end{equation}
Therefore, by Fatou's Lemma (since $u_k \to u$ a.e. and $f_k(u_k)\to f(u)$ a.e. by \eqref{eq:fk to f} and $M=+\infty$), we also have
\begin{equation}
\label{eq:f geq j}
\int_{{{\rm supp}(\xi)}}
f(u)\,\varphi(u-j)\,dx\leq  \frac{C}{j}\qquad \textrm{for all }\,j>1.
\end{equation}
Furthermore, using again that $u_k \to u$ a.e. and $f_k(u_k)\to f(u)$ a.e., by dominated convergence we get
$$
f_k(u_k)\,[1-\varphi(u_k-j)]\to f(u)\,[1-\varphi(u-j)] \qquad \text{in $L^1({\rm supp}(\xi))$}.
$$
This, combined with \eqref{eq:fk geq j} and \eqref{eq:f geq j}, gives that
\[
\begin{split}
\limsup_{k\to \infty}\int_{{{\rm supp}(\xi)}}
|f_k(u_k)-f(u)|\,dx
\leq 
\limsup_{k\to \infty}& \int_{{{\rm supp}(\xi)}}
f_k(u_k)\,\varphi(u_k-j)\,dx\\
&+\int_{{{\rm supp}(\xi)}}
f(u)\,\varphi(u-j)\,dx \leq \frac{2C}{j}.
\end{split}\]
By the arbitrariness of $j$, this proves that
$$
f_k(u_k)\to f(u) \qquad \text{in $L^1({\rm supp}(\xi))$}.
$$
Recalling \eqref{eq:PDE k}, this concludes the proof of Step 2 in the case $M=+\infty$.

\vspace{2mm}

\noindent
{\it Case 2:} $M<+\infty$. 
Let $\delta>0$.
Since $(u_k-M-\delta)_+\geq \delta$ inside $\{u_k>M+2\delta\}$ and $-\Delta u_k=f_k(u_k) \geq 0$, we have
\[\begin{split}
\delta \int_{{\rm supp}(\xi)\cap \{u_k>M+2\delta\}}f_k(u_k)\,dx &=
\delta \int_{{\rm supp}(\xi)\cap \{u_k>M+2\delta\}}-\Delta u_k\,dx \\
&\leq \int_{{\rm supp}(\xi)\cap \{u_k>M+2\delta\}}-\Delta u_k\,(u_k-M-\delta)_+\,dx\\
&\leq \int_{U}-\Delta u_k\,(u_k-M-\delta)_+\eta\,dx\\
&=\int_{U\cap \{u_k>M+\delta\}}|\nabla u_k|^2\eta\,dx+\int_{U}\nabla u_k\cdot \nabla \eta\,(u_k-M-\delta)_+\,dx.
\end{split}\]
Note that, thanks to the higher integrability estimate \eqref{eq:W12g L1 int} applied to $u_k$ (recall Corollary~\ref{CalphaforS}), the functions $u_k$ 
are uniformly bounded in $W^{1,2+\gamma}({\rm supp}(\eta))$.
Thus, since  $1_{\{u_k>M+\delta\}}\to 0$ and $(u_k-M-\delta)_+\to 0$ a.e., we deduce from H\"older's inequality that the last two integrals tend to $0$ as $k \to \infty$, and therefore
\begin{equation}
\label{eq:uk M delta}
\lim_{k\to \infty}\int_{{\rm supp}(\xi)\cap \{u_k>M+2\delta\}}f_k(u_k)\,dx=0.
\end{equation}
On the other hand, we note that $f_k(u_k-3\delta)\leq f_k(M-\delta)\leq C_\delta$ inside ${\rm supp}(\xi)\cap \{u_k\leq M+2\delta\}$,
for some constant $C_\delta$ depending on $\delta$ but not on $k$.  
Hence, thanks to \eqref{eq:uk M delta} and the uniform convergence of $f_k$ to $f$ on $(-\infty,M-\delta]$, we get (recall that $u \leq M$ a.e.)
\[\begin{split}
&\lim_{k\to \infty} \int_U f_k(u_k)\, \xi\,dx =\lim_{k\to \infty} \int_{U\cap \{u_k\leq M+2\delta\}}f_k(u_k)\,\xi\,dx\\
&= \lim_{k\to \infty} \bigg\{ \int_{U\cap \{u_k\leq M+2\delta\}} f_k(u_k-3\delta)\, \xi\,dx +\int_{U\cap \{u_k\leq M+2\delta\}} \bigl(f_k(u_k)-f(u_k-3\delta)\bigr)\, \xi\,dx\bigg\} \\
&= \int_U f(u-3\delta) \xi\,dx +\lim_{k\to \infty} \int_{U\cap \{u_k\leq M+2\delta\}} \big(f_k(u_k)-f(u_k-3\delta)\bigr) \xi\,dx.\end{split}\]
Now, by \eqref{eq:f' convex}, the definition of $f'_-$, and the stability of $u_k$, we have (recall that $\eta=1$ on the support of $\xi$)
\[\left|\int_U \big(f_k(u_k)-f(u_k-3\delta)\bigr) \xi\,dx\right|\leq 3\delta \int_U (f_k)'_-(u_k)|\xi|\,dx\leq 3\delta \|\xi\|_\infty \int_U (f_k)'_-(u_k)\eta^2\,dx\leq C\delta\]
and therefore, letting $\delta\to0$, by  monotone convergence we find
\[\lim_{k\to \infty} \int_U f_k(u_k)\, \xi\,dx= \int_U f(u)\, \xi\,dx.\]
Recalling \eqref{eq:PDE k},
this proves that $-\Delta u = f(u)$ inside $U$ in the case $M<+\infty$.
\\

\noindent
{\bf Step 3:} {\it $u$ is stable.} 

Thanks to the convexity of $f_k$, it follows from \eqref{eq:f' convex}
and the stability inequality for $u_k$ that, for any $\delta>0$,
$$
\int_U \frac{f_k(u_k-2\delta)-f_k(u_k-3\delta)}{\delta}\, \xi^2\,dx\leq \int_U |\nabla \xi|^2 \,dx\qquad \mbox{ for all } 
\xi\in C_c^{\infty}(U).
$$
Hence, since $u_k\to u$ a.e. in $U$ and 
$f_k\to f$ locally uniformly in $(-\infty,m]$ for all $m<M$, and since $f_k$ is nondecreasing, it follows by Fatou's lemma applied to the sequence $1_{\{u_k\leq \min\{j,M\}+\delta\}}\, \delta^{-1}\bigl(f_k(u_k-2\delta)-f_k(u_k-3\delta)\bigr)$ that, for any $j>1,$
\[\begin{split}
\int_{U\cap \{u\leq \min\{j,M\}\}} &\frac{f(u-2\delta)-f(u-3\delta)}{\delta} \,\xi^2\,dx\leq  \\
&\leq \liminf_{k\to \infty} \int_{U\cap  \{u_k\leq \min\{j,M\}+\delta\}} \frac{f_k(u_k-2\delta)-f_k(u_k-3\delta)}{\delta}\, \xi^2\,dx\\
&\leq \liminf_{k\to \infty} \int_U \frac{f_k(u_k-2\delta)-f_k(u_k-3\delta)}{\delta}\, \xi^2\,dx\\
&\leq 
\int_U |\nabla \xi|^2 \,dx\qquad \mbox{ for all } 
\xi\in C_c^{\infty}(U).
\end{split}\]
Since 
$$
\frac{f(t-2\delta)-f(t-3\delta)}{\delta}\uparrow f'_-(t)\qquad \text{as $\delta \to 0$,}\quad \mbox{for all }\,t\leq \min\{j,M\},
$$
the result follows by the monotone convergence theorem, letting first $\delta \to 0$ and then $j\to +\infty$.
\end{proof}

\section{Boundary  $W^{1, 2+\gamma}$ estimate}
\label{sect:bdry W12g}

In this section we prove a uniform $W^{1, 2+\gamma}$ bound near the boundary in terms only of the $L^1$ norm of the solution.
As in the interior case (see Section \ref{sect:interior}), this is done by first controlling $\|\nabla u\|_{L^{2+\gamma}}$ with  $\|\nabla u\|_{L^{2}}$, and then $\|\nabla u\|_{L^{2}}$ with $\|u\|_{L^{1}}$.

We begin by introducing the notion of a small deformation of a half-ball.
It will be useful in several proofs, particularly in that of Lemma \ref{lem:x Du def}.
Given $\rho>0$, we denote by $B_\rho^+$ the upper half-ball in the $e_n$ direction, namely
$$
B_\rho^+:=B_\rho\cap\{x_n>0\}.
$$

\begin{definition}
\label{def:deform}
Given $\vartheta\geq0$, we say that $\Omega\subset \R^n$ is a $\vartheta$-deformation of $B_2^+$ if 
\[
\Omega = \Phi(B_2\cap \{x_n>0\})  
\]   
for some $\Phi \in C^3(B_2; \R^n)$ satisfying $\Phi(0)=0$, $D\Phi(0) = {\rm Id}$, and \
\[
\|D^2\Phi\|_{L^\infty (B_2)} + \|D^3\Phi\|_{L^\infty (B_2)}   \le \vartheta.
\]
Here, the norms of $D^2\Phi$ and $D^3\Phi$ are computed with respect to the operator norm.
\end{definition}

Note that, given a bounded $C^3$ domain, 
one can cover its boundary with finitely many small balls so that, after rescaling these balls, the boundary of the domain is given by a  finite union of  
$\vartheta$-deformations of $B_2^+$ (up to isometries) {with $\vartheta$ arbitrarily small.}

\begin{proposition}\label{prop:W12+epbdry}
Let $\Omega\subset \R^n$  be a $\vartheta$-deformation of $B_2^+$  for $\vartheta \in[0,\frac{1}{100}]$. Let $u\in C^2(\overline \Omega\cap B_1)$ be a nonnegative stable solution of $-\Delta u=f(u)$ in $\Omega\cap B_1$, with $u=0$ on $\partial \Omega\cap B_1$.
Assume that $f$ is locally Lipschitz, nonnegative, and nondecreasing. 
Then 
\[
 \|\nabla u\|_{L^{2+\gamma}(\Omega\cap B_{3/4})}  \le C \|\nabla u\|_{L^{2}(\Omega \cap B_{1})}, 
\]
where $\gamma>0$ and $C$ are dimensional constants.
\end{proposition}

The proof will make us of  the following lemma, which is based on a Pohozaev-type identity.

\begin{lemma}\label{lem:auxbdry2}
Under the assumptions of Proposition \ref{prop:W12+epbdry} we have 
\begin{equation} 
\| u_\nu\|_{L^2(\partial\Omega \cap B_{7/8})}  \leq C\|\nabla u\|_{L^2(\Omega\cap B_1)},
\end{equation}
where $u_\nu$ is the normal derivative of $u$ at $\partial\Omega$ and $C$ is a dimensional constant.
\end{lemma}

\begin{proof}
Take a cut-off function  $\eta\in C^2_c(B_1)$ such that $\eta=1$ in $B_{7/8}$,
and consider the vector-field $\mathbf{X}(x):= x+\boldsymbol e_n$. Multiplying  the identity 
\[ 
{\rm div}\big( | \nabla u|^2\mathbf{X}- 2(\mathbf{X}\cdot \nabla u) \nabla u \big) =  (n-2) |\nabla u|^2  - 2(\mathbf{X}\cdot \nabla u) \Delta u
\]
by $\eta^2$
and integrating in $\Omega\cap B_1$,
since $u=0$ on $\partial \Omega$ and $u\geq0$ in $\Omega\cap B_1$ (and hence the exterior unit normal $\nu$ is given by $-\frac{\nabla u}{|\nabla u|}$), we obtain
\[\begin{split}
-\int_{\partial\Omega \cap B_{1}}  (\mathbf{X}\cdot \nu) |\nabla u|^2 \eta^2\,d\mathcal H^{n-1}  - &\int_{\Omega \cap B_{1}} \big( | \nabla u|^2\mathbf{X}- 2(\mathbf{X}\cdot \nabla u) \nabla u \big)\cdot \nabla\eta^2 \,dx= 
\\
&\qquad=\int_{\Omega \cap B_{1}} \big((n-2) |\nabla u|^2  - 2(\mathbf{X}\cdot \nabla u) \Delta u\big) \eta^2 \,dx.
\end{split}\]
Note that, since $\Omega$ is a small deformation of $B_2^+$, we have  $-\mathbf{X}\cdot \nu\ge \frac{1}{2}$ on $\partial\Omega\cap B_1$.
Hence, since $F(t) : = \int_0^t f(s)ds$ satisfies $\mathbf{X}\cdot \nabla (F(u))= f(u)\mathbf{X}\cdot \nabla u  = - \Delta u\,  \mathbf{X}\cdot \nabla u$, we obtain
\[\begin{split}
\frac{1}{2}\int_{\partial\Omega \cap B_{1}}   |\nabla u|^2 \eta^2 d\mathcal H^{n-1}  &
\le C\int_{\Omega \cap B_{1}} |\nabla u|^2\,dx + 2\int_{\Omega \cap B_{1}}  \mathbf{X}\cdot \nabla  (F(u)) \eta^2\,dx
\\
&=  C\int_{\Omega \cap B_{1}} |\nabla u|^2\eta^2\,dx - 2 \int_{\Omega \cap B_{1}}  F(u) \,{\rm div}(\eta^2\mathbf{X})\,dx.
\end{split}
\]
We now observe that, since $f$ is nondecreasing, $0\leq F(t) \leq f(t)t$ for all $t \geq 0$. Hence, noticing that the function $g:=|{\rm div}(\eta^2\mathbf{X})|$ is Lipschitz,
we can bound
\begin{align*}
- \int_{\Omega \cap B_{1}}  F(u) \,{\rm div}(\eta^2\mathbf{X})\,dx &\le \int_{\Omega\cap B_1} u\,f(u)\,g\,dx
 = -\int_{\Omega\cap B_1} u\,\Delta u\,g\,dx\\
 & = \int_{\Omega\cap B_1}\big(|\nabla u|^2g + u\,\nabla u\cdot \nabla g\big)\,dx \leq C\int_{\Omega \cap B_1} \big(u^2+|\nabla u|^2\big)\,dx,
\end{align*}
and we conclude using Poincar\'e inequality (since $u$ vanishes on $\partial\Omega \cap B_1$).
\end{proof}

We next give the:

\begin{proof}[Proof of Proposition \ref{prop:W12+epbdry}]
The key idea is to use a variant of 
\[\xi=\bigl(|\nabla u|-u_n\bigr)\eta\]
as test function in the stability inequality (note that this function vanishes on the boundary if $\partial\Omega \cap B_1\subset\{x_n=0\}$ is flat).\\

\noindent
{\bf Step 1:} {\it We prove that, whenever $B_{\rho}(z)\subset B_{7/8}$,
\begin{equation}\label{hwioghwoih} 
 \int_{\Omega\cap B_{\rho/2}(z)} \rho^4\J^2\,dx \le C\int_{\Omega\cap B_{\rho}(z)}  \left(\rho^3 |D^2 u| |\nabla u| + \rho^2|\nabla u|^2\right)\,dx,
\end{equation}
where $\J$ is as in Lemma \ref{conseqestab}.}

By scaling and a covering argument, it is enough to prove the result for $z=0$ and $\rho =1$.\footnote{{For this, note that when  $B_{\rho}(z)\subset \Omega$ then \eqref{hwioghwoih}  follows from Lemma \ref{conseqestab}. Note also that if $z\in \partial\Omega\cap B_{7/8}$ then, within a small ball centered at $z$, $\Omega$ is (after a translation, rotation, and dilation) a $\vartheta$-deformation of $B_2^+$.}}
Observe that, thanks to Lemma \ref{lem:eq Du}(iii), $\nabla u\in (W^{2,p}\cap C^1)(\overline\Omega \cap {B_{7/8}})$ 
for all $p\in(1,\infty)$.

Since $\Omega$ is a $\vartheta$-deformation of $B_2^+$ with $\vartheta\leq 1/100$, $\Phi$ is a diffeomorphism. 
Let 
\[\mathbf{Y}:= \nabla(\boldsymbol e_n\cdot\Phi^{-1})= \nabla((\Phi)^{-1})^n \]
be the gradient of the pushforward of the $n$-coordinate  $x_n: B_1^+ \to \R$ through $\Phi$.
Note that $\mathbf{Y}$ is orthogonal to $\partial \Omega$.
We define $\NN=\mathbf{Y}/|\mathbf{Y}|$, and note that $\NN$ belongs to $C^2(\overline \Omega)$ and that $\NN=-\nu$ on $\partial \Omega$.

Consider the following convex $C^{1,1}$ regularization of the absolute value:  for $r>0$ small, we set
\begin{equation}\label{phi-r}
\phi_r(z) :=  |z|1_{\{|z|>r\}} + \Big(\frac{r}{2} +\frac{|z|^2}{2r}\Big) 1_{\{|z|<r\}}.
\end{equation}
Then $\phi_r(\nabla u)\in (W^{2,p}\cap C^1)(\overline\Omega \cap{B_{7/8}})$ for all $p<\infty$.
Moreover, since $u$ is nonnegative and superharmonic,
unless $u\equiv 0$ (in which case there is nothing to prove) then it follows by the Hopf lemma that
$|\nabla u|\ge c>0$ on $\partial{\Omega}\cap B_{7/8}$, for some constant $c$. 
Hence, since $\nabla u$ is $C^1$ up to the boundary, for $r>0$ small enough we have
\begin{equation}\label{inaneighborhood}
\phi_r(\nabla u)  = |\nabla u| \quad \mbox{in a neighborhood of }\partial \Omega\mbox{ inside }B_{7/8}.
\end{equation}
After choosing $r>0$ small enough such that \eqref{inaneighborhood} holds, we set 
\[\textbf{c} := \phi_r(\nabla u) - \NN\cdot \nabla u,\]
 and we take $\eta\in C^2_c (B_{7/8})$ with $\eta=1$ in $B_{1/2}$. 
Note that $\textbf{c}\equiv 0$ on $\partial\Omega\cap B_{7/8}$, and $\textbf{c}\in (W^{2,p}\cap C^1)(\overline\Omega\cap B_{7/8})$.  
Then, {since $\textbf{c}$ vanishes on $\partial\Omega\cap B_{7/8}$, thanks to an approximation argument we are allowed to take $\xi = \textbf{c}\eta$ as a test function in the stability inequality \eqref{stabilityLip}. Thus, with this choice,} integration by parts yields
\begin{equation}\label{ajgpen1}
\int_{\Omega\cap B_1} \big(\Delta \textbf{c} +f_-'(u)\textbf{c}\big)\,\textbf{c}\, \eta^2 \,dx\le \int_{\Omega\cap B_1} \textbf{c}^2 |\nabla \eta|^2 \,dx.
\end{equation}

Note now that
\begin{equation}\label{ajgpen2}
\begin{split}
\bigl(\Delta \textbf{c} + f_-'(u)\textbf{c}\bigr) \,\textbf{c} &=  \bigl({\Delta[\phi_r(\nabla u)]} +f_-'(u)\phi_r(\nabla u)\bigr) \phi_r(\nabla u)
\\
& \hspace{25mm}  - \big(\Delta (\NN\cdot \nabla u) +f_-'(u) \NN\cdot \nabla u \big) \big( \phi_r(\nabla u)- \NN\cdot \nabla u\big)
\\
& \hspace{25mm}  - \big( {\Delta[\phi_r(\nabla u)]} +f_-'(u) \phi_r(\nabla u)\big) \NN\cdot \nabla u.
\end{split}
\end{equation}

Since $\Delta \nabla  u  = -f_-'(u) \nabla u$ (see Lemma \ref{lem:eq Du}(ii)), we have
\begin{eqnarray}\label{may19}
\bigl({\Delta[\phi_r(\nabla u)]} +f_-'(u)\phi_r(\nabla u)\bigr) \phi_r(\nabla u) 
 &=  f'_-(u)\phi_r(\nabla u)\Bigl(\phi_r(\nabla u)-\sum_j u_j(\partial_j\phi_r)(\nabla u)\Bigr) \\
 & +\,\phi_r(\nabla u)\sum_{i,j,k} (\partial_{jk}^2\phi_r)(\nabla u)u_{ij}u_{ik}.\label{may19b}
\end{eqnarray}
Note that, inside the set $\{|\nabla u|\leq r\}$, the term \eqref{may19b} is nonnegative since $\phi_r$ is convex, while the term \eqref{may19} is equal to 
\[f'_-(u)\phi_r(\nabla u)\Bigl(\frac{r}{2}-\frac{|\nabla u|^2}{2r}\Bigr)\]
and therefore it is also nonnegative (all three factors are nonnegative).
On the other hand, inside the set $\{|\nabla u|>r\}$, the term \eqref{may19} vanishes, while the term \eqref{may19b} equals $\J^2$.
Therefore, we conclude that
\begin{equation}\label{ajgpen3}
\bigl({\Delta[\phi_r(\nabla u)]} +f_-'(u)\phi_r(\nabla u)\bigr) \phi_r(\nabla u)  \ge  \J^2 \, 1_{\{|\nabla u|>r\}},
\end{equation}
where $\J^2$ is as in \eqref{defAAA}.

{Coming back to \eqref{ajgpen2}, we note that
\begin{equation}\label{agniowh}
\begin{split}
\Delta (\NN\cdot \nabla u) +f'_-(u)\NN\cdot \nabla u =   \sum_i \Delta \NN^i u_i   + 2\sum_{ij}  \NN^i_j u_{ij},
\end{split}
\end{equation}
so it follows from the bound $|\phi_r(\nabla u)|\leq |\nabla u|+r$ that 
 \begin{multline}
 \label{eq:du r}
\left| \int_{\Omega \cap B_1}  \big(\Delta (\NN\cdot \nabla u) +f_-'(u) \NN\cdot \nabla u \big) \big( \phi_r(\nabla u)- \NN\cdot \nabla u\big) \eta^2\,dx\right|\leq \\
\le   C\int_{\Omega\cap B_1} (|\nabla u|+r)\left( |D^2u|  + |\nabla u|\right)\,dx.
 \end{multline}}

Also, since  $\eta\in C^2_c (B_{7/8})$, {integrating by parts} and recalling  \eqref{inaneighborhood} we have
\begin{equation}\label{ahoighwioh}
\begin{split}
\int_{\Omega\cap B_1}  {\Delta[\phi_r(\nabla u)]}\,  \NN\cdot \nabla u \,\eta^2\,dx &=  \int_{\Omega\cap B_1}    \phi_r(\nabla u) \,  \Delta (\NN\cdot \nabla u)  \,\eta^2\,dx \ + 
\\
&\hspace{-15mm} +\int_{\Omega\cap B_1}  \left( 2 \phi_r(\nabla u) \, \nabla (\NN\cdot \nabla u)\cdot  \nabla (\eta^2) + |\nabla u| \, \NN\cdot \nabla u \,\Delta (\eta^2)\right)\,dx
\\
&\hspace{-15mm} +\int_{\partial \Omega\cap B_1}  \left( |\nabla u|_\nu  \,\NN\cdot \nabla u\, \eta^2 - |\nabla u|  (\NN\cdot \nabla u\, \eta^2)_\nu\right) d\mathcal H^{n-1}.
\end{split}
\end{equation}
Since $\NN =  \frac{\nabla u}{|\nabla u|} = -\nu$ on the boundary, it follows that on $\partial\Omega\cap B_1$ it holds 
 \[|\nabla u|_\nu  \NN\cdot \nabla u  =  - \frac{\sum_{ij }u_{ij}u_j u_j}{|\nabla u|}
 \qquad \text{and}\qquad 
 |\nabla u|  (\NN\cdot \nabla u)_\nu  = -\sum_{i,j}\NN^i_j  u_i u_j  - \frac{\sum_{ij }u_{ij}u_j u_j}{|\nabla u|},\] 
and therefore, thanks to Lemma \ref{lem:auxbdry2},
\begin{equation}\label{may19c}\begin{split}
 \bigg|  \int_{\partial \Omega\cap B_1}  \left(  |\nabla u|_\nu  \,\NN\cdot \nabla u\, \eta^2 - |\nabla u|  (\NN\cdot \nabla u\, \eta^2)_\nu\right) \,d\mathcal H^{n-1}  \bigg| 
 & \le C\int_{\partial \Omega\cap B_{7/8}}  |u_\nu|^2\,d\mathcal H^{n-1} \\
 & \le C\int_{\Omega\cap B_1}|\nabla u|^2dx.
\end{split}\end{equation}

Thus, combining \eqref{ahoighwioh} and \eqref{agniowh}, and then using \eqref{may19c}, we conclude that
\[
\left| \int_{\Omega\cap B_1} \big( {\Delta[\phi_r(\nabla u)]} +f'_-(u) \phi_r(\nabla u)\big) \NN\cdot \nabla u\,  \eta^2\,dx\right| \le   C\int_{\Omega\cap B_1} (|\nabla u|+r)\left( |D^2u|  + |\nabla u|\right)\,dx.
\]
Combining this bound with \eqref{ajgpen1},  \eqref{ajgpen2}, \eqref{ajgpen3}, and \eqref{eq:du r}, we finally obtain 
\[
 \int_{\Omega\cap B_1} \J^2\eta^2 1_{\{|\nabla u|>r\}}\,dx \le 
 C\int_{\Omega\cap B_{1}}  (|\nabla u|+r)^2+(|\nabla u|+r)\left( |D^2u|  + |\nabla u|\right)\,dx.
\]
Recalling that $\eta=1$ in $B_{1/2}$, letting $r\downarrow 0$ this proves \eqref{hwioghwoih} for $z=0$ and $\rho=1$, as desired.
\\

\noindent
{\bf Step 2: } {\it We prove that} 
\begin{equation}\label{wgiowhoiwh}
\| \J \|^2_{L^2(\Omega\cap B_{{7/8}})}  \le  C  \|\nabla u\|_{L^2(\Omega\cap B_1)}^2.
\end{equation}

It suffices to prove that, for every  $B_{\rho}(z)\subset B_{7/8}$ and $\ep>0$, we have
\begin{equation}\label{angoiwnown}
\rho^2\| \J \|^2_{L^2(\Omega\cap B_{{2\rho/5}}(z))}  \le  \ep \rho^2\| \J \|^2_{L^2(\Omega\cap B_{\rho}(z))} +
\frac{C}{\ep}  \|\nabla u\|_{L^2(\Omega\cap B_\rho(z))}^2 
\end{equation}
where $\J$ is  as in \eqref{defAAA}. 
Indeed, it follows from Lemma \ref{lem_abstract} applied with $\sigma(B):=\| \J \|^2_{L^2(\Omega\cap B)}$ that
\eqref{angoiwnown} leads to \eqref{wgiowhoiwh} with $\Omega\cap B_{{7/8}}$ replaced by $\Omega\cap B_{7/16}$.
A covering and scaling argument then gives \eqref{wgiowhoiwh} with $\Omega\cap B_{{7/8}}$ in the left hand side.

To prove \eqref{angoiwnown}, we argue as at the beginning of Step 1 to note that we may assume $z=0$ and $\rho=1$.

We observe that, for any given  $\eta \in C^2_c(B_{7/8})$ with $\eta\equiv1$ in $B_{{4/5}}$, 
it follows from \eqref{hwuighwiu} and 
\eqref{hwuighwiu2}
that
\begin{equation}\label{ahigohwiowb}
-\int_{\Omega\cap B_1} {\rm div}\big(  |\nabla u| \nabla u\big) \eta^2\,dx  \ge  \int_{\Omega\cap B_1}  \big( -2 \Delta u   - C\J\big)|\nabla u|\,\eta^2\,dx.
\end{equation}
Hence, since $|D^2u|\leq |\Delta u|+C\J$ and $\Delta u \leq 0$, using \eqref{ahigohwiowb} we get
\begin{equation}\label{ahigohwiowb2}
\int_{\Omega\cap B_1}|D^2u|\,|\nabla u|\,\eta^2\,dx \leq 
\biggl|\frac12\int_{\Omega\cap B_1} {\rm div}\big(  |\nabla u| \nabla u\big) \eta^2\,dx  \biggr|+C  \int_{\Omega\cap B_1} \J\,|\nabla u|\,\eta^2\,dx.
\end{equation}
On the other hand, 
using Lemma \ref{lem:auxbdry2} we obtain
\begin{equation}\label{ahigohwiowb3}
\begin{split}
\biggl|\int_{\Omega\cap B_1} {\rm div}\big(  |\nabla u| \nabla u\big) \eta^2\,dx\biggr|  &=
\biggl|-\int_{\partial \Omega\cap B_1}   |u_\nu|^2 \eta^2 \,d\mathcal H^{n-1}   - \int_{\Omega}  |\nabla u| \nabla u \cdot \nabla (\eta^2)\,dx\biggr|\\
&\leq C \int_{\Omega\cap B_1}   |\nabla u|^2\,dx.
\end{split}
\end{equation}
Thus, combining \eqref{ahigohwiowb2}
and \eqref{ahigohwiowb3},
we get
\begin{equation}
\label{eq:Du D2u}
\int_{\Omega\cap B_1} |D^2u| \, |\nabla u|\,\eta^2\,dx \le C\int_{\Omega\cap B_1} \J \,|\nabla u|\,\eta^2\,dx + C\int_{\Omega\cap B_1} |\nabla u|^2\,dx.
\end{equation}
{Recalling that $\eta\equiv1$ in $B_{{4/5}}$,
\eqref{eq:Du D2u} and \eqref{hwioghwoih} yield, for every $\ep\in (0,1)$,}
\[\begin{split}
 \int_{\Omega\cap B_{{2/5}}} \J^2\,dx 
&\le C\|\nabla u\|_{L^2(\Omega\cap B_{4/5})}^2 + C \int_{\Omega\cap B_{{4/5}}} |D^2u| \, |\nabla u|\,dx
\\
&\le C\|\nabla u\|_{L^2(\Omega\cap B_1)}^2 + C \int_{\Omega\cap B_1} \J\, |\nabla u|\,dx 
\le \frac{C}{\ep}\|\nabla u\|_{L^2(\Omega\cap B_1)}^2 + \ep  \int_{\Omega\cap B_1} \J^2\,dx,
\end{split}\]
which proves \eqref{angoiwnown}.
\\

\noindent
{\bf Step 3: } {\it We show  that} 
\[
\int_{\Omega\cap B_{{4/5}}} \big| {\rm div}(|\nabla u|\, \nabla u) \big| \,dx  \le C\int_{\Omega\cap B_{1}} |\nabla u|^2\,dx.
\]

{As in the previous step, we take  $\eta \in C^2_c(B_{7/8})$ with $\eta\equiv1$ in $B_{{4/5}}$.}
Then it suffices to combine \eqref{hwuighwiu}, \eqref{hwuighwiu2}, \eqref{eq:Du D2u},
and \eqref{wgiowhoiwh}, to get
\begin{multline*}
\int_{\Omega\cap B_{{4/5}}} \big| {\rm div}(|\nabla u|\, \nabla u) \big|  \,dx \le \int_{\Omega\cap B_{{4/5}}} -2|\nabla u|\,\Delta u\,dx  + C \int_{\Omega\cap B_{{4/5}}} |\nabla u| \,\J\,dx \\
\leq C\int_{\Omega\cap B_{{7/8}}} \J^2\,dx+C\int_{\Omega\cap B_{1}} |\nabla u|^2\,dx+ C\biggl(\int_{\Omega\cap B_{{4/5}}} |\nabla u|^2\,dx\biggr)^{1/2} \biggl(\int_{\Omega\cap B_{{4/5}}} \J^2\,dx\biggr)^{1/2}\\
 \le C\int_{\Omega\cap B_{1}} |\nabla u|^2\,dx,
\end{multline*}
as desired.\\

\noindent
{\bf Step 4: } {\it Conclusion.}

Here it is convenient to assume, after multiplying $u$ by a constant, that $\|\nabla u\|_{L^2(\Omega\cap B_1)}=1$.

Thanks to Step 3, we can repeat the same argument as the one used in Step 2  in the proof of  Proposition \ref{prop:W12+ep} to deduce that, for a.e. $t>0$, 
\begin{equation}\label{ahjsioghoibg}
\int_{\Omega\cap \{u=t\}\cap B_{3/4}} |\nabla u|^2 \,d\HH^{n-1}  \le  C\int_{\Omega\cap B_1} |\nabla u|^2\,dx=C.
\end{equation}
Also, since $u$ vanishes on $\partial\Omega\cap B_1$, setting $h(t)=\max\{1,t\}$, by the Sobolev embedding we deduce that
\begin{equation}\label{ashgowobb}
\int_{\R^+} dt \int_{\Omega\cap \{u=t\}\cap B_{1}\cap \{|\nabla u|\neq0\}}  h(t)^p \,|\nabla u|^{-1} \, d\HH^{n-1} \leq |\Omega\cap B_1\cap \{u<1\}|+ \int_{\Omega\cap B_1} u^p  \, dx\leq C,
\end{equation}
for some $p>2$.
Hence, choosing dimensional constants  $q>1$ and $\theta\in (0,1/3)$ such that $p/q = (1-\theta)/\theta$, we can write
\begin{multline}\label{nwgoinwon}
\int_{\Omega\cap B_{3/4}} |\nabla u|^{3-3\theta}\,dx  =\int_{\R^+} dt\int_{\Omega\cap \{u=t\}\cap B_{3/4}\cap \{|\nabla u|\neq0\}}  h(t)^{p\theta -q(1-\theta)}  |\nabla u|^{-\theta + 2(1-\theta)} d\HH^{n-1} 
\\
\le   \left(\int_{\R^+} dt\int_{\Omega\cap \{u=t\}\cap B_1\cap \{|\nabla u|\neq0\}}   h(t)^{p}  |\nabla u|^{-1}d\HH^{n-1} \right)^\theta \times
\\\qquad\qquad \times \bigg(\int_{\R^+} dt\int_{\Omega\cap \{u=t\}\cap B_{3/4}\cap \{|\nabla u|\neq0\}}   h(t)^{-q} |\nabla u|^2 d\HH^{n-1}  \bigg)^{1-\theta},
\end{multline}
and by \eqref{ashgowobb} and the very same argument as the one used at the end of Step 3  in the Proof of  Proposition \ref{prop:W12+ep}  (now using \eqref{ahjsioghoibg})
 we obtain  
\[
\int_{\Omega\cap B_{3/4}} |\nabla u|^{3-3\theta}  \,dx \le C,
\]
which concludes the proof.
\end{proof}

\begin{remark}\label{asngosnaokn}
Note that, in Step 4 of the previous proof, one may also take any exponent $p>2$, and then $\theta =1/3$ and $q=p/2>1$. 
With these choices, if we normalize $u$ so that $\|u\|_{L^p(\Omega\cap B_1)}=1$ (instead of the normalization $\|\nabla u\|_{L^2(\Omega\cap B_1)}=1$ made in Step 4 of the previous proof), setting $h(t):=\max\{1,t\}$ 
it follows from  \eqref{nwgoinwon}, \eqref{ashgowobb},
and the inequality in \eqref{ahjsioghoibg}, that
\begin{equation}\label{eq:2 p}
\int_{\Omega\cap B_{3/4}} |\nabla u|^{2}\,dx \leq C\Big(\int_{\Omega\cap B_{1}} |\nabla u|^{2}\,dx\Big)^{2/3}\qquad \text{{whenever $\|u\|_{L^p(\Omega\cap B_1)}=1$}},
\end{equation}
where we used that $\int_{\R^+} h(t)^{-q}\,dt\leq C$.

In the general case, applying this estimate to $u/\|u\|_{L^p(\Omega\cap B_1)}$, we deduce that
\begin{equation}\label{eq:2 p 2}
\int_{\Omega\cap B_{3/4}} |\nabla u|^{2}  \,dx\le C \left(\int_{\Omega\cap B_1} |u|^p\,dx  \right)^{\frac{2}{3p}} \left(\int_{\Omega\cap B_1} |\nabla u|^{2}\,dx  \right)^{\frac{2}{3}}
\end{equation}
for every $p>2$.
\end{remark}

As a consequence of this remark, we deduce the following important a priori estimate.

\begin{proposition}\label{prop:L1controlsW12bdry}
Under the assumptions of Proposition \ref{prop:W12+epbdry}, there exists a dimensional constant $C$ such that  
\begin{equation} 
\| \nabla u\|_{L^{2}(\Omega \cap B_{1/2})}  \leq C\| u\|_{L^1(\Omega\cap B_1)}.
\end{equation}
\end{proposition}

\begin{proof}
By Remark \ref{asngosnaokn}, {we can choose  $p\in (2, 2^*)$ (here $2^*$ is the Sobolev exponent, or any number less than infinity if $n=2$) and then $\zeta\in(0,1)$
such that $p=\zeta 2^*+(1-\zeta)$,} to obtain 
\[
\begin{split}
\| \nabla u\|_{L^2(\Omega \cap B_{3/4})} & \le C \| u\|_{L^p(\Omega \cap B_{1})}^{1/3} \| \nabla u \|_{L^{2}(\Omega \cap B_{1})}^{2/3} \le C  \| u\|_{L^{2^*}(\Omega \cap B_{1})}^{\zeta/3} \|u\|_{L^{1}(\Omega \cap B_{1})}^{(1-\zeta)/3}\| \nabla u \|_{L^{2}(\Omega \cap B_{1})}^{2/3}
\\& \le 
C\| \nabla u \|_{L^{2}(\Omega \cap B_{1})}^{(2+\zeta)/3} \|u\|_{L^{1}(\Omega \cap B_{1})}^{(1-\zeta)/3}\leq \ep  \| \nabla u \|_{L^{2}(\Omega \cap B_{1})}
+\frac{C}{\ep}\|u\|_{L^{1}(\Omega \cap B_{1})}.
\end{split}
\]
Hence, applying this estimate to the functions $u_{r,y}(x):=u(y+rx)$ for all balls $B_r(y)\subset B_1$ (as in the proof of Proposition \ref{prop:L1controlsW12}), { we can use Lemma \ref{lem_abstract} with $\sigma(B)=\| \nabla u\|_{L^2(\Omega \cap B)}$ to conclude.
}
\end{proof}

\section{\for{toc}{Boundary $C^\alpha$ estimate for $n\leq 9$}\except{toc}{Boundary $C^\alpha$ estimate for $n\leq 9$, and proof of Theorem \ref{thm:globalCalpha}}}

\label{sect:global}

In order to prove Theorem \ref{thm:globalCalpha}, as observed at the beginning of Section \ref{sect:bdry W12g},
every bounded domain of class $C^3$ can be covered by finitely many balls so that, after rescaling the balls to have size $1$, inside each ball the boundary is a $\vartheta$-deformation of $B_2^+$ for some $\vartheta\leq \frac{1}{100}$. 
Hence, by applying Propositions \ref{prop:W12+epbdry} and \ref{prop:L1controlsW12bdry}, we deduce that there exists a neighborhood of $\partial\Omega$ in which the $W^{1,2+\gamma}$-norm of $u$ is controlled by $\|u\|_{L^1(\Omega)}$. Combining this information with  \eqref{eq:W12g L1 int} and a covering argument, we conclude the validity of 
 \eqref{eq:W12g L1 glob}. 
 Hence, we are left with proving \eqref{eq:C0a L1 glob}.
 
 By the same reasoning as the one we just did, but now using \eqref{eq:Ca L1 int} instead of  \eqref{eq:W12g L1 int}, to show \eqref{eq:C0a L1 glob} when $n\leq9$ it suffices to obtain a uniform $C^{\alpha}$ control near the boundary when $\partial\Omega$ is a small $\vartheta$-deformation of $B_2^+$ (recall Definition \ref{def:deform}). 
 Hence, to conclude the proof of Theorem \ref{thm:globalCalpha}, it suffices to show the following:

\begin{theorem}
\label{thm:uptoboundary}
Let $n\leq9$,  $\vartheta\in [0, \frac {1}{100}]$, and $\Omega\subset \R^n$  be a $\vartheta$-deformation of $B_2^+$. Assume that  $u\in C^0(\overline\Omega \cap B_1) \cap C^2(\Omega\cap B_1)$  is a nonnegative stable solution of
$$
-\Delta u=f(u) \quad \mbox{in}\ \Omega\cap B_1\qquad \mbox{and} \qquad u=0 \quad \mbox{on}\ \partial \Omega\cap B_1
$$
for some nonnegative, nondecreasing, convex function $f:\R\to \R$.
Then
$$
\|u\| _{C^\alpha(\overline\Omega\cap B_{1/2} )}  \leq C\|u\|_{L^1(\Omega\cap B_1)},
$$
where $\alpha>0$ and $C$  are dimensional constants.
\end{theorem}

To prove this theorem, we first need the boundary analogue of the key interior estimate \eqref{eq:firstest}.

\begin{lemma}
\label{lem:x Du def}
Let $\Omega\subset \R^n$  be a $\vartheta$-deformation of $B_2^+$  for $\vartheta \in[0,\frac{1}{100}]$, and
let $u\in C^2(\overline \Omega\cap B_1)$ be a nonnegative stable solution of $-\Delta u=f(u)$ in $\Omega\cap B_1$, with $u=0$ on $\partial \Omega\cap B_1$.
Assume that $f$ is locally Lipschitz. 

Then there exists a dimensional constant $C$ such that, for all $\eta \in C^{0,1}_c (B_1)$,
\begin{multline*}
\int_{\Omega\cap B_1} \Big(\big\{(n-2)\eta + 2 x\cdot\nabla \eta \}\,\eta\, |\nabla u|^2 - 2(x\cdot \nabla u) \nabla u\cdot \nabla(\eta^2) - |x\cdot \nabla u|^2 |\nabla \eta|^2 \Big)\,dx \\
\le C\vartheta\int_{\Omega\cap B_1}  |\nabla u|^2 
\bigl(\eta^2+|x|\,|\nabla(\eta^2)| +|x|^2|\nabla\eta|^2 \bigr)\,dx.
\end{multline*}
\end{lemma}

\begin{proof}
The key idea is to use a variant of $\xi=(x\cdot \nabla u)\eta$ as test function in the stability inequality (note that this function vanishes on the boundary if $\partial\Omega\cap B_1 =\{x_n=0\}\cap B_1$ is flat).

We consider the vector-field
$$
\mathbf{X}(x)=(D\Phi)(\Phi^{-1}(x))\cdot\Phi^{-1}(x)\qquad \mbox{for all }\,x \in \Omega \cap B_1,
$$
with $\Phi$ as in Definition \ref{def:deform}.
Note that $\mathbf{X}$ is tangential to $\partial\Omega$ since, for $x\in \partial\Omega\cap B_1$, $\Phi^{-1}(x)$ is tangent to the flat boundary of $B_1^+$.
Hence, since $u=0$ on $\partial \Omega\cap B_1$, we deduce that $\mathbf{X}\cdot \nabla u=0$ on $\partial\Omega\cap B_1$.
Also, since $\Omega$ is a $\vartheta$-deformation of $B_2^+$, it is easy to check that
\begin{equation}
\label{eq:X}
|\mathbf{X}-x|\leq C\vartheta|x|^2,\qquad |\nabla \mathbf{X}-{\rm Id}|\leq C\vartheta|x|,\qquad
|D^2 \mathbf{X}|\leq C\vartheta,
\end{equation}
where $C$ is a dimensional constant.
The bound on $D^2\mathbf{X}$ follows by a direct computation, while the two first ones follow by integrating the latter and using that $\nabla\mathbf{X}(0)=\textrm{Id}$ and $\mathbf{X}(0)=0$.

Set $\textbf{c} := \mathbf{X}\cdot \nabla u$,
and take $\eta \in C^{2}_c\big({B_1}\big)$.
Note that $\textbf{c}\equiv 0$ on $\partial\Omega\cap B_1$ and $\textbf{c}\in (W^{2,p}_{\rm loc}\cap C^1)(\overline\Omega\cap B_1)$ for all $p<\infty$ (thanks to Lemma \ref{lem:eq Du}).  
Hence, {arguing as usual by approximation, one is allowed to take} $\xi = \textbf{c}\eta$ as a test function in the stability inequality \eqref{stabilityLip}.
Thus, using that $\textbf{c}$ vanishes on $\partial\Omega\cap B_1$, integration by parts yields
\begin{equation}\label{ajgpen1 2}
\int_{\Omega\cap B_1} \big\{\Delta \textbf{c} +f_-'(u)\textbf{c}\big\}\,\textbf{c}\, \eta^2 \,dx\le \int_{\Omega\cap B_1} \textbf{c}^2 |\nabla \eta|^2 \,dx.
\end{equation}
By a direct computation it follows that
\begin{align*}
\Delta \textbf{c}&=\mathbf{X}\cdot \nabla \Delta u+2\nabla \mathbf{X}:D^2u+\Delta \mathbf{X}\cdot \nabla u\\
&=-f'_-(u)\,\mathbf{X}\cdot \nabla u+2(\nabla \mathbf{X})^s:D^2u+\Delta \mathbf{X}\cdot \nabla u\\
&=-f'_-(u)\,\textbf{c}+2{\rm div}\bigl((\nabla \mathbf{X})^s \nabla u\bigr)+\bigl[\Delta \mathbf{X}-2{\rm div}\bigl((\nabla\mathbf{X})^s\bigr)\bigr]\cdot \nabla u,
\end{align*}
where  $(\nabla \mathbf{X})^s:=\frac12 (\nabla \mathbf{X}+(\nabla \mathbf{X})^*)$ is the symmetrized version of $\nabla \mathbf{X}$
and 
we used that $\nabla \mathbf{X}:D^2u=(\nabla \mathbf{X})^s:D^2u$ (since $D^2u$ is a symmetric matrix).

Hence, substituting this identity in \eqref{ajgpen1 2} and using \eqref{eq:X} we get
\begin{equation}\label{ajgpen1 3}
\int_{\Omega\cap B_1} |\mathbf{X}\cdot \nabla u|^{2}\left|\nabla \eta\right|^{2} \,dx \geq 2\int_{\Omega \cap B_1} (\mathbf{X}\cdot \nabla u)\,{\rm div}\bigl((\nabla \mathbf{X})^s \nabla u\bigr)\,\eta^{2}\, dx
- C\vartheta \int_{\Omega \cap B_1} |\nabla u|^2\,\eta^2\, dx.
\end{equation}
Noticing that $|\nabla \mathbf{X}-{\rm Id}|+|(\nabla \mathbf{X})^s-{\rm Id}|+|{\rm div}\mathbf{X}-n|+|\nabla (\nabla \mathbf{X})^s|\leq C\vartheta $ ({as a consequence of \eqref{eq:X}}), we see that
\[\begin{split}
{\rm div}\bigl(2 &(\mathbf{X}\cdot \nabla u)\,[(\nabla \mathbf{X})^s \nabla u] -\big\{[(\nabla \mathbf{X})^s \nabla u]\cdot\nabla u\bigr\}\mathbf{X}\bigr) \\
&=2(\mathbf{X}\cdot \nabla u)\,{\rm div}\bigl((\nabla \mathbf{X})^s \nabla u\bigr)
+2[\nabla \mathbf{X} \nabla u]\cdot [(\nabla \mathbf{X})^s\nabla u]\\
&\qquad-{\rm div}\mathbf{X}\,\big\{[(\nabla \mathbf{X})^s \nabla u]\cdot\nabla u\bigr\}-
\big\{[\mathbf{X}\cdot \nabla (\nabla \mathbf{X})^s]\cdot \nabla u\bigr\}\cdot\nabla u\bigr\}\\
&= 2(\mathbf{X}\cdot \nabla u)\,{\rm div}\bigl((\nabla \mathbf{X})^s \nabla u\bigr)+(2-n)|\nabla u|^2 +O(\vartheta |\nabla u|^2).
\end{split}\]
Hence, using this identity in \eqref{ajgpen1 3}, and taking into account \eqref{eq:X} and that $\mathbf{X}\cdot\nabla u=0$ and $\mathbf{X}\cdot\nu=0$ on $\partial\Omega\cap B_1$, we get
\[\begin{split}
\int_{\Omega\cap B_1} &|x\cdot \nabla u|^{2}\left|\nabla \eta\right|^{2} \,dx+C\vartheta \int_{\Omega \cap B_1} |\nabla u|^2\,\eta^2\, dx+C\vartheta\int_{\Omega\cap B_1}|\nabla u|^2|x|^2|\nabla\eta|^2\,dx\\
& \geq \int_{\Omega \cap B_1}\Big({\rm div}\bigl(2 (\mathbf{X}\cdot \nabla u)\,[(\nabla \mathbf{X})^s \nabla u] -\big\{[(\nabla \mathbf{X})^s \nabla u]\cdot\nabla u\bigr\}\mathbf{X}\bigr)+(n-2)|\nabla u|^2\Bigr)\eta^2\,dx\\
&=\int_{\Omega\cap B_1}\Big(-2(\mathbf{X}\cdot \nabla u) [(\nabla \mathbf{X})^s \nabla u]\cdot \nabla (\eta^2) +\big\{[(\nabla \mathbf{X})^s \nabla u]\cdot\nabla u\bigr\}\mathbf{X}\cdot \nabla (\eta^2)\Bigr)\,dx\\
&\qquad +\int_{\Omega\cap B_1}(n-2)|\nabla u|^2\eta^2\,dx\\
&\geq \int_{B_1}\Big(-2(x\cdot \nabla u) \nabla u\cdot \nabla (\eta^2) +|
\nabla u|^2x\cdot \nabla (\eta^2)+(n-2)|\nabla u|^2\eta^2\Bigr)\,dx\\
&\qquad-C\vartheta \int_{\Omega \cap B_1} |\nabla u|^2\,|x|\,|\nabla (\eta^2)|\, dx.
\end{split}\]
This proves the result for $\eta \in C^2_c(B_1)$, and the general case follows by approximation.
\end{proof}

To prove Theorem \ref{thm:uptoboundary} we will use a blow-up argument that will rely on the following Liouville-type result in a half-space.
In the blown-up domains, the constant $\vartheta$ in Lemma \ref{lem:x Du def} will tend to zero.
Recall that the class $\mathcal S(U)$, for $U\subset\R^n$, was defined in \eqref{class-S}.
We use the notation $\R^n_+:=\R^n\cap \{x_n>0\}.$

\begin{proposition}\label{classification}
When $3\leq n\leq 9$, there exists  a dimensional constant $\alpha_n>0$ such that the following holds. 
Assume that $u: \R^n_+ \rightarrow \R$ belongs to  $W^{1,2}_{\rm loc }\big(\overline{\R^n_+}\big)\cap C^0_{\rm loc}(\R^n_+)$,  $u \in \mathcal S(\R^n_+)$, and $u=0$ on $\{x_n=0\}$ in the trace sense.
Suppose in addition that, for some $\alpha \in (0,\alpha_n)$ and $\gamma>0$, denoting $u_R(x):= u(Rx)$ we have
\begin{equation}\label{P4}
  \|\nabla u_R\|_{L^{2+\gamma}(B_{3/2}^+)} \le C_1\|\nabla u_R\|_{L^{2}(B_2^+)} \le C_2 R^{\alpha}\qquad \mbox{for all }\,R\geq 1
\end{equation}
with constants $C_1$ and $C_2$ independent of $R$,
and that $u$ satisfies
\begin{equation}\label{P6}
 \int_{\R^n_+} \Big(\big\{(n-2)\eta + 2 x\cdot\nabla \eta \}\,\eta\, |\nabla u|^2 - 2(x\cdot \nabla u) \nabla u\cdot \nabla(\eta^2) - |x\cdot \nabla u|^2 |\nabla \eta|^2 \Big)\,dx \le 0
\end{equation}
for all $\eta \in C^{0,1}_c \big( \overline{\R^n_+ }\big)$.
Then $u\equiv 0$.
\end{proposition}

\begin{proof}
Let us define,
for $\rho>0$, 
\[
\mathcal D(\rho) :=  \rho^{2-n}\int_{B_\rho^+} |\nabla u|^2\,dx\qquad \mbox{and} \qquad\mathcal  R(\rho) :=  \int_{B_\rho^+}  |x|^{-n} |x\cdot \nabla u|^2\,dx.
\]
We divide the proof in three steps. {As we shall see, for the validity of Step 1 the assumption $3\leq n \leq 9$ is crucial.}\\

\noindent
{\bf Step 1: } {\it We prove that, for all $\rho>0$,}
\begin{equation}\label{ghiowhwwoh}
\mathcal  R(\rho)\le C\rho^{2-n}\int_{B_{2\rho}^+ \setminus B_\rho^+} |\nabla u|^2\,dx
\end{equation}
{\it for some dimensional constant $C>0$.}

Let $\psi\in C^\infty_c (B_2)$ be some  radial decreasing nonnegative cut-off function with $\psi\equiv 1$ in $B_1$, and set $\psi_\rho(x) := \psi(x/\rho)$.
Then, as in the interior case, for $a<n$ and $\varepsilon\in(0,\rho)$ we use the Lipschitz function $\eta_\varepsilon(x):=\min\{|x|^{-a/2},\varepsilon^{-a/2}\}\psi_\rho(x)$ as a test function in \eqref{P6}.
Hence, noting that $\nabla \psi_\rho$  has size $C/\rho$ and vanishes outside of the annulus $B_{2\rho}\setminus B_\rho$, and throwing away the term $\int_{\R^n_+\cap B_\varepsilon}(n-2)\eta_\varepsilon^2|\nabla u|^2dx$, we obtain
$$
 \int_{\R^n_+\setminus B_\varepsilon}   \Big\{(n-2-a)    |\nabla u|^2 + \Big( 2a -\frac{a^2}{4} \Big) |x\cdot \nabla u|^2 |x|^{-2}   \Big\}  |x|^{-a} \psi_\rho^2\,dx
 \le C(n, a) \rho^{-a}\int_{B_{2\rho}^+\setminus B_\rho} |\nabla u|^2\,dx.
$$
Choosing $a : = n-2$, since   $ 2a -\frac{a^2}{4}   = (n-2)\bigl(2-\frac{n-2}{4}\bigr)=\frac14(n-2)(10-n) >0$ for $3\le n\le 9$ we obtain
\[  
\int_{\R^n_+\setminus B_\varepsilon} |x|^{-n}  |x\cdot \nabla u|^2 \, \psi_\rho^2\,dx \le C \rho^{2-n}\int_{B_{2\rho}^+\setminus B_\rho} |\nabla u|^2\,dx.
\]
Recalling that  $\psi_\rho^2\equiv 1$ in $B_\rho$,
the claim follows by letting $\varepsilon\downarrow0$.\\

\noindent
{\bf Step 2:} {\it We prove that there exists a {dimensional} constant $C$ such that, if for some $R \ge 1$ we have
\[
\int_{B_1^+}  |\nabla u_R |^2\,dx \ge \frac 1 2 \int_{B_2^+}  |\nabla u_R |^2 \,dx,
\]
then
\[
\int_{B_{3/2}^+}  |\nabla u_R |^2\,dx \le   C \int_{B_{3/2}^+\setminus B_1^+} 
|x|^{-n}  |x\cdot \nabla u_R|^2  \,dx.
\]
}

The proof is by compactness. We assume by contradiction that 
we have a sequence $u_k:=u_{R_k}/\|\nabla u_{R_k}\|_{L^2(B_{3/2}^+)} \in \mathcal S(B_2^+)\cap W^{1,2}_{\rm loc}(\overline{\R^n_+})$, with $u_k = 0$ on $\{x_n=0\}$, satisfying 
\begin{equation}\label{agipwhioghwio0}
\int_{B_1^+}  |\nabla u_k |^2\,dx \ge \frac 1 2 \int_{B_2^+}  |\nabla u_k |^2\,dx,
\end{equation}
\begin{equation}\label{agipwhioghwio}
\int_{B_{3/2}^+}  |\nabla u_k |^2\,dx=1,\qquad \textrm{and}\qquad \int_{B_{3/2}^+\setminus B_1^+} |x|^{-n}  |x\cdot \nabla u_k|^2 \,dx \to 0.
\end{equation}
Note that, since $\int_{B_2^+}  |\nabla u_k |^2\,dx\leq2$, thanks to Lemma \ref{strongconvergence} and our interior $W^{1,2+\gamma}$ estimate
there exists a function $u$ such that, up to a subsequence, $u_k\to u$ strongly in $W^{1,2}_{\rm loc}(B_2^+).$
On the other hand, using the first bound in \eqref{P4},
for every $\delta \in (0,1)$ we have
\[\begin{split}
\int_{B_{3/2}^+\cap \{x_n \leq \delta\}} |\nabla u_k|^2\,dx&\leq \Big(\int_{B_{3/2}^+\cap \{x_n \leq \delta\}} |\nabla u_k|^{2+\gamma}\,dx\Big)^{2/(2+\gamma)} \bigl|B_{3/2}^+\cap \{x_n \leq \delta\}\bigr|^{\gamma/(2+\gamma)}\\
& \leq C\delta^{\gamma/(2+\gamma)}.
\end{split}\]
This means that the mass of $|\nabla u_k|^2$ near the boundary can be made arbitrarily small by choosing $\delta$ small enough.
Combining this information with the convergence of $u_k\to u$ in $W^{1,2}_{\rm loc}(B_2^+)$ we deduce that
$u_k\to u$ strongly in $W^{1,2}(B_{3/2}^+)$. 
Moreover, by  Theorem \ref{thm:stability} we obtain that $u \in \mathcal S(B_{3/2}^+)$, and taking the limit in \eqref{agipwhioghwio} we obtain 
\[
\int_{B_{3/2}^+}  |\nabla u|^2\,dx =1  \qquad \mbox{and} \qquad x\cdot\nabla u \equiv 0\quad \mbox{ in } B_{3/2}^+\setminus B_1^+. 
\]
Moreover, since the trace operator is continuous in $W^{1,2}(B_{3/2}^+)$, we deduce that $u=0$ on  $\{x_n=0\}\cap B_{3/2}$.

Hence, we have found a function $u\in \mathcal S(B_{3/2}^+)$ which  is $0$-homogeneous
 in the half annulus $B_{3/2}^+\setminus\overline{B_1^+}$.
 In particular, since $u$ is a weak solution of $-\Delta u = f(u)$ in $B_{3/2}^+$ with $-\Delta u=f(u)\in L^1_{\rm loc}\cap C^0(B_{3/2}^+\setminus \overline{B_1^+})$, this is only possible if $f\equiv 0$ (this follows from the fact that $\Delta u$ is $(-2)$-homogeneous while $f(u)$ is $0$-homogeneous). 
 It follows that $u$ is a $0$-homogeneous harmonic function in the half annulus $B_{3/2}^+ \setminus \overline{B_1^+}$ vanishing on $\partial(B_{3/2}^+\setminus \overline{B_1^+})\cap\{x_n=0\}$.
Hence, as in the proof of Lemma \ref{lem:22}, the {supremum and infimum of $u$ are attained at interior points, and thus $u$ must be zero by the strong maximum principle}. 
Furthermore, exactly as in the proof of Lemma \ref{lem:22}, the superharmonicity of $u$ combined with the fact that $u$  vanishes in $B_{3/2}^+\setminus B_1^+$ gives that 
$u$ vanishes in $B_{3/2}^+$.
This contradicts the fact that $\int_{B_{3/2}^+}  |\nabla u|^2\,dx =1$ and concludes the proof.\\

\noindent 
 {\bf Step 3:} {\it Conclusion}.
 
 Exactly as in Step 1 of the proof of Theorem \ref{thm:L1-CalphaW12},
 using Lemma \ref{lem:33} (combined with Steps 1 and 2 above) we deduce that
 \[
 \|x\cdot \nabla u_{r R} \|_{L^2(B_1^+)} \leq Cr^{\alpha_n} \|\nabla u_R\|_{L^2(B_1^+)}\qquad \text{for all $r\in (0,1/2)$ and $R\ge 1$}, 
 \]
 where $C$ and $\alpha_n>0$ are dimensional constants. 
 Hence, since by assumption $\|\nabla u_R\|_{L^2(B_1^+)}\le C R^\alpha$ with $\alpha<\alpha_n$, given a constant $M>0$, we choose $r = M/R$ and let $R\to \infty$ to find 
  \[
 \|x\cdot \nabla u_{M} \|_{L^2(B_1^+)} =0. 
 \]
 Since $u_M\in \mathcal S(B_1^+)$ and $u_M=0$ on $\{x_n=0\}\cap \overline{B_1^+}$, as in the previous Step 2 we conclude that $u_M\equiv 0$.
 Since $M>0$ is arbitrary, the proof is finished.
\end{proof}

We can now prove Theorem  \ref{thm:uptoboundary}.

\begin{proof}[Proof of Theorem  \ref{thm:uptoboundary}]
Note that, as in the interior case, we may assume that $3\leq n\leq 9$ by adding superfluous variables and considering a ``cylinder'' with base $\Omega$.
Also, by Lemma \ref{lem:eq Du}, $u\in C^2(\overline\Omega\cap B_1)$.

Recalling that $\Omega$ is a $\vartheta$-deformation of $B_2^+$ with $\vartheta\in[0,\frac{1}{100}]$, it suffices to prove that there exists a dimensional constant $C$ such that
\begin{equation}
\label{eq:dist a}
r^{2-n}\int_{\Omega\cap B_r}|\nabla u|^2\,dx \leq C\, r^{\alpha_n}\|\nabla u\|_{L^2(\Omega\cap B_1)}^2\qquad \mbox{for all }\, r\in (0,1),
\end{equation}
where $\alpha_n$ is given by Proposition \ref{classification}.
Indeed, given $r\in(0,\frac14)$ there exists a dimensional constant $c\in(0,1)$ such that $B_{cr}(re_n)\subset\Omega$, and the $L^\infty$ estimate from  \eqref{eq:Ca L1 int} applied in this ball, together with the inclusion $B_{cr}(re_n)\subset B_{2r}$, give
\[u(re_n)\leq Cr^{-n} \int_{B_{cr}(re_n)}u\,dx\leq Cr^{-n} \int_{\Omega\cap B_{2r}}u\,dx.\]
Thus, once \eqref{eq:dist a} is proven, it follows from this, the Sobolev inequality, and
 Proposition~\ref{prop:L1controlsW12bdry}, that
\[\begin{split}
u(re_n) &\leq C
r^{-n}\int_{\Omega\cap B_{2r}}u\,dx \leq C \Big( r^{2-n}\int_{\Omega\cap B_{2r}}|\nabla u|^2\,dx\Big)^{1/2} \\
&\leq C\, r^{\alpha_n/2}\|\nabla u\|_{L^2(\Omega\cap B_{1/2})}
\leq C\, r^{\alpha_n/2}\|u\|_{L^1(\Omega\cap B_{1})}
\end{split}\]
for all $r \in (0,1/4)$.
Applying this estimate to the functions $u_y(z):=u(y+z)$ with $y \in \partial\Omega\cap B_{1/2}$,
we deduce that
$$
u(x)\leq C\,{\rm dist}(x,\partial \Omega)^{\alpha_n/2}\|u\|_{L^1(\Omega\cap B_{1})}\qquad \mbox{for all }\,x \in \Omega\cap B_{1/2}.
$$
Combining this growth control with \eqref{eq:Ca L1 int} it follows by a standard argument that 
$$
\|u\| _{C^\beta(\overline\Omega\cap B_{1/2} )}  \leq C\|u\|_{L^1(\Omega\cap B_1)},
$$
where $\beta:=\min\{\alpha_n/2,\alpha\}$, with $\alpha$ as in \eqref{eq:Ca L1 int}.
Hence, we only need to prove \eqref{eq:dist a}.

We argue by contradiction, similarly to \cite{S-parabolic,RS-Duke}. 
Assume that there exist a sequence of radii $r_k \in (0,1)$ and of stable solutions $u_k$ with nonlinearities $f_k$ in domains $\Omega_k$, with $u_k,f_k,\Omega_k$ satisfying the hypotheses of the theorem and such that
\begin{equation}\label{agniownonw}
r_k^{2-n}\int_{\Omega_k\cap B_{r_k}}|\nabla u_k|^2\,dx \ge k\, r_k^{\alpha_n}\|\nabla u_k\|^2_{L^2(\Omega_k\cap B_1)} 
\end{equation}
for all $k\in \mathbb N$. Then, for $r\in (0,1)$  we define the nonincreasing function
\[
\Theta(r) : =  \sup_k \sup_{s\in (r,1)} \frac{s^{2-n}\int_{\Omega_k\cap B_{s}}|\nabla u_k|^2\,dx  }{s^{\alpha_n}\|\nabla u_k\|^2_{L^2(\Omega_k\cap B_1)}}  \\\
\]
and note that $\Theta$ is finite since obviously $\Theta(r)\leq r^{2-n-\alpha_n}<\infty$ for all $r>0$. 
By \eqref{agniownonw} and since $\Theta$ is nonincreasing we have $\Theta(r)\uparrow +\infty$ as $r\downarrow 0$.
 Also, by the definition of $\Theta$, for any given $m\in \mathbb N$ there exists $\bar r_m\in (1/m,1)$ and $k_m$ such that 
\begin{equation}\label{agnwiohbgiowubow}
\Theta(\bar r_m) \ge \frac{\bar r_m^{2-n}\int_{\Omega_{k_m}\cap B_{\bar r_m}}|\nabla u_{k_m}|^2\,dx  }{\bar r_m^{\alpha_n}\|\nabla u_{k_m}\|^2_{L^2(\Omega_{k_m}\cap B_1)}}  \ge \frac{9}{10}\Theta(1/m) \ge\frac{9}{10}\Theta(\bar r _m) .     
\end{equation}
Since $\Theta(1/m) \uparrow \infty$ as $m\uparrow \infty$, it follows that $\bar r_m\downarrow 0$. 

Consider the sequence of functions
\[
\bar u_m : = \frac{u_{k_m} (\bar r_m \, \cdot \,)  }{\bar r_m^{\alpha_n} \,\Theta(\bar r_m)\,\|\nabla u_{k_m}\|^2_{L^2(\Omega_{k_m}\cap B_1)}},
\]
and denote $\tilde\Omega_m : = \frac{1}{\bar r_m}\Omega_{k_m}$. 
Then $\tilde \Omega_m\to \R^n_+$ locally uniformly as $m \to \infty$, and 
for all $R\in [1,1/\bar r_m)$ we have
\begin{equation}\label{growth111}
\begin{split}
R^{2-n} \int_{\tilde\Omega_m\cap B_{R}}  |\nabla \bar u_m|^2\,dx &= 
\frac{ (R\bar r_m)^{2-n}\int_{{\Omega_{k_m}} \cap B_{R \bar r_m}}   |\nabla u_{k_m}|^2\,dx }{\bar r_m^{\alpha_n} \,\Theta(\bar r_m)\,\|\nabla u_{k_m}\|^2_{L^2(\Omega_{k_m}\cap B_1)}} 
\\ 
&\le  \frac{ (R\bar r_m)^{2-n}\int_{\Omega_{k_m} \cap B_{R\bar r_m}}   |\nabla u_{k_m}|^2 \,dx}{(R\bar r_m)^{\alpha_n}\, \Theta( R\bar r_m)\,\|\nabla u_{k_m}\|^2_{L^2(\Omega_{k_m}\cap B_1)}}  R^{\alpha_n} \le R^{\alpha_n},
\end{split}
\end{equation}
where we used that $\Theta( R\bar r_m) \le    \Theta( \bar r_m)$ since $R\ge1$.

On the other hand, using \eqref{agnwiohbgiowubow} we have
\begin{equation}\label{growth222}
 \int_{\tilde \Omega_m\cap B_{1}}  |\nabla \bar u_m|^2\,dx \ge \frac{9}{10}.
\end{equation}

Now, similarly to Step 2 in the proof of Proposition \ref{classification}, thanks to Proposition \ref{prop:W12+epbdry} and Lemma \ref{strongconvergence}
there exists a function $\bar u$ such that, up to a subsequence, $\bar u_m\to \bar u$ strongly in $W^{1,2}_{\rm loc}(\R^n_+).$
In addition, since $\tilde \Omega_m\to \R^n_+$, using again Proposition \ref{prop:W12+epbdry} we see that, for every  $R\ge 1$ and $\delta \in (0,1)$,  
\[\begin{split}
\int_{\tilde \Omega_m \cap B_{R} \cap \{x_n \leq \delta\}} |\nabla \bar u_m|^2\,dx &\leq \Big( \int_{\tilde\Omega_m \cap B_{R} \cap \{x_n \leq \delta\}} \hspace{-1mm}|\nabla \bar u_m|^{2+\gamma}\,dx\Big)^{\frac{2}{2+\gamma}} \bigl|\tilde \Omega_m \cap B_{R} \cap \{x_n \leq \delta\}\bigr|^{\frac{2}{2+\gamma}} \\
&\leq C(R) \,\left(\delta^{\gamma/(2+\gamma)}+o_m(1)\right),
\end{split}
\]
where $o_m(1)\to 0$ as $m\to \infty$.
{Hence, as $m\to \infty$} the mass of $|\nabla \bar u_m|^2$ near the boundary can be made arbitrarily small by choosing $\delta$ small enough, and combining this information with the convergence of $\bar u_m\to \bar u$ in $W^{1,2}_{\rm loc}(\R^n_+)$ we deduce that
$\bar u_m\to \bar u$ strongly in $W^{1,2}(B_{R}^+)$ for all $R\ge 1$.

Moreover, by  Theorem \ref{thm:stability} we obtain that $\bar u \in \mathcal S(\R^n_+)$, and taking the limit in \eqref{growth111} and  \eqref{growth222} we obtain 
\[
\int_{B_1^+}  |\nabla \bar u|^2\,dx \ge \frac {9}{10}  \qquad \mbox{and}\qquad   \|\nabla \bar u_R\|^2_{L^2(B_1^+)} =  R^{2-n}\int_{B_R^+}  |\nabla \bar u|^2\,dx \le R^{\alpha_n}   \quad \mbox{for all } R\geq1, 
\]
where $\bar u_R : = \bar u(R\,\cdot\,)$.
Moreover, since the trace operator is continuous in $W^{1,2}(B_{R}^+)$, we have $\bar u=0$ on  $\{x_n=0\}$.
The last bound (applied with $R$ replaced by $2R$) and Proposition \ref{prop:W12+epbdry} give that $\bar u$ satisfies the hypothesis \eqref{P4} in Proposition \ref{classification} with $\alpha=\alpha_n/2$.

Therefore, to show that  $\bar u$ satisfies the assumptions Proposition \ref{classification} with $\alpha = \alpha_n/2$, it only remains to prove that \eqref{P6} holds (with $u$  replaced by $\bar u$).  
This is a consequence of Lemma~\ref{lem:x Du def}:  since $\frac{1}{R}\tilde \Omega_m$ is a $\vartheta$-deformation of $B_2^+$ with $\vartheta = CR \bar r_m $, for all $\eta \in C^{0,1}_c \big( \overline{B_1^+ }\big)$ we have
\[\begin{split}
\int_{(\frac{1}{R}\tilde \Omega_m)\cap B_1}& \Big(\big\{(n-2)\eta + 2 x\cdot\nabla \eta \}\,\eta\, |\nabla \bar u_{m,R}|^2 - 2(x\cdot \nabla \bar u_{m,R}) \nabla \bar u_{m,R}\cdot \nabla(\eta^2) \Big)\,dx \\
&\qquad - \int_{(\frac{1}{R}\tilde \Omega_m)\cap B_1} |x\cdot \nabla \bar u_{m,R}|^2 |\nabla \eta|^2\,dx\le CR\bar r_m \int_{(\frac{1}{R}\tilde \Omega_m)\cap B_1}  |\bar u_{m,R}|^2\,dx,
\end{split}\]
and hence, by letting $m \to \infty$, we deduce that
$$
\int_{B_1^+} \Big(\big\{(n-2)\eta + 2 x\cdot\nabla \eta \}\,\eta\, |\nabla  \bar u_{R}|^2 - 2(x\cdot \nabla  \bar u_{R}) \nabla  \bar u_{R}\cdot \nabla(\eta^2) - |x\cdot \nabla  \bar u_{R}|^2 |\nabla \eta|^2 \Big)\,dx\leq 0
$$
for all $\eta \in C^{0,1}_c \big( \overline{B_1^+}\big)$.
Since this holds for all $R>1$, this proves that \eqref{P6} holds for every $\eta \in C^{0,1}_c \big( \overline{\R^n_+}\big)$ with $u$ replaced by $\bar u$.
Thus, it follows by Proposition \ref{classification} that $\bar u\equiv 0$, a contradiction since $ \int_{B_{1}^+}  |\nabla \bar u|^2\,dx \ge \frac{9}{10} $.
\end{proof}

As explained at the beginning of this section, Theorem \ref{thm:globalCalpha} follows immediately from Theorem \ref{thm:uptoboundary}.
Thus, it only remains to give the:

\begin{proof}[Proof of Corollary \ref{thm:conjecture}]
Since $u\in W^{1,2}_0(\Omega)$, it follows from \eqref{weak-sol-intro} and a standard approximation argument that 
\[\int_\Omega f(u){\rm dist}(\cdot,\partial\Omega)\,dx\leq C\|\nabla u\|_{L^2(\Omega)}.\]
Thus, thanks to the approximation argument in \cite[Theorem 3.2.1 and Corollary 3.2.1]{Dup}, $u$ can be written as the limit of classical solutions $u_\varepsilon \in C^2_0(\overline\Omega)$ of $-\Delta u_\varepsilon=(1-\varepsilon)f(u_\varepsilon)$ in $\Omega$, {as  $\varepsilon\downarrow0$}.
Thus, applying Theorem \ref{thm:globalCalpha} to the functions $u_\varepsilon$, using Proposition \ref{prop:L1}, and letting $\varepsilon\downarrow 0$, the result follows.
\end{proof}

\section{\for{toc}{Estimates for $n \geq 10$}\except{toc}{Estimates for $n \geq 10$: Proof of Theorem \ref{thm11}}}
\label{sect:n10}

In this section we show how our method also gives sharp information in higher dimensions.
We first deal with the interior case, and we prove a strengthened version of Theorem \ref{thm11}.
Recall the definition of the Morrey space $M^{p,\beta}(\Omega)$ given in Section \ref{section1.3}.
Here $p\geq1$ and $\beta\in(0,n)$.

\begin{theorem}\label{thm11 1}
Let $u\in C^2(B_1)$ be a stable 
solution of
$$
-\Delta u=f(u) \quad \text{in }B_1,
$$
{
with  $f:\R \to\R$ locally Lipschitz,}
and assume that $n\geq 10$.
Then
$$
\Vert u\Vert_{M^{\frac{2\beta}{\beta -2},\beta}(B_{1/4})}+\Vert\nabla u\Vert_{M^{2,\beta}(B_{1/4})}
\leq C \Vert u\Vert_{L^1(B_1)}\quad\ \text{for every }\ \
\beta\in (n- 2\sqrt{n-1} -2,n),
$$
for some constant $C$ depending only on $n$ and $\beta$. 
In particular \eqref{eq:p n11} holds.
\end{theorem}

{
Recall that, in the radially symmetric case, 
\begin{equation}\label{emb}
\mbox{if}\ u\ \mbox{is radial and}\ \nabla u\in M^{2,\beta}(B_{1/4}),\ \mbox{then} \  u\in L^{p}(B_{1/8}) \ \mbox{for all}\ p<2n/(\beta-2);
\end{equation}
}indeed this follows from \cite{CC19} after cutting-off $u$ outside $B_{1/8}$ to have compact support in $B_{1/4}$.

{Thus, Theorem \ref{thm11 1} together with \eqref{emb} yield the following $L^p$ bound for radial solutions}:
\begin{equation}
\label{eq:p n11 2}
\Vert u\Vert_{L^{p}(B_{1/8})}\leq C \Vert u\Vert_{L^1( B_{1})}\quad\ \text{ for every }\ \
p<\bar p_n := \frac{2n}{n-2\sqrt{n-1}-4}.
\end{equation}
Hence, in the radial case we recover the $L^p$ estimates established by Capella and the first author in \cite{CC06}, which are known to be sharp: \eqref{eq:p n11 2} cannot hold for $p=\bar p_n$.

Unfortunately, as shown recently by Charro and the first author in~\cite{CC19}, the embedding \eqref{emb} is {false} for non-radial functions,\footnote{When $\beta\in(2,n)$ is an integer, this can be easily shown considering functions in $\R^n$ depending only on $\beta$ Euclidean variables; see \cite{CC19}. We thus encounter here the same obstruction as in Remark \ref{rem22}.} 
and thus it is not clear whether \eqref{eq:p n11 2} holds in the nonradial case, too.
From $\nabla u\in M^{2,\beta}$, the best one can say is $u\in M^{\frac{2\beta}{\beta-2},\beta}\subset L^{\frac{2\beta}{\beta-2}}$ as stated in Theorem \ref{thm11 1}.

\begin{proof}[Proof of Theorem \ref{thm11 1}]
We split the proof into two cases.

\vspace{2mm}

\noindent
{\bf Case 1:} Assume first $n\geq 11$.
Then, repeating the proof of Lemma \ref{conseqestab2}, in Step 2 we can take an exponent $a$ satisfying
\begin{equation}
\label{eq:choice a}
8<a<2(1+\sqrt{n-1})<n-2.
\end{equation}
Then, choosing  $0\leq \zeta\leq 1$ such that $\zeta_{|B_{1/4}}=1$, $\zeta_{|\R^n\setminus B_{1/2}}=0$,
and  $|\nabla \zeta| \leq C$, we 
obtain
$$
 \int_{B_{1/4}}   \Big\{(n-2-a)  |\nabla u|^2 + \Big( 2a -\frac{a^2}{4} \Big) |x\cdot \nabla u|^2 |x|^{-2}   \Big\}  |x|^{-a} \,dx
 \le C(n, a) \int_{B_{1/2}\setminus B_{1/4}} |\nabla u|^2\,dx.
$$
Since $2a-a^2/4<0$, 
the left hand side above can be bounded from below by
$$
(n-2+a-a^2/4) \int_{B_{1/4}}    |\nabla u|^2  |x|^{-a} \,dx,
$$
and because $n-2+a-a^2/4>0$ (thanks to the choice of $a$ in \eqref{eq:choice a}), we deduce that 
$$
\int_{B_{1/4}}    |\nabla u|^2  |x|^{-a} \,dx
\leq   C \int_{B_{1/2}\setminus B_{1/4}} |\nabla u|^2\,dx\leq C\|u\|_{L^1(B_1)}^2,
$$
where the last inequality follows from Proposition \ref{prop:L1controlsW12}.

Applying this estimate to the functions $u_y(x):=u(y+x)$ with $y \in \overline B_{1/4}$, it follows that 
$$
\rho^{-a}\int_{B_\rho(y)}    |\nabla u|^2  \,dx \leq \int_{B_{1/4}(y)}    |\nabla u(x)|^2  |x-y|^{-a} \,dx \leq C\|u\|_{L^1(B_1)}^2\quad \mbox{for all }\, y \in \overline B_{1/4}, \, \rho\in (0,{\textstyle \frac14}).
$$
This proves that $\nabla u\in M^{2,\beta}(B_{1/4})$ for every $\beta:=n-a > n-2\sqrt{n-1}-2$.

Now, after cutting-off $u$ outside of $B_{1/8}$ to have compact support in $B_{1/4}$, we can apply  \cite[Proposition~3.1 and Theorems~3.1 and 3.2]{A} (see also the proof in \cite[Section 4]{CC19}) and, since $\beta \in (2,n)$, we deduce that
$u\in M^{\frac{2\beta}{\beta -2},\beta}(B_{1/8})$.
This estimate in $B_{1/8}$ can also be stated in $B_{1/4}$, as in Theorem  \ref{thm11 1}, after an scaling and covering argument.
Taking $p=\frac{2\beta}{\beta-2}$, this leads to \eqref{eq:p n11}.
\\

\noindent
{\bf Case 2:} Assume now $n=10$.
Then, repeating the proof of Lemma \ref{conseqestab2}, in Step 2 we take 
\begin{equation}
\label{eq:choice-n=10}
\eta=|x|^{-4}\big|\log|x|\big|^{-\delta/2}\zeta,
\end{equation}
with $\delta>0$ small.
Then, choosing  $0\leq \zeta\leq 1$ such that $\zeta_{|B_{1/4}}=1$, $\zeta_{|\R^n\setminus B_{1/2}}=0$, 
{and  $|\nabla \zeta| \leq C$}, we 
obtain
\[\begin{split}
 &\int_{B_{1/4}}   \delta\big|\log|x|\big|^{-1-\delta}  |\nabla u|^2|x|^{-8}\,dx \\ 
 &\qquad+ \int_{B_{1/4}}\Big\{2\delta\big|\log|x|\big|^{-1-\delta}|x\cdot \nabla u|^2 |x|^{-2}
 -(\delta^2/4)\big|\log|x|\big|^{-2-\delta}|x\cdot \nabla u|^2 |x|^{-2}   \Big\}  |x|^{-8} \,dx \\
& \le C(n,\delta) \int_{B_{1/2}\setminus B_{1/4}} |\nabla u|^2\,dx.
\end{split}\]
Now, using that 
\[(\delta^2/4)\big|\log|x|\big|^{-2-\delta}\leq 2\delta\big|\log|x|\big|^{-1-\delta}\quad \mbox{in }\, B_{1/4},\]
we deduce
\[\int_{B_{1/4}} \big|\log|x|\big|^{-1-\delta}  |\nabla u|^2|x|^{-8}\,dx
\leq C(n, \delta) \int_{B_{1/2}\setminus B_{1/4}} |\nabla u|^2\,dx.\]
Finally, since for every $\ep>0$ we have
$$|x|^{-8+\ep}\leq C(n,\delta,\ep)\big|\log|x|\big|^{-1-\delta}|x|^{-8}\quad \mbox{in}\ \,B_{1/4},$$
we find that 
\[\int_{B_{1/4}} |\nabla u|^2|x|^{-a}\,dx
\leq C(n, \delta,a) \int_{B_{1/2}\setminus B_{1/4}} |\nabla u|^2\,dx\]
for all $a:=8-\ep<8$.
The rest of the proof is then analogous to the case $n\geq11$.
\end{proof}

We now deal with a boundary version of the same theorem.
We first consider a domain $\Omega$ that is a $\vartheta$-deformation of $B_2^+$ (see Definition \ref{def:deform}).

\begin{theorem}\label{thm11 2}
Let $\Omega\subset \R^n$  be a $\vartheta$-deformation of $B_2^+$  for $\vartheta\in[0,\frac{1}{100}]$, and let $u\in C^0(\overline\Omega\cap B_1)\cap C^2(\Omega\cap B_1)$ be a {nonnegative} stable 
solution of
$$
-\Delta u=f(u) \quad \text{in }\Omega\cap B_1 \qquad \mbox{and} \qquad u=0 \quad \mbox{on } \partial \Omega\cap B_1,
$$
with  $f:\R \to\R$ locally Lipschitz, nonnegative, {and nondecreasing}.
Assume that $n\geq 10$.
Then
$$
\Vert u\Vert_{M^{\frac{2\beta}{\beta -2},\beta}(\Omega \cap B_{1/2})}+\Vert\nabla u\Vert_{M^{2,\beta}(\Omega \cap B_{1/2})}
\leq C \Vert u\Vert_{L^1(\Omega\cap B_1)}\quad\ \text{for all }\, \
\beta\in (n- 2\sqrt{n-1} -2,n),
$$
for some constant $C$ depending only on $n$ and $\beta$. 
\end{theorem}

\begin{proof}
Assume $n\geq11$; the case $n=10$ can be handled similarly (as done in the proof of Theorem \ref{thm11 1}).
In this case we start from Lemma \ref{lem:x Du def} and, as in Step 1 in the proof of Proposition \ref{classification}, we let 
$\psi\in C^\infty_c (B_1)$ be some  radial decreasing nonnegative cut-off function with $\psi\equiv 1$ in $B_{1/2}$.
In Lemma \ref{lem:x Du def} we use the test function $\eta(x) := |x|^{-a/2} \psi(x)$ with $a<n$.
Then, since the domain $\rho^{-1}(\Omega\cap B_\rho)$ is a $(c\rho\vartheta)$-deformation of $B_2^+$ (for some dimensional constant $c$), we deduce that 
\begin{multline*}
 \int_{\Omega\cap B_\rho}   \Big\{(n-2-a-C\rho\vartheta)  |\nabla u|^2 + \Big( 2a -\frac{a^2}{4} \Big) (x\cdot \nabla u)^2 |x|^{-2}   \Big\}  |x|^{-a} \,dx
 \\
 \le C(n, a) \rho^{-a}\int_{\Omega\cap B_{2\rho}\setminus B_\rho} |\nabla u|^2\,dx.
 \end{multline*}
 Hence, given $a$ satisfying \eqref{eq:choice a}, we can take $\rho_0$ sufficiently small (depending on $n$ and $a$) so that $n-2+a-a^2/4-C\rho\vartheta>0$ for all $\rho\leq \rho_0$.
 This allows us to argue as in the proof of Theorem \ref{thm11 1} to get
\[\int_{\Omega\cap B_{\rho_0}}    |\nabla u|^2  |x|^{-a} \,dx 
\leq C\int_{\Omega\cap B_{1/2}} |\nabla u|^2\,dx \leq C\|u\|_{L^1(\Omega\cap B_1)}^2\]
by Proposition \ref{prop:L1controlsW12bdry}.
We now conclude as in  Theorem \ref{thm11 1}.
\end{proof}

We finally give the:

\begin{proof}[Proof of Theorem \ref{thm11}]
The estimate \eqref{eq:p n11} follows from  Theorem \ref{thm11 1} by taking $p=\frac{2\beta}{\beta-2}$ and using a covering argument.
On the other hand, \eqref{eq:p n11 global} follows from Theorems \ref{thm11 1} and \ref{thm11 2}, using again a covering argument.
\end{proof}

\appendix

\section{Technical lemmata}

The next lemma is a regularity and compactness result for superharmonic functions.
For an integrable function $v$ to be superharmonic, we mean it in the distributional sense.
For all our applications of the lemma one could further assume that $v\in W^{1,2}(B_R)$ and that $-\Delta v\geq0$ is meant in the usual $W^{1,2}$ weak sense (which, in this case, is equivalent to the distributional sense), but we do not need this additional hypothesis.

\begin{lemma}\label{strongconvergence}
Let $v\in L^1(B_R)$ be superharmonic in a ball $B_R\subset\R^n$, and let $r\in (0,R)$.
Then:
\begin{itemize}
\item[(a)] The distribution $-\Delta v=|\Delta v|$ is a nonnegative measure in $B_R$, $v\in W^{1,1}_{\rm loc}(B_R)$, 
$$
\int_{B_r}|\Delta v|\leq \frac{C}{(R-r)^2} \int_{B_R}|v|\,dx,
\qquad \mbox{and}
\qquad 
\int_{B_r}|\Delta v|\leq \frac{C}{R-r} \int_{B_R}|\nabla v|\,dx,
$$ 
where $C>0$ is a dimensional constant.
In addition, 
$$\int_{B_r}|\nabla v|\,dx\leq C(n,r,R) \int_{B_R}|v|\,dx
$$ 
for some constant $C(n,r,R)$ depending only on $n$, $r$, and $R$.
\end{itemize}

Assume now that $v_k\in L^1(B_R)$, $k=1,2,...$, is a sequence of superharmonic functions with $\sup_k \|v_k\|_{L^1(B_R)}<\infty$.
Then:
\begin{itemize}
\item[(b1)]  Up to a subsequence, $v_k\to v$ strongly in $W^{1,1}(B_r)$ to some superharmonic function~$v$.

\item[(b2)]
In addition, if for some $\gamma>0$ we have  $\sup_k \|v_k\|_{W^{1,2+\gamma}(B_r)}<\infty$, then $v_k\to v$  strongly in $W^{1,2}(B_r)$.
\end{itemize}
\end{lemma}

\begin{proof}
(a) By assumption we know that
\begin{equation}\label{xi-d}
\langle -\Delta v,\xi \rangle = \int_{B_R} v(-\Delta \xi)\,dx\geq0\quad \mbox{for all nonnegative}\ \xi\in C^\infty_c(B_R).
\end{equation}
Let $0<r<\rho<R$ and choose a nonnegative function $\chi\in C^\infty_c(B_R)$ with $\chi\equiv1$ in $B_\rho$.
Now, for all $\eta\in C^\infty_c(B_\rho)$, using \eqref{xi-d} with the test functions $\|\eta\|_{C^0}\chi\pm \eta\geq0$ in $B_R$, we deduce that $\pm \langle -\Delta v,\eta \rangle \leq 
\|\eta\|_{C^0}\|v\|_{L^1(B_R)} \|\Delta\chi\|_{C^0}\leq C\|\eta\|_{C^0}.$
Thus, $-\Delta v$ is a nonnegative measure in $B_\rho$, for all $\rho<R$.

Let us now take $\rho=\frac12(r+R)$, and consider $\chi$ as before satisfying $|\nabla \chi|\leq \frac{C}{R-r}$ and $|D^2\chi|\leq \frac{C}{(R-r)^2}$.
Then, since $-\Delta v \geq 0$, we have
\begin{equation}\label{uuuu}
\int_{B_r}|\Delta v|
\leq -\int_{B_R} \Delta v \,\chi=-\int_{B_R} v\, \Delta \chi\,dx \leq \frac{C}{(R-r)^2}\|v\|_{L^1(B_R)}.
\end{equation}

To prove that $v\in W^{1,1}_{\rm loc}(B_R)$, we define {on $\R^n$} the measure 
$\mu:=\chi\,(-\Delta v)$,
and we consider the fundamental solution $\Phi=\Phi(x)$ of the Laplacian in $\R^n$ ---{that is, $\Phi(x)= c \log|x|$ if $n=2$ and $\Phi(x)= c_n |x|^{2-n}$ if $n\geq3$}.

Define the $L^1_{\rm loc}(\R^n)$ function  $\tilde v :=  \Phi \ast \mu$. 
Since $\Phi \in W^{1,1}_{\rm loc}(\R^n)$ it is easy to check (using the definition of weak derivatives) that $\tilde v\in W^{1,1}(B_R)$ and $\nabla \tilde v= \nabla \Phi\ast \mu$.
Furthermore, from \eqref{uuuu} (with $r$ replaced by $\rho$), {one easily deduces}  that
\begin{equation}
\label{eq:W11 lapl}
\|\tilde v\|_{W^{1,1}(B_\rho)} \leq C\|v\|_{L^1(B_R)},
\end{equation}
where the constant $C$ depends only on $n$, $r$, and $R$ ({recall that $\rho=\frac12(r+R)$}).

On the other hand, using \eqref{eq:W11 lapl}, we see that $w := v-\tilde v$ satisfies
\[
\|w\|_{L^1(B_\rho)} \le \|v\|_{L^1(B_\rho)} +  \|\tilde v\|_{L^1(B_\rho)}  \le C\|v\|_{L^1(B_R)} 
\]
and 
\[
\Delta w  = 0 \quad \mbox{in }B_\rho.
\]
By standard interior estimates for harmonic functions, this leads to
$$
\|w\|_{C^2(\overline B_r)}
\leq C\|w\|_{L^1(B_\rho)} \le C\|v\|_{L^1(B_R)}.
$$
In particular, recalling \eqref{eq:W11 lapl},
we have shown that $v\in W^{1,1}(B_r)$ and  $\|v\|_{W^{1,1}(B_r)}\leq C\|v\|_{L^1(B_R)}$.
 
 Finally, exactly as in \eqref{uuuu}, we have
$$
\int_{B_r}|\Delta v|
\leq -\int_{B_R} \Delta v \,\chi=\int_{B_R} \nabla v\cdot \nabla  \chi\,dx \leq \frac{C}{R-r}\|\nabla v\|_{L^1(B_R)},
$$
finishing the proof of (a).

\vspace{2mm}

(b1) Let now $v_k$ be a bounded sequence in $L^1(B_R)$.
Define $\mu_k$, $\tilde v_k$, $w_k$ {as we did in the proof of (a), but} with $v$ replaced by $v_k$.
Note that the operators $\mu\mapsto \Phi\ast \mu$ and $\mu\mapsto \nabla \Phi\ast \mu$ are compact from the space of measures (with finite mass and support in $B_R$) to $L^1(B_r)$.
This is proved in a very elementary way in \cite[Corollary 4.28]{Brezis} when these operators are considered from $L^1(\R^n)$ to $L^1(B_r)$, but the same exact proof works for measures.

Thus, up to a subsequence, $\tilde v_k$ converges in $W^{1,1}(B_r)$.
Since $w_k=v_k-\tilde v_k$ are harmonic and uniformly bounded in $L^1(B_\rho)$, up to a subsequence also $w_k$ converges in $W^{1,1}(B_r)$.
Therefore, we deduce that a subsequence of $v_k$ converges in $W^{1,1}(B_r)$, which proves (b1).

\vspace{2mm}

(b2) If in addition we have $\sup_k \|v_k\|_{W^{1,2+\gamma}(B_r)}<\infty$, using H\"older inequality we obtain
\[
\|\nabla (v_k -v)\|_{L^2(B_r)} \le \|\nabla (v_k -v)\|_{L^1(B_r)}^{\frac{\gamma}{2(1+\gamma)} }\|\nabla (v_k -v)\|_{L^{2+\gamma}(B_r)}^{\frac{2+\gamma}{2(1+\gamma)}} \le C \|v_k -v\|^{\frac{\gamma}{2(1+\gamma)}}_{W^{1,1}(B_r)} \to 0,
\]
which shows that  $v_k\to v$ strongly in $W^{1,2}(B_r)$.
\end{proof}

We now discuss a result about the composition of Lipschitz functions with $C^2$ functions. This result is far from being sharp in terms of the assumptions, but it suffices for our purposes. 
For its proof (as well as for other results proved in this paper) we shall need the coarea formula,
which we recall here for the convenience of the reader (we refer to \cite[Theorem 18.8]{Maggi} for a proof):

\begin{lemma}
\label{lem:coarea}
Let $\Omega\subset\R^n$ be an open set, and let  $u:\Omega\to \R$ be a Lipschitz function.
Then, for every function $g:\Omega \to \R$ such that {$g_-\in L^1(\Omega)$}, 
the integral of $g$ over $\{u=t\}$ is well defined in $(-\infty,+\infty]$ for a.e. $t\in\R$ and
$$
\int_\Omega g\,|\nabla u|\,dx =\int_\R \biggl(\int_{\{u=t\}}g\,d\mathcal H^{n-1}\biggr)\,dt.
$$
\end{lemma}

We recall that, given a locally Lipschitz function $f$, we defined
\[ f'_-(t) : = \liminf_{h\to 0}\frac{f(t+h)-f(t)}{h}.\]

\begin{lemma}
\label{lem:eq Du}
Let $\Omega\subset\R^n$ be a bounded open set and let  $u \in C^2(\Omega)\cap C^0(\overline\Omega)$ solve $-\Delta u=f(u)$ in $\Omega$, where $f:\R\to \R$ is locally Lipschitz. 
Then:
\begin{itemize}
\item[(i)] Inside the region $\{\nabla u \neq 0\}$ the function $f'(u)$ is well-defined and it coincides a.e. with $f'_-(u)$.
\item[(ii)] $u \in W^{3,p}_{\rm loc}(\Omega)$ for every $p<\infty$ and
$
-\Delta \nabla u=f'(u)\nabla u=f_-'(u)\nabla u$ in the weak sense and also a.e. in $\Omega$.

\item[(iii)] If $\partial\Omega\cap B_1$ is of class $C^3$ and $u|_{\partial\Omega\cap B_1}=0$,
then $u \in (W^{3,p}_{\rm loc}\cap C^2)(\overline\Omega\cap B_1)$ for every $p<\infty.$
\end{itemize}
\end{lemma}

\begin{proof}
The first point is a simple application of the coarea formula.
Indeed, if we set $M:=\|u\|_{L^\infty(\Omega)}$, given any Borel set $E\subset \{\nabla u \neq 0\}$ we can apply Lemma \ref{lem:coarea} with $g=\frac{1_E}{|\nabla u|} f'_-(u)$ (in fact we apply the lemma to both $g^+$ and $g^-$, the positive and negative part of $g$, obtaining finite quantities for both $\int_\Omega g^\pm |\nabla u|\,dx$ since $f'_-(u)\in L^\infty(\Omega)$) to get
\begin{equation}
\label{eq:coarea}
\int_Ef_-'(u)\,dx=\int_{-M}^M f_-'(t) g_E(t)\,dt,\qquad \text{with}\quad g_E(t):=\int_{\{u=t\}\cap E}\frac{1}{|\nabla u|}\,d\mathcal H^{n-1}.
\end{equation}
Then, since 
$$
 \int_\R g_E(t)\,dt \leq  \int_\R \biggl(\int_{\{u=t\}\cap \{|\nabla u|\neq0\}}\frac{1}{|\nabla u|}\,d\mathcal H^{n-1}\biggr)\,dt =\bigl|\Omega \cap \{\nabla u\neq 0\}\bigr|<\infty, 
$$
it follows that the function $g_E$ belongs to $L^1(\R)$.
Thus, since $f_{-}'(t)$ belongs to  $L^\infty([-M,M])$,
this proves that the right hand side in \eqref{eq:coarea} is independent of the specific representative chosen for $f'$, and therefore so is the left hand side. 
{Since $E$ is arbitrary and} $f_-'(t)=f'(t)$ a.e., (i) follows.

To prove (ii) we first notice that, since $f(u)$ is Lipschitz inside $\Omega$ (because both $u$ and $f$ are so), it follows that $f(u)\in W^{1,p}$ and by interior elliptic regularity (see for instance \cite[Chapter 9]{GT}) that $u \in W^{3,p}_{\rm loc}(\Omega)$ for every $p<\infty$.
This means that $\nabla u\in W^{2,p}$, and therefore it suffices to show that the identities in (ii) hold a.e. (because then they automatically hold in the weak sense).

Now, in the region $\{\nabla u=0\}$, we have
$$
f'(u)\nabla u=f_-'(u)\nabla u=0\qquad \text{and}\qquad \Delta \nabla u=0\quad \text{a.e.}
$$
(see, e.g., \cite[Theorem 1.56]{Tro}), so the result is true there.

On the other hand,  in the region $\{\nabla u\neq0\}$, for $h>0$ and $1\le i\le n$, let 
\[\delta_i^h w : = \frac{w(\cdot+h\boldsymbol e_i) -w}{h}.\]  
Since $-\Delta u=f(u)$ in $\Omega$, given $\Omega'\subset \subset \Omega$, for $h>0$ sufficiently small we have
\begin{equation}
\label{eq:delta h}
-\Delta \delta_i^h  u= \delta_i^h \big( f(u) \big) \qquad \text{inside }\Omega'.
\end{equation}
Thus, if we define by $D_f\subset \R$ the set of differentiability points of $f$, we see that 
$$
\delta_i^h[f(u)] \to f’(u) \partial_iu=f_-’(u) \partial_iu\qquad \mbox{for all }\,x\in \Omega' \text{ such that }u(x)\in D_f
$$
as $h\to0$.
On the other hand, if we set $N:=\R\setminus D_f$, since $N$ has measure zero ({because $f$ is differentiable a.e., being Lipschitz}) it follows from Lemma \ref{lem:coarea} applied with $g=\frac{1}{|\nabla u|}\,1_{\Omega'\cap \{\nabla u\neq 0\}}\,1_N\circ u$ that
$$
\int_{\Omega'\cap \{\nabla u\neq 0\}} 1_N(u(x)) \,dx=\int_\R 1_N(t)\biggl(\int_{\{u=t\}\cap \Omega'\cap \{\nabla u\neq 0\}}\frac{1}{|\nabla u|}\,d\mathcal H^{n-1}\biggr)\,dt=0,
$$
which proves that  $u(x) \not\in N$ for a.e. $x\in \Omega'\cap \{\nabla u\neq 0\}$.

Hence, we have shown that $\delta_i^h[f(u)] \to f'(u) \partial_iu$ for a.e. $x\in \Omega'\cap\{\nabla u\neq0\}$ (and so also in $L^p$ for any $p<\infty,$ by dominated convergence).
Letting $h\to 0$ in \eqref{eq:delta h} we deduce that $-\Delta\nabla u=f'(u)\nabla u$ a.e. in $\Omega'$, since we already checked the equality a.e. in $\{\nabla u=0\}$. 
Recalling that $\Omega'\subset\subset \Omega$ is arbitrary, this proves (ii).

Finally, (iii) follows by elliptic regularity up to the boundary (see for instance~\cite[Chapter 9]{GT} or \cite[Section 9.2]{Krylov}).
\end{proof}

We conclude this section with a general abstract lemma  due to Simon \cite{Simon} (see also  \cite[Lemma 3.1]{CSV}):

\begin{lemma}\label{lem_abstract}
Let $\beta\in \R$ and $C_0>0$. Let $\sigma :\mathcal B \rightarrow [0,+\infty]$ be a nonnegative function defined on the class $\mathcal B$  of open balls $B\subset \R^n$ and satisfying the following subadditivity property:
\[ \mbox{if}\quad  B \subset \bigcup_{j=1}^N B_j \quad  \mbox{ then }\quad \sigma(B)\le \sum_{j=1}^N \sigma(B_j). \]
Assume also that $\sigma(B_1)< \infty.$

Then, there exists $\delta>0$, depending only on $n$ and $\beta$, such that if
\begin{equation*}\label{hp-lem}
 r^\beta \sigma\bigl(B_{r/4}(y)\bigr) \le \delta r^\beta \sigma\bigl(B_r(y)\bigr)+ C_0\quad \mbox{whenever }B_r(y)\subset B_1,
 \end{equation*}
then
\[ \sigma(B_{1/2}) \le CC_0,\]
where $C$ depends only on $n$ and $\beta$.
\end{lemma}

\section{A universal bound on the $L^1$ norm}

In this section we recall a classical and simple a priori estimate on the $L^1$ norm of solutions when $f$ grows at infinity faster than a linear function with slope given by the first eigenvalue of the Laplacian.

\begin{proposition}
\label{prop:L1}
Let $\Omega\subset \R^n$ be a bounded domain of class {$C^1$},
and let $u\in C^0(\overline\Omega)\cap C^2(\Omega)$ solve
\[
\left\{
\begin{array}{cl}
-\Delta u=f(u) & \text{in }\Omega\subset \R^n\\
u=0 & \text{on }\partial\Omega
\end{array}
\right.
\]
for some $f: \R\to [0,+\infty)$ satisfying
\begin{equation}
\label{eq:superlin f}
f(t)\geq A t-B\qquad \mbox{for all }\,t \geq 0,\quad \text{with $A>\lambda_1$ and $B\geq 0$,}
\end{equation}
where $\lambda_1=\lambda_1(\Omega)>0$ is the {first eigenvalue of the Laplacian in $\Omega$ with Dirichlet homogeneous boundary condition}.
Then there exists a constant $C$, depending only on $A$, $B$, and $\Omega$, such that
$$
\|u\|_{L^1(\Omega)}\leq C.
$$
\end{proposition}

\begin{proof}
First of all we note that, since $f\geq 0$, then $u\geq 0$ inside $\Omega$ by the maximum principle.

Let $\Phi_1>0$ be the first Dirichlet eigenfunction of the Laplacian in $\Omega$, so that $-\Delta \Phi_1=\lambda_1\Phi_1$ in $\Omega$ and $\Phi_1=0$ on $\partial\Omega$. 
Then
\begin{equation}
\label{eq:rel u fu}
\lambda_1\int_\Omega u\,\Phi_1\,dx=-\int_\Omega u\,\Delta\Phi_1\,dx=-\int_\Omega \Delta u\,\Phi_1\,dx=\int_\Omega f(u)\,\Phi_1\,dx.
\end{equation}
Thanks to assumption \eqref{eq:superlin f}, we have
$$
\int_\Omega f(u)\,\Phi_1\,dx\geq A\int_\Omega u\,\Phi_1\,dx - B\int_\Omega \Phi_1\,dx,
$$
that combined with \eqref{eq:rel u fu} gives
$$
(A-\lambda_1)\int_\Omega u\,\Phi_1\,dx\leq B\int_\Omega \Phi_1\,dx.
$$
Note that, using \eqref{eq:rel u fu} again, this implies that
$$
\int_\Omega f(u)\,\Phi_1\,dx=\lambda_1\int_\Omega u\,\Phi_1\,dx\leq \lambda_1\frac{B}{A-\lambda_1}\int_\Omega \Phi_1\,dx.
$$
This proves that
\begin{equation}
\label{eq:dist f}
\int_\Omega  f(u)\,\Phi_1\,dx \leq C
\end{equation}
for some constant $C$ depending only on $A$, $B$, and $\Omega$.

Consider now $\phi:\Omega \to \R$ the solution of
$$
\left\{
\begin{array}{cl}
-\Delta \phi=1& \text{in }\Omega\\
\phi=0&\text{on }\partial\Omega.
\end{array}
\right.
$$
We claim that
\begin{equation}
\label{eq:dist phi}
0\leq \phi \leq C\,\Phi_1\qquad \textrm{in } \Omega,
\end{equation}
with $C$ depending only on $\Omega$.
Indeed, the nonnegativity of $\phi$ follows from the maximum principle,
while the second inequality 
follows from the boundary Harnack principle in \cite[Lemma 3.12]{AS}\footnote{See also \cite{RS-C1} for a different proof of the boundary Harnack principle in $C^1$ domains, that is written for nonlocal operators but works as well for the Laplacian. Note that, in $C^2$ domains,  the bound $\phi \leq C\,\Phi_1$ follows immediately from the fact that both $\phi$ and $\Phi_1$ are comparable to the distance function ${\rm dist}(\cdot,\partial\Omega)$.}, after rescaling.

Thus, using \eqref{eq:dist f} and \eqref{eq:dist phi} we get
$$
\int_\Omega u\,dx=-\int_\Omega u\,\Delta \phi\,dx
=\int_\Omega f(u)\,\phi\,dx \leq C\int_\Omega f(u)\,\Phi_1\,dx \leq C,
$$
as desired.
\end{proof}


\begin{thebibliography}{000000}

\bibitem{A}
Adams, D. R. 
A note on Riesz potentials.
{\it Duke Math. J.} 42 (1975), 765-778. 


\bibitem{AS}  Allen, M.; Shahgholian, H.  A new boundary Harnack principle (equations with right hand side). \emph{Arch. Rat. Mech. Anal.} 234 (2019), 1413-1444. 


\bibitem{BDG}
Bombieri, E; De Giorgi, E; Giusti, E.  Minimal cones and the Bernstein problem.
{\it Invent. Math.} 7 (1969), 243-268.


\bibitem{Br}
Brezis, H.
Is there failure of the inverse function theorem? {\it Morse theory, minimax theory and their applications to nonlinear differential equations}, 23-33, 
New Stud. Adv. Math., 1, Int. Press, Somerville, MA, 2003.

\bibitem{Brezis}
Brezis, H.
{\it Functional Analysis, Sobolev Spaces and Partial Differential Equations}. 
Universitext, Springer-Verlag New York, 2011.


\bibitem{BV}
Brezis, H; V\'azquez. J. L.  
Blow-up solutions of some nonlinear elliptic problems. 
{\it Rev. Mat.
Univ. Complut. Madrid} 10 (1997), 443-469.

\bibitem{C10}
Cabr\'e, X.  Regularity of minimizers of semilinear elliptic problems up to dimension 4.
{\it Comm.
Pure Appl. Math.} 63 (2010), 1362-1380.

\bibitem{C17}
Cabr\'e, X.  
Boundedness of stable solutions to semilinear elliptic equations: a survey.
{\it Adv. Nonlinear Stud.} 17 (2017),  355-368. 


\bibitem{C19}
Cabr\'e, X.  
A new proof of the boundedness results for stable solutions to semilinear elliptic equations. 
{\it Discrete Contin. Dyn. Syst.} 39 (2019), 7249-7264.


\bibitem{CC06}
Cabr\'e, X.; Capella, A.
Regularity of radial minimizers and extremal solutions of semilinear
elliptic equations. {\it J. Funct. Anal.} 238 (2006), 709-733.

\bibitem{CC19}
Cabr\'e, X.; Charro, F.
The optimal exponent in the embedding into the Lebesgue spaces for functions with gradient in the Morrey space. {\it Preprint arXiv}, 2019.


\bibitem{CMS19}
Cabr\'e, X.; Miraglio, P.; Sanch\'on, M.
Optimal regularity of stable solutions to nonlinear equations involving the $p$-Laplacian. {\it Preprint arXiv}, 2019.


\bibitem{CR-O}
Cabr\'e, X.; Ros-Oton, X. Regularity of stable solutions up to dimension 7 in domains of double revolution. {\it Comm. Partial Differential Equations} 38 (2013), 135-154. 

\bibitem{CSS}
Cabr\'e, X.; Sanch\'on, M.; Spruck, J.
A priori estimates for semistable solutions of semilinear elliptic equations.
{\it Discrete Contin. Dyn. Syst. Ser.} 36 (2016), 601-609.


\bibitem{CS19}
Cabr\'e, X.; Sanz-Perela, T.
BMO and $L^\infty$ estimates for stable solutions to fractional semilinear elliptic equations. {\it Forthcoming}, 2019.

\bibitem{CV}
Caffarelli, L.; Vasseur, A.
The De Giorgi method for regularity of solutions of elliptic equations and its applications to fluid dynamics. 
{\it Discrete Contin. Dyn. Syst. Ser. S} 3 (2010), 409-427. 


\bibitem{CSV} 
Cinti, E.; Serra, J.; Valdinoci, E. 
Quantitative flatness results and BV-estimates for stable nonlocal minimal surfaces.
{\it J. Diff. Geom.} 112 (2019), 447-504.

\bibitem{CR}
Crandall, M. G.; Rabinowitz, P. H. Some continuation and variational methods for positive
solutions of nonlinear elliptic eigenvalue problems.
{\it Arch. Ration. Mech. Anal.} 58 (1975), 207-218.


\bibitem{FLN} de Figuereido, D. G.; Lions, P.-L.; Nussbaum, R. D.  {A priori estimates and existence of positive solutions of semilinear elliptic equations}. {\it J. Math. Pures Appl.} 61 (1982), 41-63.


\bibitem{dCP} do Carmo, M.; Peng, C. K. {Stable complete minimal surfaces in $\R^3$ are planes}. {\it Bull. Amer. Math. Soc. (N.S.)} 1 (1979), 903-906.

\bibitem{Dup}
Dupaigne, L. {\it Stable Solutions of Elliptic Partial Differential Equations.} Chapman \& Hall/CRC Monogr. Surv. Pure Appl.
Math. 143, CRC Press, Boca Raton, 2011.


\bibitem{Xavi2} Fern\'andez-Real, X.; Ros-Oton, X. {\it Regularity Theory for Elliptic PDE.} Forthcoming book, 2020.


\bibitem{FishS} Fischer-Colbrie, D.; Schoen, R. {The structure of complete stable minimal surfaces in 3-manifolds of nonnegative scalar curvature}.
{\it Comm. Pure Appl. Math.} 33 (1980), 199-211.

\bibitem{Gelf}
Gel'fand, I. M.
Some problems in the theory of quasilinear equations. 
{\it Amer. Math. Soc. Transl. (2)} 29 (1963), 295-381. 

\bibitem{GT}
Gilbarg, D.; Trudinger, N. S. {\it Elliptic partial differential equations of second order.} Reprint of the 1998 edition. Classics in Mathematics. Springer-Verlag, Berlin, 2001.

\bibitem{JL}
Joseph, D. D.; Lundgren, T. S. Quasilinear Dirichlet problems driven by positive sources.
{\it Arch. Ration. Mech. Anal.} 49
(1972/73), 241-269.

\bibitem{Krylov}
Krylov, N. {\it Lectures on Elliptic and Parabolic Equations in Sobolev Spaces.}
Graduate Studies in Mathematics Vol. 96, AMS, 2008.


\bibitem{Maggi} Maggi, F. \emph{Sets of Finite Perimeter and Geometric Variational Problems}. Cambridge Studies in Advanced Mathematics 135, Cambridge University Press, 2012.

\bibitem{Martel}
Martel, Y. Uniqueness of weak extremal solutions of nonlinear elliptic problems.
{\it Houst. J. Math.} 23 (1997), 161-168.

\bibitem{Ned00}
Nedev, G. Regularity of the extremal solution of semilinear elliptic equations.
{\it C. R. Acad. Sci.
Paris} 330 (2000), 997-1002.

\bibitem{NedPrep}
Nedev, G. Extremal solutions of semilinear elliptic equations.
{\it Unpublished preprint}, 2001.


\bibitem{RS-Duke}
 Ros-Oton, X.; Serra, J.
Boundary regularity for fully nonlinear integro-differential equations.
{\it Duke Math. J.} 165 (2016), 2079-2154.

\bibitem{RS-C1} Ros-Oton, X.; Serra, J. {Boundary regularity estimates for nonlocal elliptic equations in $C^1$ and $C^{1,\alpha}$ domains}.  \emph{Ann. Mat. Pura Appl.} 196 (2017), 1637-1668.



\bibitem{S-parabolic}
Serra, J.
Regularity for fully nonlinear nonlocal parabolic equations with rough kernels.
{\it Calc. Var. Partial Differential Equations} 54 (2015), 615-629. 



\bibitem{Simon} 
Simon, L.
Schauder estimates by scaling. {\it Calc. Var. Partial Differential Equations} 5 (1997), 391-407.


\bibitem{Simons} Simons, J.
{Minimal varieties in Riemannian manifolds}. {\it Ann. of Math.} 88 (1968),  62-105.

\bibitem{SZ} Sternberg, P.; Zumbrun, K.; 
{Connectivity of phase boundaries in strictly convex domains}. {\it Arch. Rational Mech. Anal.} 141 (1998), 375-400.


\bibitem{Tro} Troianello, G. M.  \emph{Elliptic Differential Equations and Obstacle Problems}. The University Series in Mathematics, Springer, 1987.


\bibitem{Vil13}
Villegas, S. 
Boundedness of extremal solutions in dimension 4. {\it Adv. Math.} 235 (2013), 126-133.

\end{thebibliography}
\end{document}